\documentclass[12pt]{article}
\usepackage{amsmath}
\usepackage{amsthm}
\usepackage{amsfonts}
\usepackage[alphabetic]{amsrefs}
\usepackage{geometry}
\usepackage{graphicx}
\usepackage{enumitem}
\usepackage{amssymb}
\usepackage[all]{xy}
\usepackage{hyperref}
\usepackage{tikz}
\usepackage{subfig}

\usepackage{rotating}
\usepackage{tabularx}
\usepackage{caption}
\usepackage{xspace}

\newcommand{\I}{\textrm{I}}

\newcommand{\gen}[1]{\langle #1\rangle}
\newcommand{\cls}[1]{\overline{#1}}
\newcommand{\ph}{\varphi}
\newcommand{\ins}{\subseteq}

\title{Geometric Variations of Local Systems \\ and Elliptic Surfaces}
\usepackage{authblk}

\author[1]{Charles F. Doran\thanks {charles.doran@ualberta.ca}}
\affil[1]{Department of Mathematical and Statistical Sciences, University of Alberta and Center of Mathematical Sciences and Applications, Harvard University}

\author[2]{Jordan Kostiuk\thanks{jordan\_kostiuk@brown.edu}}
\affil[2]{Department of Mathematics, Brown University}
\date{}

\begin{document}
\newtheorem{problem}{Problem}
\newtheorem{theorem}{Theorem}
\newtheorem{theorem/def}{Theorem/Definition}
\newtheorem*{theorem*}{Theorem}
\newtheorem{lemma}{Lemma}
\newtheorem*{lemma*}{Lemma}
\newtheorem{proposition}{Proposition}
\newtheorem*{proposition*}{Proposition}

\newtheorem{corollary}{Corollary}
\theoremstyle{definition}
\newtheorem*{definition}{Definition}
\newtheorem*{definition*}{Definition}
\theoremstyle{remark}
\newtheorem{remark}{Remark}
\newtheorem*{example}{Example}

\maketitle

\begin{abstract}
    Geometric variations of local systems are families of variations of Hodge structure; they typically correspond to fibrations of K\"{a}hler manifolds for which each fibre itself is fibred by codimension one K\"{a}hler manifolds. 
    In this article, we introduce the formalism of geometric variations of local systems and then specialize the theory to study families of elliptic surfaces. 
    We interpret a construction of twisted elliptic surface families used by Besser-Livn\'{e} in terms of the middle convolution functor, and use explicit methods to calculate the variations of Hodge structure underlying the universal families of $M_N$-polarized K3 surfaces. 
    Finally, we explain the connection between geometric variations of local systems and geometric isomonodromic deformations, which were originally considered by the first author in 1999. 
\end{abstract}

\tableofcontents

\section{Introduction}

Monodromy preserving deformations of Fuchsian ordinary differential equations correspond to solutions to nonlinear isomonodromic deformation equations.  In \cite{doran_picard-fuchs_1999}, the first author studied the situation in which the Fuchsian equations arise as the Picard-Fuchs ODE for a one-parameter family of algebraic varieties and the monodromy preserving deformations are induced by complex structure deformations of this family, resulting in {\em geometric isomonodromic deformations}.  In particular, when applied to rational elliptic surfaces with four singular fibers, this construction produces algebraic solutions to the Painlev\'{e} VI equation.

The goal of the present paper is to revisit this construction, generalizing and recasting it in terms of the notion of  variation of local systems of Dettweiler and Wewers, yielding the {\em geometric variations of local systems} (GVLS) of the title.  

The organization of this article is as follows. 
In Section \ref{sec:GVLS}, we review some prerequisites from the theory of local systems and the definition of a variation of local systems given in \cite{dettweiler_variation_2006}. 
After recalling Hodge-theoretic results from \cite{zucker_hodge_1979}, we introduce the notion of a geometric variation of local system  and its parabolic cohomology. 
Geometric variations of local systems were first introduced in the second author's Ph.D. thesis \cite{kostiuk_geometric_2018}; the definitions in this article are more precise respect to various geometric notions and take into account the full lattice structure obtained by considering the associated variations of Hodge structure. 
Section \ref{sec:computing} goes into some detail about the Dettweiler-Wewers algorithm that is used in \cites{dettweiler_variation_2006,dettweiler_variation_2006-1} to compute the parabolic cohomology associated to a variation of local systems. 
This algorithm is used extensively in the rest of this paper to perform our calculations. 

In Section \ref{sec:parabolic_cohomology_of_elliptic}, we review some of the basic theory of elliptic surfaces and describe in detail the nature of the parabolic cohomology of an elliptic fibration. 
We tie together various results in the literature to identify the parabolic cohomology lattice embedded in the second cohomology in Proposition \ref{Prop:EllipticSurfacePCohomStructure} and then explain the relationship between this viewpoint and the de Rham perspective espoused in \cites{stiller_elliptic_1980,stiller_special_1984}.

In Section \ref{sec:middle_convolution}, we use the middle convolution functor---the original motivation for introducing variations of local systems---to construct a class of geometric variations of local systems.
Specifically, we calculate the integral variation of Hodge structure that underlies the family of elliptic surfaces obtained by starting with a fixed elliptic surface and then performing a quadratic twist at one singular fibre and one varying smooth fibre. 
After proving some general results about this construction, we calculate the parabolic cohomology local systems and Picard-Fuchs equations obtained by applying this construction to all 38 of the rigid elliptic surfaces with 4 singular fibres that were classified in \cite{herfurtner_elliptic_1991}. 
A subset of these Picard-Fuchs equations were calculated in \cite{besser_picard-fuchs_2012} and our results relate their work back to the middle convolution, proving some of their observations along the way, as well as allowing one to compute the global monodromy representations.
The Picard-Fuchs equations and transcendental lattices for these families are tabulated in the Appendix.

In Section \ref{sec:K3surfaces}, we use our methods to calculate the integral variation of Hodge structure on the modular curve $X_0(N)^+$ underlying the universal family $\mathcal{X}_N$ of $M_N$-polarized K3 surfaces for $N\in\{2,3,4,5,6,7,8,9,11\}$.
These values of $N$ correspond exactly to the values of $N$ for which families of Calabi-Yau threefolds can be fibred by $M_N$-polarized K3 surfaces.
This calculation is done by finding an internal elliptic fibration on the family $\mathcal{X}_N$ and calculating the parabolic cohomology of the associated geometric variation of local systems. 
In contrast to the simplicity of the computations in Section \ref{sec:middle_convolution}, the computations in Section \ref{sec:K3surfaces} demonstrate the fact that determining the so-called braiding map---required to run the Dettweiler-Wewers algorithm---is, in general, not a trivial task. 

Finally, in Section \ref{sec:conclusion}, we relate our work back to the theory of geometric isomonodromic deformations, which were introduced by the first author in \cite{doran_algebraic_2001} with the goal of constructing interesting solutions to isomonodromic deformation equations. 
In particular, the perspective taken in the present work allows for non-algebraic geometric isomonodromic deformations coming from families of non-algebraic K\"{a}hler manifolds.


\section{Geometric Variations of Local Systems}\label{sec:GVLS}

\subsection{Variations of Local Systems}

We assume the reader is familiar with the basic theory of local systems, but will review some definitions below. 
Let $X$ be a locally contractible topological space, and let $R$ be a commutative ring with unit. 

\begin{definition}
A \emph{local system of $R$-modules on $X$} is a locally constant sheaf $\mathcal{V}$ of free $R$-modules of finite rank. 
The stalk of $\mathcal{V}$ at $x\in X$ is denote by $\mathcal{V}_{x}$. 
If $x_0\in X$ is a base point and $V=\mathcal{V}_{x_0}$, then the fundamental group $\pi_1(X,x_0)$ acts on $V$. 
The corresponding representation $\rho\colon\pi_1(X,x_0)\to\textrm{GL}(V)$ is called the \emph{monodromy representation} associated to $\mathcal{V}$.
\end{definition}

\begin{remark}
		We use the convention that the fundamental group acts on $V$ \emph{on the right}. 
	After picking a basis for $V$, we represent elements of $V$ as row vectors and the monodromy representation 
	$$\rho\colon \pi_1(X,x_0)\to\textrm{GL}_p(R)$$
	is given by right-multiplication:
	$$v^\gamma=v\cdot \rho(\gamma).$$
	These conventions are the same as those that appear in \cite{dettweiler_variation_2006}, and will allow us to most easily implement their algorithms.
\end{remark}

We now focus on the case where $X\cong\mathbf{P}^1_\mathbf{C}$ is the Riemann sphere. 
Let $D=\{x_1,\dots, x_r\}\ins X$ be a subset of $r$ pairwise distinct points, and let $U=X-D$. 
Then, one can choose simple loops $\gamma_i\in\pi_1(U,x_0)$  that  go around $x_i$ counter-clockwise in such a way that
$$\gamma_1\cdots\gamma_r=1.$$
This gives us a presentation of $\pi_1(U,x_0)$ as a free group on $r-1$ generators. 
A local system of $R$-modules on $U$ corresponds to a representation $\rho\colon\pi_1(U,x_0)\to\textrm{GL}(V)$ which, in turn, corresponds to an $r$-tuple of transformations $g_i=\rho(\gamma_i)$ satisfying
$$g_1\cdots g_r=1.$$
Conversely, an $r$-tuple of transformations  $\mathbf{g}=(g_1,\dots, g_r)$ with trivial product determines a local system via the categorical equivalence between local systems and representations of the fundamental group acting on the fibre, as was proved by Deligne \cite{deligne_equations_1970}.

Let $j\colon U\to X$ denote the inclusion map. The parabolic cohomology of a local system is the following cohomological invariant:
\begin{definition}
	The (first) \emph{parabolic cohomology} of the local system $\mathcal{V}$ is the sheaf cohomology of $j_*\mathcal{V}$ and will be denoted by 
	$$H^1_p(U,\mathcal{V}):=H^1(X,j_*\mathcal{V}).$$
	According to \cite{dettweiler_variation_2006}, this cohomology group is a subgroup of $H^1(\pi_1(U,x_0),V)$, computed using group cohomology.
\end{definition}

In concrete terms, a cocycle for $\pi_1(U,x_0)$ with values in $V$ is a map $\delta\colon\pi_1(U,x_0)\to V$ satisfying $\delta(\alpha\beta)=\delta(\alpha)\cdot\rho(\beta)+\delta(\beta)$. 
The subgroup of \emph{parabolic cocycles}, i.e., the parabolic cohomology subgroup, corresponds to the cocycles satisfying $$\delta(\gamma_i)\in\textrm{im}(g_i-1),\ i=1,\dots, r,$$ as is explained in \cite{dettweiler_variation_2006}.

If the stabilizer $V^{\pi_1(U,_0)}$ is trivial, then the rank of the parabolic cohomology group is computed as follows when $R=K$ is a field: 
\begin{equation}\label{rank_formula}
\dim_K H^1_p(U,\mathcal{V})=(r-2)\cdot\dim_K V-\sum_{i=1}^r\dim_K\ker(g_i-1).
\end{equation}
Thus, the rank of parabolic cohomology is completely determined by the \emph{local monodromy matrices}.

Local systems are very closely related to \emph{flat connections}. 
\begin{definition}
Let $\mathcal{V}$ be a quasi-coherent sheaf of $\mathcal{O}_X$-modules. 
A \emph{connection} on $\mathcal{V}$ is a $\mathbf{C}$-linear homomorphism
$$\nabla\colon\mathcal{V}\to\Omega^1_X\otimes_{\mathcal{O}_X}\mathcal{V}:=\Omega^1_X(\mathcal{V})$$
that satisfies the Leibniz identity:
$$\nabla(gs)=dg\otimes s+g\nabla s.$$
A connection $\nabla=\nabla_0\colon\mathcal{V}\to\Omega_X^1\otimes\mathcal{V}$ extends to a $\mathbf{C}$-linear map 
$$\nabla_i\colon\Omega^i\otimes\mathcal{V}\to\Omega^{i+1}\otimes\mathcal{V}$$ via
$$\nabla_i(\omega\otimes s):=d\omega\otimes s+(-1)^i\omega\wedge\nabla_0(s).$$
The composition $R=\nabla_1\circ \nabla_0$ is called the \emph{curvature} of the connection $\nabla$; a connection $\nabla$ is \emph{flat}, or \emph{integrable} if $R=0$. 
\end{definition}

Suppose now that $E$ is a complex local system on $X$, and let $\mathcal{E}=\mathcal{O}_X\otimes E$. 
Then we can give $\mathcal{E}$ a natural connection $\nabla$ for which $E=\ker\nabla$ by setting
$$\nabla(gs)=dg s,$$
where $g\in\mathcal{O}_X$ and $s\in E$. 
This connection is flat and known as the \emph{Gauss-Manin connection} associated to the local system $E$.
Conversely, we have the following theorem of Deligne:
\begin{theorem*}[Deligne \cite{deligne_equations_1970}]
	Let $\nabla$ be a connection on a locally free sheaf $\mathcal{E}$ over a connected domain $X$, and set $E=\ker\nabla$. 
	If $\nabla$ is flat, then $E$ is a local system on $X$ and $\mathcal{E}=\mathcal{O}_X\otimes E$. 
\end{theorem*}

In this paper, we will be mostly interested in the Gauss-Manin connection on cohomology bundles associated to algebraic varieties. 
Specifically, suppose that $f\colon X\to T$ is a fibration of smooth projective varieties, and let $X_t:=f^{-1}(t)$ denote the fibre. Then, $\mathcal{V}:=R^kf_*\mathbf{Z}$, the higher direct-image sheaf, is a local system on $T$ whose fibres are the cohomology groups $H^k(X_t,\mathbf{Z})$. 
The Gauss-Manin connection $\nabla_{GM}$ on the cohomology bundle $\mathcal{H}(X/T):=R^kf_*\mathbf{Z}\otimes\mathcal{O}_T$ is the unique holomorphic connection whose flat sections are the elements of $R^kf_*\mathbf{C}$.

\subsection{Variations of Local Systems and Parabolic Cohomology}

We now recall some of the theory of \emph{variations of local systems}, as introduced in \cite{dettweiler_variation_2006}, which should be thought of as a varying family of local systems. 
\begin{definition}
	Let $A$ be a connected complex manifold and $r\geq 3$. 
	An \emph{$r$-configuration} over $A$ consists of a smooth and proper morphism $\cls{\pi}\colon X\to A$ of complex manifolds together with a smooth relative divisor $D\ins X$ for which the fibres $X_a$ are isomorphic to $\mathbf{P}^1_\mathbf{C}$  and $D\cap X_a$ consists of $r$ pairwise distinct points.
	We will often use $(X,D)$ to denote the $r$-configuration, suppressing the map $\cls{\pi}\colon X\to A$ from the notation.
\end{definition}

\begin{definition}
Fix an $r$-configuration $(X,D)$ over $A$ and basepoint $a_0\in A$.
Let $U=X-D$, $D_0=D_{a_0}$, and  $X_0=\cls{\pi}^{-1}(a_0)$, with $U_0=X_0-D_0\cong\mathbf{P}^1-\{x_1,\dots, x_r\}$. 
Let $\mathcal{V}_0$ be a local system on $U_0$. 
    A \emph{variation of the local system $\mathcal{V}_0$ over the $r$-configuration $(X,D)$} is a local system $\mathcal{V}$ of $R$-modules on $U$ whose restriction to $U_0$ is identified with $\mathcal{V}_0$. 
    We  often refer to $\mathcal{V}$ is a variation of local systems, suppressing the $r$-configuration $(X,D)$. 
\end{definition}
Notation as in the previous definition, let $j\colon U\to X$ be the inclusion and $\pi\colon U\to A$ the projection. 
Let $x_0\in U_0$ be a base point. 
The fibration $\pi\colon U\to A$ gives rise to a short exact sequence of fundamental groups \cite{dettweiler_variation_2006}:
\begin{equation}\label{VLS_ses}
\xymatrix{1\ar[r]&\pi_1(U_0,x_0)\ar[r]&\pi_1(U,x_0)\ar[r]&\pi_1(A,a_0)\ar[r]&1.}
\end{equation}

\begin{definition}
    	The \emph{parabolic cohomology} of the variation $\mathcal{V}$ is the higher direct image sheaf
	$$\mathcal{W}=R^1\cls{\pi}_*(j_*\mathcal{V}).$$
\end{definition}

	By definition, the parabolic cohomology of the variation $\mathcal{V}$ is a sheaf of $R$-modules on $A$. 
	Locally on $A$, the configuration $(X,D)$ is topologically trivial,  and it follows that $\mathcal{W}$ is locally constant with fibre 
	$$W=H^1_p(U_0,\mathcal{V}_0).$$
	Therefore, $\mathcal{W}$ is itself a local system of $R$-modules. 
	Let $\eta\colon\pi_1(A,a_0)\to\textrm{GL}(W)$ denote its corresponding monodromy representation.
The following lemma describes the monodromy representation for this local system:
\begin{lemma*}[Lemma 2.2 \cite{dettweiler_variation_2006}]
	Let $\beta\in\pi_1(A)$ and $\delta\colon\pi_1(U_0)\to V$ be a parabolic cocycle, with $[\delta]$ the corresponding equivalence class. 
	Let $\tilde{\beta}\in\pi_1(U)$ be a lift of $\beta$. 
	Then $[\delta]^{\eta(\beta)}=[\delta']$, where
	$\delta'\colon\pi_1(U_0)\to V$ is the cocycle
	$$\alpha\mapsto\delta(\tilde{\beta}\alpha\tilde{\beta}^{-1})\cdot\rho(\tilde{\beta}),\ \alpha\in\pi_1(U_0).$$
\end{lemma*}

In other words, if we have a family of local systems $\mathcal{V}_a$ parameterized by $A$, then the parabolic cohomology groups of the local systems $\mathcal{V}_a$ fit together in the form of a local system on the space $A$. 
The monodromy of this local system is calculated in terms of the short exact sequence \eqref{VLS_ses}, and the monodromy representation of the initial local system $\mathcal{V}_{a_0}$.

The local systems that we work in this paper come equipped with bilinear pairings
$$\mathcal{V}\times\mathcal{V}\to R.$$
As explained in \cite{dettweiler_variation_2006-1}, this pairing gives rise to a pairing on the parabolic cohomology local system
$$\mathcal{W}\times\mathcal{W}\to R.$$
We remark that if the pairing on $\mathcal{V}$ is $(-1)^i$-symmetric, then the induced pairing on parabolic cohomology is $(-1)^{i+1}$-symmetric \cite{dettweiler_variation_2006-1}.

\subsection{Hodge Structures on Parabolic Cohomology}

We are most interested in the local systems that arise from geometry. 
Specifically, we will be studying \emph{variations of Hodge structures}.
We review some preliminaries on the subject, and refer the reader to the  survey paper \cite{laza_introduction_2015}, and the text \cite{carlson_period_2017}, for more details.

\begin{definition}
	A (pure) Hodge structure of weight $n\in\mathbf{Z}$, denoted by $(H_\mathbf{Z},H^{p,q})$ is a finitely generated abelian group $H_\mathbf{Z}$ together with a decomposition of the complexification:
	$$\mathbf{H}_\mathbf{C}=\bigoplus_{p+q=n}H^{p,q}$$
	satisfying $H^{p,q}=\cls{H^{q,p}}$. 
	
		A \emph{polarized hodge structure of weight $n$} consists of an (integral) hodge structure of weight $n$, together with a non-degenerate bilinear form $Q$ on $H_\mathbf{Z}$ which, when extended to $\mathbf{H}_\mathbf{C}$, enjoys the following properties:
	\begin{itemize}
	\item $Q$ is $(-1)^n$-symmetric;
	\item $Q(\xi,\eta)=0$ for $\xi\in H^{p,q}$ and $\eta\in H^{p',q'}$ with $p\neq q'$;
	\item $(-1)^{\frac{n(n-1)}{2}}i^{p-q}Q(\xi,\cls{\xi})>0$ for $\xi\neq 0\in H^{p,q}$.
\end{itemize}
\end{definition}

Equivalent to the above Hodge decomposition is the Hodge filtration. 
This is a finite decreasing filtration $\{F^p\}$ of $H_\mathbf{C}$ 
$$H_\mathbf{C}\supset\cdots\supset F^p\supset F^{p+1}\supset \cdots,$$ such that
$$H_\mathbf{C}\cong F^p\oplus\cls{F^{n-p+1}}.$$
Given the Hodge decomposition, we obtain the filtration by setting
$$F^p:=\bigoplus_{i\geq p} H^{i,n-i};$$ given the filtration, we recover the decomposition by setting
$$H^{p,q}:=F^p\cap\cls{F^q}.$$
The filtration perspective is a useful reformulation as it varies holomorphically in families \cite{laza_introduction_2015}. 

As is well-known, the $n$-th cohomology group of a K\"{a}hler manifold admits a polarized Hodge structure of weight $n$, with the polarization being induced by the K\"{a}hler form. When $n$ is equal to the (complex) dimension of the manifold, the polarization agrees with that of the cup-product pairing. 
Now suppose we have a family $f\colon X\to T$ of K\"{a}hler manifolds. 
As $t\in T$ varies, so do the cohomology groups $H^n(X_t,\mathbf{Z})$ and their Hodge structures, giving rise to a variation of Hodge structure.

More precisely: let $T$ be a complex manifold and $\mathcal{E}_\mathbf{Z}$ a local system of finitely generated free $\mathbf{Z}$-modules on $T$ and set $\mathcal{E}:=\mathcal{E}_\mathbf{Z}\otimes\mathcal{O}_T$. 
Then $\mathcal{E}$ is a complex vector bundle and is equipped with the Gauss-Manin connection $\nabla\colon\mathcal{E}\to\mathcal{E}\otimes\Omega^1_T$ induced by $d\colon\mathcal{O}_B\to\Omega^1_T$. 
Let $\{\mathcal{F}^p\}$ be a filtration by sub-bundles.

\begin{definition}
	The data $(\mathcal{E}_\mathbf{Z},\mathcal{F})$ defines a \emph{variation of Hodge structure of weight $n$} on $B$ if
	\begin{itemize}
		\item $\{\mathcal{F}^p\}$ induces Hodge structures on weight $n$ on the fibres of $\mathcal{E}$;
		\item if $s$ is a section of $\mathcal{F}^p$ and $\zeta$ is a vector field of type $(1,0)$, then $\nabla_{\zeta}s$ is a section of $\mathcal{F}^{p-1}$ (Griffiths transversality). 
	\end{itemize}
	Furthermore, if $\mathcal{E}_\mathbf{Z}$ carries a non-degenerate bilinear form $Q\colon\mathcal{E}_\mathbf{Z}\times\mathcal{E}_\mathbf{Z}\to\mathbf{Z}$, we have a \emph{polarized variation of Hodge structures} if
	\begin{itemize}
		\item $Q$ defines a polarized Hodge structure on each fibre;
		\item $Q$ is flat with respect to $\nabla$; that is, we have
			$$d Q(s,s')=Q(\nabla s,s')+Q(s,\nabla s').$$
	\end{itemize}
\end{definition}

Given a family $f\colon X\to T$ of K\"{a}hler manifolds, the pair $(R^nf_*\mathbf{Z},\mathcal{F})$, where $\mathcal{F}$ is the holomorphic bundle with fibre equal to $F^pH^n(X_t,\mathbf{C})$, forms a variation of Hodge structure of weight $n$. 
The fibre-wise polarization gives rise to a polarization on $R^nf_*\mathbf{Z}$, and so we in fact have a polarized variaion of Hodge structures.

A variation of Hodge structure is an integral local system; 
 we may therefore consider the associated parabolic cohomology. 
Results of Zucker show that this group can be given a Hodge structure in the case where the base of the family is a curve. 
More precisely, he proves the following theorem
\begin{theorem*}[Theorem 7.12 \cite{zucker_hodge_1979}]
	Let $T$ be a non-singular algebraic curve over $\mathbf{C}$, $\cls{T}$ its smooth completion, $j\colon T\to\cls{T}$ the inclusion, and $\mathcal{V}$ a local system of complex vector spaces underlying a polarizable variation of Hodge structure of weight $m$. 
	There is a natural polarizable Hodge structure of weight $m+i$ on $H^i(\cls{S},j_*\mathcal{V})$ associated to the variation of Hodge structure. 
\end{theorem*}

As Zucker explains, when $\mathcal{V}=R^mf_*\mathbf{C}$, the sheaf $j_*\mathcal{V}$ is the sheaf of local invariant ``cycles'' and the Hodge structure is most interesting when $i=1$. 
A Hodge structure can always be placed extrinsically on $H^1(\cls{T},j_*R^mf_*\mathbf{C})$ using the Leray spectral sequence for $\cls{f}$; one of the main results of \cite{zucker_hodge_1979} is that these two Hodge structures coincide:
\begin{theorem*}[Theorem 15.5 \cite{zucker_hodge_1979}]
	The Hodge structure on $H^1(T,R^if_*\mathbf{C})$ is induced by that of $H^{i+1}(X)$.
\end{theorem*} 

That is, there is an inclusion of $H^1(T,R^if_*\mathbf{Q})$ inside $H^{i+1}(X,\mathbf{Q})$ for which the Hodge structure on $H^1(T,R^if_*\mathbf{C})$ agrees with the one it inherits from the Hodge structure on $H^{i+1}(X,\mathbf{C})$.

\subsection{Geometric Variations of Local Systems}
Here we explain how to combine the notions of variations of local systems and variations of Hodge structures to produce \emph{geometric variations of local systems}, the subject of the second author's Ph.D. thesis \cite{kostiuk_geometric_2018}. 

Let $(X,D)$ be an $r$-configuration.
That is, we consider a proper morphism of complex manifolds $\cls{\pi}\colon X\to A$, together with a smooth relative divisor $D\ins X$ for which each fibre $X_a$ is isomorphic to $\mathbf{P}^1_\mathbf{C}$, and $D_a=D\cap X_a$ consists of $r$ distinct points; let $U=X-D$ and set $U_a=X_a-D_a\cong\mathbf{P}^1-\{\textrm{$r$ points}\}$ for each $a\in A$.

\begin{definition}
A \emph{geometric variation of local systems over the $r$-configuration $(X,D)$} is a variation of local systems $\mathcal{V}$ over $(X,D)$ satisfying the following conditions:
	\begin{enumerate}
		\item the local system $\mathcal{V}$ is a polarized variation of Hodge structures of weight $n$ over $A$;
		\item  for each $a\in A$, the restriction $\mathcal{V}_a$ is itself a polarized variation of Hodge structures of weight $n$ over $U_a$;
		\item the parabolic cohomology $\mathcal{W}$ is a polarized variation of  Hodge structures of weight $n+1$. 
	\end{enumerate}
	
\end{definition}

\begin{definition}\label{Def:GVLS}
Let $(X,D)$ be an $r$-configuration and $U=X-D$. 
A family of K\"{a}hler manifolds \emph{over the $r$-configuration $(X,D)$} is a family $f\colon\mathcal{X}\to U$ of K\"{a}hler manifolds. 
Let $\mathcal{X}\to A$ be a family of K\"{a}hler manifolds. An \emph{internal fibration structure} is an $r$-configuration $(X,D)$, together with a morphism $f\colon\mathcal{X}\to U$ compatible with $\mathcal{X}\to A$. 
We refer to the fibration $\mathcal{X}\to A$ (with fibres $\mathcal{X}_a$) as the \emph{external fibration}; the \emph{internal fibrations} $f_a\colon\mathcal{X}_a\to U_a$ have fibres $\mathcal{X}_{a,x}$ of codimension one in $\mathcal{X}_a$. 

$$\xymatrix{
\mathcal{X}_{x,a}\subseteq\mathcal{X}_{a}\ar[r]\ar^{f_a}[d]&\mathcal{X}\ar@{=}[r]\ar^f[d]&\mathcal{X}\supseteq\mathcal{X}_a\ar[d]\\
x\in U_a\ar[r]&U\ar^{\pi}[r]& A\ \rotatebox[origin=c]{180}{\(\in\)} a}$$
\end{definition}

\begin{proposition}
Let $(X,D)$ be an $r$-configuration, and let $f\colon \mathcal{X}\to U\to A$ be a family of K\"{a}hler manifolds with internal fibration structure. 
The polarized variation of Hodge structure $\mathcal{V}=R^nf_*\mathbf{Z}$ on $U$ defines a geometric variation of local systems over $A$.
The polarized variation of Hodge structure corresponding to the parabolic cohomology $\mathcal{W}$ is a sub-variation of $R^{n+1}(\pi\circ f)_*\mathbf{Z}$.
\end{proposition}

\begin{proof}
	For each $a\in A$, the local system $\mathcal{V}_a$ is equal to the local system $R^n_{f_{a,*}}\mathbf{Z}$ on $U_a\ins X_a\cong\mathbf{P}^1$, where $f_a\colon\mathcal{X}_a\to U_a$ is the restriction of $f$ to the fibre over $a\in A$. 
The local system $\mathcal{W}:=R^1\pi_*j_*\mathcal{V}$ on $A$ is the local system whose stalk at each $a\in A$ is the parabolic cohomology of $\mathcal{V}_a$. 
That is, $\mathcal{W}_a\cong R^1{j_a}_*\mathcal{V}_a$.
Each $\mathcal{W}_a$ carries a Hodge structure of weight $n+1$ by \cite{zucker_hodge_1979}. 
On the other hand, the local system $\mathcal{W}$ is contained in $R^{n+1}(\pi\circ f)_*\mathbf{Z}$, the local system whose stalks are the cohomology groups $H^{n+1}(\mathcal{X}_a,\mathbf{Z})$, and the Hodge structure on $\mathcal{W}_a$ is the same as the one induced by this inclusion. 
Since $R^{n+1}(\pi\circ f)_*\mathbf{Z}$ is a variation of Hodge structure, its restriction to $\mathcal{W}$ is also a variation of Hodge structure.  
Therefore, $\mathcal{V}$ defines a geometric variation of local systems.
\end{proof}

\section{Computing Parabolic Cohomology}\label{sec:computing}

In this section, we review the algorithms given in \cites{dettweiler_variation_2006,dettweiler_variation_2006-1} that compute the parabolic cohomology local system of a variation of local systems, and the induced bilinear pairing.
This allows to compute the piece of the external variation of polarized Hodge structure corresponding to the parabolic cohomology of an internal fibration of K\"{a}hler manifolds.

We begin with some preliminaries. 
Suppose that $\mathcal{V}$ is a local system on $U=\mathbf{P}^1-\{t_1,\cdots, t_r\}$, let $V$ denote the stalk at $t_0$, and suppose that the corresponding monodromy representation $\rho$ is given by the $r$-uple of matrices
$$\mathbf{g}=(g_1,\dots, g_r).$$
A cocycle $\delta\in H^1(\pi_1(U,x_0),V)$ is a map $\delta\colon\pi_1(U,t_0)\to V$ satisfying the \emph{cocycle condition}:
$$\delta(\alpha\beta)=\delta(\alpha)\cdot\rho(\beta)+\delta(\beta).$$
Set $v_i=\delta(\gamma_i)$; then, since $\delta(1)=0$, the following relation holds:
\begin{equation}\label{cocycle_relation}
	v_1\cdot g_2\cdots g_r+v_2\cdot g_3\cdots g_r+\cdots+v_r=0.
\end{equation}
Conversely, given an $r$-tuple of vectors $(v_1,\dots, v_r)\in V^r$ that satisfy relation \eqref{cocycle_relation}, we obtain a unique cocycle by setting $\delta(\gamma_i)=v_i$ and extending to the rest of the fundamental group using the cocycle condition. 
The cocycle is a coboundary if and only if there is $v\in V$ for which $v_i=v\cdot (g_i-1)$ holds for all $i$.

Let $H_\mathbf{g}$ and $E_\mathbf{g}$ denote the following subspaces of $V^r$:
$$\begin{array}{ll}
     H_\mathbf{g}=&\{(v_1,\dots, v_r)|\ v_i\in\textrm{image}(g_i-1),\ \textrm{and condition \eqref{cocycle_relation} holds}\}  \\
     E_\mathbf{g}=&\{(v\cdot(g_1-1),\cdots,v\cdot(g_r-1))|\ v\in V\} 
\end{array}$$
The association $\delta\mapsto (v_1,\dots, v_r)$ is an isomorphism 
$$H^1_p(U,\mathcal{V})\cong W_\mathbf{g}:=H_\mathbf{g}/E_{\mathbf{g}},$$
as shown in \cite{dettweiler_variation_2006}.
This description in terms of a quotient of free modules allows us to work with parabolic cohomology at a computational level. 

\begin{remark}
	While it is not emphasized in \cite{dettweiler_variation_2006}, it should be noted that the sheaf $\mathcal{W}$ may not, in general, be free---there may be torsion. 
	We obtain a local system of $R$-modules, in the sense of \cite{dettweiler_variation_2006}, by dividing out by the torsion subgroup. 
	The braid companion quotient $W_{\mathbf{g}}$ that one works with when implementing the algorithms in \cite{dettweiler_variation_2006} is equal to the intersection of parabolic cohomology tensored with the field of fractions and the $R$-valued cohomology group, i.e., is identified with parabolic cohomology modulo torsion. 
	This is not an issue for us because we divide out by the torsion anyway when working with Hodge structures, but we point it out because some of the parabolic cohomology groups we will work with are \emph{not} torsion-free, as we will see in the next section.
\end{remark}

Now we describe how the parabolic cohomology varies in the context of a variation of local systems. 
 Let 
$$\mathcal{O}_{r-1}=\{D'\ins\mathbf{C}|\ |D'|=r-1\}=\{D\ins\mathbf{P}^1_\mathbf{C}|\ |D|=r,\ \infty\in D\}$$
be the configuration space of $r-1$ points in the plane, or of $r$ points on the Riemann sphere with one of the points at $\infty$. 
The fundamental group $\pi_1(\mathcal{O}_{r-1},D_0)$ is known as the \emph{Artin} braid group on $r-1$ strands. 
The braid group admits standard generators $\beta_1,\dots, \beta_{r-2}$ that rotate counterclockwise, exchanging the position of $x_i,x_{i+1}$ \cite{birman_braids_1975}. 
These generators satisfy the relations
$$\beta_i\beta_{i+1}\beta_i=\beta_{i+1}\beta_i\beta_{i+1},\ \beta_i\beta_j=\beta_j\beta_i\ |i-j|>1.$$

Let
$$\mathcal{E}_r(V)=\{\mathbf{g}=(g_1,\dots, g_r)|\ g_i\in\textrm{GL}(V),\ g_1\dots g_r=1\}.$$
Then, $A_{r-1}$ acts on $\mathcal{E}_r(V)$ from the right via
$$\mathbf{g}^{\beta_i}=(g_1,\dots, g_{i+1},g_{i+1}^{-1}g_ig_{i+1},\dots, g_r).$$
Define $R$-linear isomorphisms
$$\Phi(\mathbf{g},\beta)\colon H_\mathbf{g}\to H_{\mathbf{g}^\beta},$$
by declaring
\begin{equation}\label{phi_map_formula}
	(v_1,\dots, v_r)^{\Phi(\mathbf{g},\beta_i)}=(v_1,\dots, v_{i+1},v_{i+1}(1-g_{i+1}^{-1}g_ig_{i+1})+v_ig_{i+1},\cdots, v_r),
\end{equation}
and extending to all of the braid group using the ``cocycle'' rule:
$$\Phi(\mathbf{g},\beta)\Phi(\mathbf{g}^\beta,\beta')=\Phi(\mathbf{g},\beta\beta').$$

These maps act appropriately on the submodules $E_\mathbf{g}$ and therefore induce  isomorphisms of parabolic cohomology groups:
$$\cls{\Phi}(\mathbf{g},\beta)\colon W_\mathbf{g}\to W_{\mathbf{g}^\beta}.$$
The group $\textrm{GL}(V)$ acts on $\mathcal{E}_r(V)$ on the right via global conjugation: $$\mathbf{g}^h=(h^{-1}g_1h,\dots, h^{-1}g_rh).$$
Define isomorphisms
$$\Psi(\mathbf{g},h):\left\{\begin{array}{ccc}
	H_{\mathbf{g}^h}&\to &H_\mathbf{g}\\
	(v_1,\dots,v_r)&\mapsto &(v_1\cdot h,\cdots, v_r\cdot h)
\end{array}\right..$$
Then, the
maps $\Psi(\mathbf{g},h)$ induce isomorphisms on parabolic cohomlogy:
$$\cls{\Psi}(\mathbf{g},h)\colon W_{\mathbf{g}^h}\to W_{\mathbf{g}}.$$

For the rest of this section, we make the following assumptions about the $r$-configuration $\cls{\pi}\colon X\to A$:
\begin{itemize}
	\item $X=\mathbf{P}^1_A$ is the relative projective line over $A$;
	\item the divisor $D$ contains $\{\infty\}\times A\ins\mathbf{P}^1_A$;
	\item there exists a point $a_0\in A$ such that $D_0$ is contained in the real-line.
\end{itemize}

\begin{remark}
	These assumptions are only to make computations more feasible in \cite{dettweiler_variation_2006}. 
	As they will hold in all applications in this paper, we choose to make these assumptions ourselves for clarity of exposition.
\end{remark}

Since $\{\infty\}\times A\ins D$, we can use $D_0$ as a base point for the configuration space $\mathcal{O}_{r-1}$. 
The divisor $D\ins\mathbf{P}^1_A$ gives rise to a holomorphic map $A\to\mathcal{O}_{r-1}$ by sending each $a\in A$ to $D\cap X_a$.
Let $A_{r-1}=\pi_1(\mathcal{O}_r,D_0)$ be the Artin braid group.

\begin{definition}
	Notation as above, let $\ph\colon\pi_1(A,a_0)\to A_{r-1}$ be the corresponding push-forward homomorphism on fundamental groups. 
	The map $\ph$ is called the \emph{braiding map} induced by the $r$-configuration $(X,D)$. 
\end{definition}

Recall the short exact sequence \eqref{VLS_ses}:
\begin{equation*}
\xymatrix{1\ar[r]&\pi_1(U_0,x_0)\ar[r]&\pi_1(U,x_0)\ar[r]&\pi_1(A,a_0)\ar[r]&1.}
\end{equation*}
The variation $\mathcal{V}$ corresponds to a monodromy representation $\rho\colon\pi_1(U)\to\textrm{GL}(V)$. 
Let $\rho_0\colon\pi_1(U_0)\to\textrm{GL}(V)$ denote its restriction, via the exact sequence \eqref{VLS_ses}. 
As explained in \cite{dettweiler_variation_2006}, the short exact sequence is \emph{split}, so that $\rho$ is determined by $\rho_0$ and a representation $\chi\colon\pi_1(A)\to\textrm{GL}(V)$. 
A loop $\gamma\in\pi_1(A)$ acts on the initial representation $\rho_0$ in two different ways.
First, the loop $\gamma$ lifts to a loop in $\pi_1(U)$ and acts by conjugation; this has the effect of conjugating the representation $\rho_0$ by $\chi(\gamma)^{-1}$.
On the other hand, $\ph(\gamma)\in A_{r-1}$ acts via the braid action defined above. 
These actions are compatible:
$$\mathbf{g}^{\ph(\gamma)}=\mathbf{g}^{\chi(\gamma)^{-1}}.$$
The following theorem calculates the monodromy representation \cite{dettweiler_variation_2006}*{Theorem 2.5}:
\begin{theorem*}
	Let $\mathcal{W}$ be the parabolic cohomology of $\mathcal{V}$ and $\eta\colon\pi_1(A,a_0)\to\textrm{GL}(W)$ the monodromy representation. 
	For all $\gamma\in\pi_1(A,a_0)$, we have
	$$\eta(\gamma)=\cls{\Phi}(\mathbf{g},\ph(\gamma))\cdot\cls{\Psi}(\mathbf{g},\chi(\gamma)).$$ 
\end{theorem*}

\begin{remark}
	As is pointed out in \cite{dettweiler_variation_2006}, if $R$ is a field and the local system is irreducible, then the homomorphism $\chi$ is determined up to scalar multiples by the braiding map because of Schur's lemma. 
	It follows that the braiding map is enough to determine the \emph{projective monodromy} representation of parabolic cohomology. 
	In most geometric examples, the monodromy matrices we compute will lie in the subgroup of invertible integer-entry matrices. 
	Thus, the ambiguity of the projective monodromy representation is at worst given by multiplication by $-1$. 
	We can use other knowledge of the parabolic cohomology local system, such as the Picard-Fuchs equation, to pin down the representation precisely. 
	Lastly, in this setting, projectively equivalent local systems will be related by a quadratic twist. 
\end{remark}

This theorem, together with the preceding discussion, describes an algorithm to compute the monodromy representation of the parabolic cohomology of a variation of local systems, which we now summarize. 
Start with a variation of local systems $\mathcal{V}$ defined on an $r$-configuration $(X,D)$, and fix a base point $a_0\in A$ and the initial monodromy representation $\rho_0$, which corresponds to an $r$-tuple of matrix $\mathbf{g}$. 
Further, suppose that $\gamma_1,\dots, \gamma_s$ are generators for $\pi_1(A,a_0)$. 
The following steps compute the projective monodromy representation for the parabolic cohomology local system on $A$:
\begin{enumerate}
	\item construct the spaces $H_\mathbf{g},E_\mathbf{g},W_\mathbf{g}$;
	\item for each $i=1,\dots, s$, find matrices $h_i\in\textrm{GL}_r(\mathbf{C})$ for which
	$$\mathbf{g}^{\ph(\gamma_i)}=\mathbf{g}^{h_i^{-1}};$$
	\item compute the transformations $\Phi(\mathbf{g},\ph(\gamma_i))$ and $\Psi(\mathbf{g},h_i)$;
	\item the projective monodromy is given by $\eta(\gamma_i)=\cls{\Phi}(\mathbf{g},\ph(\gamma_i))\cdot\cls{\Psi}(\mathbf{g},h_i)$
\end{enumerate}

Thus, in order to compute the monodromy representation, we must know the braiding map $\ph$ and the representation $\chi$. 
As explained above, if the local systems are irreducible, then $\chi$ is determined up scalar multiples by the braiding map. 
Therefore, the determination of the braiding map is the most non-trivial component of this algorithm. 

If the local systems are equipped with a non-degenerate bilinear form, there is a natural non-degenerate cup-product  pairing on the parabolic cohomology $\mathcal{W}$.  
Dettweiler-Wewers give an explicit formula for this cup-product pairing \cite{dettweiler_variation_2006-1}*{Theorem 2.5}, which we adapt to our specific needs:
\begin{theorem*}
	Suppose that $\mathcal{V}_0$ is equipped with the non-degenerate bilinear pairing $$\gen{-,-}_{\mathcal{V}_0}\colon\mathcal{V}_0\times\mathcal{V}_0\to R.$$
	Let $\delta_1,\delta_2\colon\pi_1(U)\to V$ be two parabolic cocycles determined by $\delta_j(\alpha_i)=v_i^j$. Set $w_1^j=v_1^j$ and define recursively:
		$$w_i^j=v_i^j+w_{i-1}^jg_{i},\ i=2,\dots, r.$$
	Choose $u_i\in V$ satisfying 
	$$w_i^2-w_{i-1}^2=u_i(g_i-1),\ i=1,\dots, r.$$
	Then, the bilinear pairing $\gen{-,-}_\mathcal{W}\colon\mathcal{W}\times\mathcal{W}\to R$ on $\mathcal{W}$ is given by the following formula:
	$$\gen{\delta_1,\delta_2}_\mathcal{W}=\sum_{i=1}^r\gen{w_i^1-w_{i-1}^1,u_i-w^2_{i-1}}_{\mathcal{V}_0}.$$
\end{theorem*}

\section{Parabolic Cohomology of Elliptic Surfaces}\label{sec:parabolic_cohomology_of_elliptic}

\subsection{Elliptic Surface Preliminaries}
In this section we review some of the basic facts about elliptic surfaces. 
\begin{definition}
	An \emph{elliptic surface} $\mathcal{E}$ over $S$ is a smooth projective surface $\mathcal{E}$ with an elliptic fibration over $S$, i.e., a surjective map
	$$f\colon \mathcal{E}\to S$$
	for which 
	\begin{itemize}
		\item all but finitely many fibres are smooth curves of genus $1$;
		\item no fibre contains an exceptional curve of the first kind.
	\end{itemize}
	
	A \emph{section} of an elliptic surface $f\colon \mathcal{E}\to S$ is a morphism
	$$\sigma\colon S\to \mathcal{E},\ \textrm{for which}\  f\circ\sigma=\textrm{id}_S.$$
\end{definition}

\begin{remark}
	All of the elliptic surfaces in this paper will be assumed to have a section. 
\end{remark}
Since we are only dealing with complex elliptic surfaces with section, we can always choose a Weierstrass presentation of $\mathcal{E}$:
\begin{equation}\label{Weierstrass_Model}
	y^2=4x^3-g_2(t)x-g_3(t),\ g_2,g_3\in K(S).
\end{equation}
The fibres of \eqref{Weierstrass_Model} are smooth elliptic curves as long as the discriminant $\Delta=g_2^3-27g_3^2$ does not vanish. 
For each $t\in S$ such that $\Delta(t)=0$, the fibre is either a cuspidal or nodal rational curve, and the singular point of the curve may or may not be a surface singularities of \eqref{Weierstrass_Model}. 
If the singular point of the fibre is a surface singularity, then we perform a sequence of blow-ups to resolve the singularity. 
Doing this for each singular fibre, we arrive at the N\'{e}ron model, which is a smooth surface with elliptic fibration whose singular fibres are chains of rational curves. 

In order to classify the the kinds of chains of rational curves that occur, i.e., classify the kinds of singular fibres that can occur, Kodaira considered the following two invariants associated to an elliptic surface \cite{kodaira_compact_1960}:
\begin{definition}
Let $\mathcal{E}\to S$ be an elliptic surface, $\Sigma$  the support of the singular fibres, and $S_0=S-\Sigma$.  
	The \emph{functional invariant} $\mathcal{J}$ is the rational function on $S$ whose value at $t\in S_0$ is the $J$-invariant of the fibre $E_t$ at $t$. 
	In terms of the Weierstrass form \eqref{Weierstrass_Model}, we have
	$$\mathcal{J}=\frac{g_2^3}{g_2^3-27g_3^2}.$$ 
	
	The sheaf $\mathcal{V}:=R^1f_*\mathbf{Z}|_{S_0}$, is  called the \emph{homological invariant}; it is characterized by its monodromy representation
	$$\rho\colon\pi_1(X_0)\to\textrm{SL}_2(\mathbf{Z}),$$
	and we will often refer to the representation as the homological invariant. 
\end{definition}

The classification of Kodaira is described in terms of the possible local monodromy transformation around the singular fibre. 
In order to compute the type of singular fibre, we only need to know the order of vanishing of $g_2,g_3$, and $\Delta$ at the singular fibre. 
The classification is tabulated in Table \ref{Table:KodairaClass}.
Of the possible singular fibre types on an elliptic surface, those of type $\textrm{I}_N$ are called \emph{multiplicative} fibres, while all other singular fibre types are called \emph{additive} (the terminology comes from the kind of singular curve obtained at this particular point). 

\begin{table}
\centering
\begin{tabular}{||c|c|c|c|c|c||}
\hline
Type&$\nu(g_2)$&$\nu(g_3)$&$\nu(\Delta)$&Graph&Monodromy\\[0.5ex]
\hline\hline
$\textrm{I}_0$&$a\geq 0$&$b\geq 0$&$0$&$-$&$\begin{pmatrix}
	1&0\\
	0&1
\end{pmatrix}$\\
\hline
$\textrm{I}_N$&$0$&$0$&$N\geq 1$&$\tilde{A}_N$&$\begin{pmatrix}
	1&N\\
	0&1
\end{pmatrix}$\\
\hline
$\textrm{II}$&$a\geq 1$&$1$&$2$&$-$&$\begin{pmatrix}
	1&1\\
	-1&0
\end{pmatrix}$\\
\hline
$\textrm{III}$&$1$&$b\geq 1$&$3$&$\tilde{A}_1$&$\begin{pmatrix}
	0&1\\
	-1&0
\end{pmatrix}$\\
\hline
$\textrm{IV}$&$a\geq 2$&$2$&$4$&$\tilde{A}_2$&$\begin{pmatrix}
	0&1\\
	-1&-1
\end{pmatrix}$\\
\hline\hline
$\textrm{I}_0^*$&$a\geq 2$&$b\geq 3$&$6$&$\tilde{D}_4$&$\begin{pmatrix}
	-1&0\\
	0&-1
\end{pmatrix}$\\
\hline
$\textrm{I}_N^*$&$2$&$3$&$N+6\geq 7$&$\tilde{D}_{N+4}$&$\begin{pmatrix}
	-1&-N\\
	0&-1
\end{pmatrix}$\\
\hline
$\textrm{IV}^*$&$a\geq 3$&$4$&$8$&$\tilde{E}_6$&$\begin{pmatrix}
	-1&-1\\
	1&0
\end{pmatrix}$\\
\hline
$\textrm{III}^*$&$3$&$b\geq 5$&$9$&$\tilde{E}_7$&$\begin{pmatrix}
	0&-1\\
	1&0
\end{pmatrix}$\\
\hline
$\textrm{II}^*$&$a\geq 4$&$5$&$10$&$\tilde{E}_8$&$\begin{pmatrix}
	0&-1\\
	1&1
\end{pmatrix}$\\
\hline
\end{tabular}
\caption{Kodaira's Classification}\label{Table:KodairaClass}
\end{table}

Now we describe the weight one integral variation of Hodge structure on the elliptic surface $f\colon \mathcal{E}\to S$ associated to the homological invariant $\mathcal{V}$. 
The polarization on $\mathcal{V}$ is induced by the cup-product and if we  choose a basis $\alpha^*,\beta^*\in H^1(\mathcal{E}_t,\mathbf{Z})$ that is Poincar\'{e}-dual to the standard cycles $\alpha,\beta\in H_1(\mathcal{E}_t,\mathbf{Z})$, then the matrix of the polarization with respect to this basis is given by 
\begin{equation}
	\begin{pmatrix}
		0&1\\
		-1&0
	\end{pmatrix}.
\end{equation}

The Hodge filtration on $H^1(\mathcal{E}_t,\mathbf{C})$ is determined by a non-zero $\omega\in H^1(\mathcal{E}_t,\mathbf{C})$ that spans the filtrant $F^1$. 
If we write $\omega=z_1\alpha^*+z_2\beta^*$ in terms of the standard basis for $H^1(\mathcal{E}_t,\mathbf{Z})$, then the Hodge-Riemann relations imply that $\frac{z_2}{z_1}\in\mathfrak{h}$, and conversely.

Given the Weierstrass form \eqref{Weierstrass_Model}, there are a number of methods one can use to compute the Picard-Fuchs equations, such as the Griffiths-Dwork algorithm. 
This calculation was carried out in \cite{clingher_normal_2007} for example, where it is shown that the Picard-Fuchs equation is
\begin{equation}\label{Picard_Fuchs_EllSurface_WForm}
	\frac{d^2f}{dt^2}+P\frac{df}{dt}+Qf=0,
\end{equation}
where 
\begin{eqnarray*}
	P&=&\frac{\frac{dg_3}{dt}}{g_3}-\frac{\frac{dg_2}{dt}}{g_2}+\frac{d\mathcal{J}}{\mathcal{J}}-\frac{\frac{d^2\mathcal{J}}{dt^2}}{\frac{d\mathcal{J}}{dt}},\\
	Q&=&\frac{\left(\frac{d\mathcal{J}}{dt}\right)^2}{144\mathcal{J}(\mathcal{J}-1)}+\frac{\frac{d\Delta}{dt}}{12\Delta}\left(P+\frac{\frac{d^2\Delta}{dt^2}}{\frac{d\Delta}{dt}}-\frac{13}{12}\frac{\frac{d\Delta}{dt}}{\Delta}\right).
\end{eqnarray*}

More than just being a Fuchsian differential equation, the Picard-Fuchs equations arising from elliptic surfaces are $K$-equations, a term coined by Stiller \cite{stiller_differential_1981}.
\begin{definition}
	A second order Fuchsian ODE is called a \emph{$K$-equation} if it possesses two solutions $\omega_1,\omega_2$ which are holomorphic non-vanishing multivalued functions on a Zariski open subset $S_0\ins S$ satisfying the following conditions:
	\begin{enumerate}[label=(\roman{*})]
	\item $\omega_1,\omega_2$ form a basis of solutions;
	\item the monodromy representation of the differential equation with respect to this basis takes values in $\textrm{SL}_2(\mathbf{Z})$;
	\item $\textrm{im}(\frac{\omega_2}{\omega_1})>0$ on $S_0$ (positivity);
	\item the Wronskian lies in $K(S)$.
	\end{enumerate}
	The pair $\omega_1,\omega_2$ is called a \emph{$K$-basis}. 

\end{definition}
The Picard-Fuchs equation of an elliptic surface is a $K$-equation, with the $K$-basis being induced by the standard $a$ and $b$ cylces on the torus. 
Conversely, every $K$-equation can be realized as the Picard-Fuchs equation associated to an elliptic surface \cite{stiller_differential_1981}[Theorem II.2.5].

Let us briefly go over some of the details of how this works. 
Stiller starts by considering the elliptic surface $\mathcal{E}\to\mathbf{P}^1_t$ given by the Weierstrass presentation
\begin{equation}\label{J=t Form}
y^2=4x^3-\frac{27t}{t-1}x-\frac{27t}{t-1}.	
\end{equation}
The functional invariant in this case is $\mathcal{J}=t$, and the Picard-Fuchs equation associated to this elliptic surface is
\begin{equation}\label{Picard_Fuchs_J=t}
\frac{d^2f}{dt^2}+\frac{1}{t}\frac{df}{dt}+\frac{\frac{31}{144}t-\frac{1}{36}}{t^2(t-1)^2}f=0.
\end{equation}

An explicit $K$-basis $\Phi_1,\Phi_2$ for \eqref{J=t Form} is constructed in \cite{stiller_differential_1981} and the quotient of these solutions $\Phi=\frac{\Phi_1}{\Phi_2}$ induces a holomorphic multivalued map 
\begin{equation}
\xymatrix{\mathbf{P}^1_t-\{0,1,\infty\}\ar[r]^(0.7){\Phi}&\mathfrak{h}},	
\end{equation}
because of the positivity condition of a $K$-basis. 
In fact, $\Phi$ is an inverse to the classical modular $J$-function. 
If one makes a branch cut on $\mathbf{P}^1_t$ joining $\infty$ to $0$ along the negative real-axis and another branch cut along the interval $[0,1]$, then one can choose a single-valued branch of $\Phi$ on the slit-sphere that takes values in the usual fundamental domain for the $\textrm{SL}_2(\mathbf{Z})$-action on $\mathfrak{h}$:
$$\{\tau\in\mathfrak{h}|\ -\frac{1}{2}<\textrm{Re}(\tau)<\frac{1}{2},\ |\tau|>1\}.$$

Choosing loops $\gamma_0,\gamma_1,$ based at $i$, looping around $0$ and $1$ once,  the monodromy transformations are:
\begin{equation}\label{Monodromy:J=t}
	\begin{array}{ccc}
		\gamma_0&\mapsto &\begin{pmatrix}
			1&1\\
			-1&0
		\end{pmatrix}\\
		\gamma_1&\mapsto &\begin{pmatrix}
			0&-1\\
			1&0
		\end{pmatrix}\\
		\gamma_\infty &\mapsto &\begin{pmatrix}
			1&1\\
			0&1
		\end{pmatrix}.
	\end{array}
\end{equation}
In particular, the elliptic surface defined by \eqref{J=t Form} has singular fibres of type $\textrm{III}^*$ at $t=0$, $\textrm{II}$ at $t=1$ and $\textrm{I}_1$ at $t=\infty$, as can be seen by consulting Table \ref{Table:KodairaClass}. 

Now consider an arbitrary $K$-equation 
\begin{equation}\label{K-Equation}
	\frac{d^2f}{dt^2}+P\frac{df}{dt}Q+Rf=0,
\end{equation}
and let $\omega_1,\omega_2$ be a $K$-basis. 
The ratio of these two solutions gives rise to a multivalued map to $\mathfrak{h}$ and the composition $\mathcal{J}:=J\circ\frac{\omega_2}{\omega_1}$ is a \emph{rational} function from $S_0$ to $\mathbf{P}^1$ \cite{stiller_differential_1981}, called the \emph{functional invariant} of the $K$-equation. 

\begin{proposition*}\label{Prop:K-EquationClass}[\cite{stiller_differential_1981}]
	Given the $K$-equation \eqref{K-Equation}, there exists an algebraic function $\lambda$ satisfying $\lambda^2\in K(S)$ for which the original $K$-equation is obtained by pulling back \eqref{Picard_Fuchs_J=t} by the functional invariant $\mathcal{J}$ and then scaling by $\lambda$. 
	Explicitly, we have 
	\begin{eqnarray*}
	P&=&\frac{\left(\frac{d\mathcal{J}}{dt}\right)^2-\mathcal{J}\frac{d^2\mathcal{J}}{dt^2}}{\mathcal{J}\frac{d\mathcal{J}}{dt}}-\frac{d}{dt}\log\lambda^2,\\
	Q&=&\frac{\left(\frac{d\mathcal{J}}{dt}\right)^2\left(\frac{31}{144}\mathcal{J}-\frac{1}{36}\right)}{\mathcal{J}^2(\mathcal{J}-1)^2}-\left(\frac{\left(\frac{d\mathcal{J}}{dt}\right)^2-\mathcal{J}\frac{d^2\mathcal{J}}{dt^2}}{\mathcal{J}\frac{d\mathcal{J}}{dt}}\right)\frac{d}{dt}\log\lambda-\frac{\frac{d^2\lambda}{dt^2}}{\lambda}+2\left(\frac{\frac{d\lambda}{dt}}{\lambda}\right)^2,
\end{eqnarray*}

In particular, the Picard-Fuchs equation of an arbitrary elliptic surface can be computed this way by taking $\lambda^2=\frac{g_2}{g_3}$.

Finally, every $K$-equation is the Picard-Fuchs equation of an elliptic surface.

\end{proposition*}

\begin{remark}
Proposition \ref{Prop:K-EquationClass} is useful to compute the homological invariant of an arbitrary elliptic surface in terms of the homological invariant of the well-understood elliptic surface given by \eqref{J=t Form}, as long as we understand the functional invariant. 
Indeed, the monodromy representation of the Picard-Fuchs equation is obtained as follows. 
The functional invariant induces a push-forward map
$$\mathcal{J}_*\colon\pi_1(X_0)\to\pi_1(\mathbf{P}^1-\{0,1,\infty\}).$$
Composing this map with the monodromy representation \eqref{Monodromy:J=t}, we obtain the monodromy representation for the pull back of \eqref{Picard_Fuchs_J=t} by $\mathcal{J}$. 
Scaling the solutions by $\lambda=\sqrt{\frac{g_2}{g_3}}$ has the effect of multiplying the monodromy transformations by $-1$ at the positions where $\lambda$ has a pole or zero. 
Thus, if we understand the push-forward map $\mathcal{J}_*$, as well as where the poles and zeroes of $\lambda$ lie, we can compute the homological invariant precisely. 
\end{remark}

\subsection{Algebraic Description of Parabolic Cohomology}
In this section, we describe the lattice $H^2(\mathcal{E},\mathbf{Z})$, its weight two integral Hodge structure, and the piece of $H^2(\mathcal{E},\mathbf{Z})$ that corresponding to parabolic cohomology. 
A good reference for this section is \cite{schuett_elliptic_2010}.
The N\'{e}ron-Severi group, the group of divisors modulo algebraic equivalence, embeds in $H^2(\mathcal{E},\mathbf{Z})$ modulo via the \emph{cycle-map} and is denoted $\textrm{NS}(\mathcal{E})$. 
If we assume the curve $S$ is genus $0$, as is the case for all examples in this work, then $\textrm{NS}(\mathcal{E})$ is torsion-free. 
The cup-product on $H^2(\mathcal{E},\mathbf{Z})$ coincides with the intersection pairing on $\textrm{NS}(\mathcal{E})$, and makes $\textrm{NS}(\mathcal{E})$ a sub-lattice; its rank is called the Picard number, which is denoted by $\rho(\mathcal{E})$.
According to the Hodge index theorem, $\textrm{NS}(\mathcal{E})$ has signature $(1,\rho-1)$. 
The orthogonal compliment of $\textrm{NS}(\mathcal{E})$ in $H^2(\mathcal{E},\mathbf{Z})$ is called the \emph{transcendental lattice} and is denote by $\textbf{T}(\mathcal{E})$. 
The Hodge decomposition of $H^2(\mathcal{E},\mathbf{Z})$ induces weight $2$ Hodge structures on both the N\'{e}ron-Severi lattice and the transcendental lattice. 
By the Lefschetz $(1,1)$-theorem, we have
\begin{equation}
	H^2(\mathcal{E},\mathbf{Z})\cap H^{1,1}=\textrm{NS}(\mathcal{E}),
\end{equation} 
from which it follows that the $h^{2,0}(\textrm{NS}(\mathcal{E}))=0$. 
Typically, we are most interested in the Hodge structure on the transcendental lattice.

Each section $\sigma$ of the elliptic surface $\mathcal{E}\to S$ corresponds to a $K(S)$-rational point on the generic fibre, making the generic fibre an elliptic surface over the function field $K(S)$. 
If we fix one section $\sigma_0$ as an origin, then the group law on the generic fibre induces a group law on the set of sections.

\begin{definition}
	The group of sections of an elliptic surface $f\colon \mathcal{E}\to S$ is called the \emph{Mordell-Weil} group of $\mathcal{E}\to S$ and is denoted $\textrm{MW}(\mathcal{E})$.
	The \emph{narrow Mordell-Weil group}, denote by $\textrm{MW}^0(\mathcal{E})$ is the subgroup of sections that meet the zero component of every fibre. 
The narrow Mordell-Weil group is torsion-free and has finite-index in the full Mordell-Weil group \cite{schuett_elliptic_2010}.
\end{definition}

\begin{definition}
	 The subgroup of $\textrm{NS}(\mathcal{E})$ generated by the zero section, general fibre, and the components of the bad fibres is called the \emph{trivial lattice}, and is denoted by $T$.
\end{definition}

The trivial lattice decomposes as follows: 
\begin{equation}
	T=\gen{\sigma_0,F}\oplus\bigoplus_{s\in \Sigma}T_s,
\end{equation}
where $F$ denotes the class of a good fibre, $\Sigma$ denotes the set of bad fibres, and $T_s$ is the lattice generate by the fibre components \emph{not} meeting $\sigma_0$.
Thus, the rank of the trivial lattice is
$$\textrm{rank}(T)=2+\sum_{t\in \Sigma}(m_s-1),$$
where $m_s$ denotes the number of fibre components.

Since each section corresponds to a divisor on $\mathcal{E}$, we obtain a map from the Mordell-Weil group to the N\'{eron}-Severi group. 
The induced map $P\mapsto P\bmod T$ induces an isomorphism \cite{schuett_elliptic_2010}
\begin{equation}
	\textrm{MW}(\mathcal{E})\cong\textrm{NS}(\mathcal{E})/T.
\end{equation}
This allows us to compute the rank of $\textrm{MW}(\mathcal{E})$:
$$\textrm{rank}(\textrm{MW}(\mathcal{E}))=\rho-2-\sum_{t\in R}(m_t-1).$$
The trivial lattice does \emph{not} embed primitively inside $\textrm{NS}(\mathcal{E})$; the cokernel is isomorphic to the torsion subgroup of $\textrm{MW}(\mathcal{E})$.

\begin{definition}
	The \emph{essential lattice} $\textrm{L}(\mathcal{E})$ is the orthogonal compliment of the trivial lattice inside $\textrm{NS}(\mathcal{E})$. 
\end{definition}

The essential lattice is even and negative-definite of rank equal to $\rho-2-\sum_x(m_x-1)$, i.e., has rank equal to the Mordell-Weil group. 
Orthogonal projection with respect to the trivial lattice defines a map
\begin{equation}
	\xymatrix{\ph\colon \textrm{NS}_\mathbf{Q}(\mathcal{E})\ar[r]&\textrm{L}(\mathcal{E})_\mathbf{Q}}.
\end{equation} 
This map is characterized by the universal properties:
$$\ph(D)\perp T(\mathcal{E})_\mathbf{Q},\ \ \ph(D)\equiv D\bmod T(\mathcal{E})_\mathbf{Q}.$$
By restricting to $\textrm{MW}(\mathcal{E})$, we obtain a well-defined map from the Mordell-Weil group $$\textrm{MW}(\mathcal{E})\to\textrm{L}(\mathcal{E})_\mathbf{Q},$$ the kernel of which is the torsion subgroup. 
Thus, $\textrm{MW}(\mathcal{E})/\{\textrm{torsion}\}$ sits inside $\textrm{L}(\mathcal{E})_\mathbf{Q}$.

\begin{remark}
	In general, if a section $\sigma$ hits a non-zero fibre component, then its image in $\textrm{L}(\mathcal{E})_\mathbf{Q}$ will \emph{not} lie in the integral part; that is, tensoring with $\mathbf{Q}$ is necessary to define the above homomorphism.
\end{remark}

This map can be used to give the Mordell-Weil group the structure of a \emph{positive-definite} lattice by setting 
$$\gen{\sigma_1,\sigma_2}:=-\ph(\sigma_1)\cdot\ph(\sigma_2).$$
\begin{definition}
	The lattice $\textrm{MW}(\mathcal{E})/\{\textrm{torsion}\}$ is called the \emph{Mordell-Weil lattice}. 
	The sublattice $\textrm{MW}^0(\mathcal{E})$ is called the \emph{narrow Mordell-Weil lattice}. 
\end{definition}

\begin{remark}
	The lattice structure on $\textrm{MW}(\mathcal{E})$ was first discovered by Cox-Zucker in \cite{zucker_intersection_1979}. 
\end{remark}


Let us now consider the parabolic cohomology $H^1(S,j_*\mathcal{V})$ and its Hodge structure.
The torsion on $H^1(S,j_*\mathcal{V})$ is isomorphic to the torsion subgroup of the Mordell-Weil group, according to \cite{zucker_intersection_1979}.  
Modulo torsion, the parabolic cohomology group sits inside $H^2(\mathcal{E},\mathbf{Z})$, and its Hodge structure agrees with the one that is induced from this embedding, as discussed previously.
According to \cite{zucker_intersection_1979}, the Leray spectral sequence for $f\colon \mathcal{E}\to S$ degenerates at the $E_2$-level over $\mathbf{Q}$. 
One can show that it degenerates at the $E_2$-level over $\mathbf{Z}$ if the Mordell-Weil group is torsion-free \cite{zucker_intersection_1979}. 
The parabolic cohomology group can be computed in terms of the Leray spectral sequence. 
The Leray filtration on $H^2(\mathcal{E},\mathbf{Z})$ is given by 
\begin{equation}\label{Leray_Filt}
\begin{array}{lll}
	L^1&=&\textrm{ker}\left(H^2(\mathcal{E},\mathbf{Z})\to H^0(S,R^2f_*\mathbf{Z})\right)\\
	L^1/L^2&\cong &H^1(S,j_*\mathcal{V})\\
	L^2&=&\textrm{image}\left(H^2(S,\mathbf{Z})\to H^2(\mathcal{E},\mathbf{Z})\right)\cong \mathbf{Z}[F].
\end{array}
\end{equation}

Let $P$ be the image of parabolic cohomology inside $H^2(\mathcal{E},\mathbf{Z})$. Define the \emph{algebraic} and \emph{transcendental} pieces of $P$ as follows:
$$P_{\textrm{alg}}:=P\cap NS(\mathcal{E}),\ \ P_{\textrm{tr}}:=P\cap\mathbf{T}(\mathcal{E}).$$
The following proposition describes how these lattices relate to each other. 
\begin{proposition}\label{Prop:EllipticSurfacePCohomStructure}
The parabolic cohomology lattice $P$ is equal to the orthogonal compliment of $T$ inside $H^2(\mathcal{E},\mathbf{Z})$:
$$P=T^\perp\ins H^2(\mathcal{E},\mathbf{Z}).$$
We have $P_{\textrm{alg}}=L(\mathcal{E})$ and $P_{\textrm{tr}}=\mathbf{T}(\mathcal{E})$, and $\mathbf{T}(\mathcal{E})=L(\mathcal{E})^\perp\ins P$. There is a decomposition
$$P_\mathbf{Q}=L(\mathcal{E})_\mathbf{Q}\oplus\mathbf{T}(\mathcal{E})_\mathbf{Q};$$
this splitting holds over $\mathbf{Z}$ if $L(\mathcal{E})$ is unimodular.
In particular, the transcendental lattice is the intersection of $L(\mathcal{E})_\mathbf{Q}^\perp$ and $P$, i.e., 
$$\mathbf{T}(\mathcal{E})_\mathbf{Z}=P_\mathbf{Z}\cap L(\mathcal{E})_\mathbf{Q}^\perp.$$

\end{proposition}
\begin{proof}
This all follows from the definitions of the lattices, the results of Cox-Zucker in \cite{zucker_intersection_1979}, and basic lattice-theory.


\end{proof}

\begin{remark}
	In particular:
	\begin{itemize}
		\item if $\mathcal{E}$ is a rational elliptic surface, then the parabolic cohomology lattice is equal to the essential lattice;
		\item if the Mordell-Weil rank is zero, then the parabolic cohomology lattice is equal to the transcendental lattice.
	\end{itemize}
\end{remark}


\begin{remark}
From this decomposition, we see that the parabolic cohomology group over $\mathbf{Q}$ breaks into one piece that does not depend on the fibration structure---the transcendental lattice---and another that \emph{does}---the Mordell-Weil lattice. 
When we study these in families, the corresponding parabolic cohomology local systems decompose into one ``extrinsic'' local system, capturing information about the transcendental data and one ``intrinsic'' local system, telling us information about the varying internal fibration structure. 
\end{remark}

\subsection{De Rham Description of Parabolic Cohomology}
In this section, we recall the de Rham description of parabolic cohomology, as given by Stiller \cite{stiller_picard_1987}. 
This gives us a dictionary between parabolic cocycles and differential 2-forms on the corresponding elliptic surface. 
Proofs of the claims in this section can be found in \cite{stiller_picard_1987}.

Let $\mathcal{E}\to T$ be an elliptic surface, and let $\Lambda$ be the associated Picard-Fuchs operator. 
We consisder inhomogenous equations of the form
$\Lambda f=Z$, where $Z$ is a rational function. 
\begin{definition}
The inhomogenous equation 
$$\Lambda f=Z,\ Z\in K(T)$$
is \emph{exact} if it has a global holomorphic solution. 
In this case, $Z=\Lambda Z'$ for some rational function $Z'$. We denote the space of exact equations by $\Lambda K(T)$. 

The above inhomogenous equation is \emph{locally exact} if for each point $p$, the equation restricted to a small open disc has a single-valued solution. The space of locally exact equation is denoted $L_\Lambda^{\textrm{para}}$.

The \emph{inhomogenous de Rham cohomology} $H^1_{\textrm{DR}}$ is the space
$$H^1_{\textrm{DR}}=L_\Lambda^{\textrm{para}}/\Lambda K(T).$$
\end{definition}

Stiller shows that $H^1_{\textrm{DR}}$ is isomorphic to the parabolic cohomology of the elliptic surface, as defined previously. 
We do not review the details of the argument here, but we will explain the dictionary. 
Choose a $K$-basis $\omega_1,\omega_2$ for the Picard-Fuchs operator $\Lambda$ and consider the inhomogenous equation 
$$\Lambda f=Z.$$
If $f$ is a solution, then analytic continuation along a loop $\gamma$ acts as
$$f^\gamma\mapsto f+m_\gamma\omega_1+n_\gamma\omega_2.$$
The map $\delta_Z(\gamma)=(m_\gamma,n_\gamma)\in\mathbf{Z}^2$ is a co-cycle and the assignment $Z\mapsto \delta_Z$ induces the dictionary between inhomogenous de Rham cohomology and the cocycle description of parabolic cohomology given earlier. 

The kinds of rational functions $Z$ that can appear as the right hand side of a locally exact equation are tightly controlled by the Picard-Fuchs operator $\Lambda$, as is the Hodge filtration of the parabolic cohomology. 
Specifically, the rational function $Z$ needs to satisfy the property that
$$\omega\frac{Z}{W}dt$$
has zero residue whenever $\omega$ is a single-valued solution to $\Lambda$ and $W$ is the Wronskian. 
Stiller describes in \cite{stiller_picard_1987}, two divisors $\mathcal{A},\mathcal{A}_0$ that depend only on $\Lambda$ for which the linear series and $L(\mathcal{A}_0)\ins L(\mathcal{A})$ determine the Hodge filtration.
Every rational function $Z\in L(\mathcal{A}_0)$ satisfies the zero-residue condition at every point and so the space $L(\mathcal{A}_0)$ can be identified with the second Hodge filtrant; the first filtrant consists of the functions in $L(\mathcal{A})$ that satisfy the zero-residue condition; this space is denoted by $L^\textrm{para}(\mathcal{A})$.
Phrased differntly, if we take  $Z\in L(\mathcal{A}_0)$ (resp. $L^\textrm{para}(\mathcal{A}))$, then it is shown in \cite{stiller_picard_1987} that the differential form
$$-\frac{Z}{W}dt\wedge\frac{dx}{y}$$
is a holomorphic $(2,0)$ (resp. $(1,1)$ form) on the elliptic surface $\mathcal{E}$. 

\begin{remark}
This last paragraph explains how to construct a global $(2,0)$ form from the family of holomorphic $1$-forms governing the fibration structure. 
Since the $(2,0)$-piece of parabolic cohomology agrees with the $(2,0)$-piece of the full Hodge structure of the total space \cite{stiller_picard_1987}, it follows that we can \emph{always} construct a basis of holomorphic $(2,0)$-forms for the total space from holomorphic families of $1$-forms on the fibres. 
\end{remark}

\begin{example}
Let $\mathcal{E}$ be the rational elliptic surface given by 
$$Y^2=4X^3+4t^4X+t.$$
Then $\mathcal{E}$ has ten $\I_1$ fibres located at the roots of $4t^{10}+27$ and a type $\textrm{II}$ fibre located at $t=0$. 
The Picard-Fuchs equation for $\mathcal{E}$ is 
\begin{equation}
    \frac{d^2f}{dt^2}+\frac{32t^{10}-54}{t(4t^{10}+27)}\frac{df}{dt}+\frac{96t^{10}-57}{4t^2(4t^{10}+27)}f=0.
\end{equation}
Note that $t=\infty$ is a singularity of the Picard-Fuchs operator, but is not a singular fibre---it is an \emph{apparent singularity}.

Following \cite{stiller_picard_1987}, the divisors $\mathcal{A}$ and $\mathcal{A}_0$ are given by
\begin{eqnarray*}
\mathcal{A}_0&=&-9(\infty)-2(0)+\sum_{i=1}^{10}(p_i)\\
\mathcal{A}&=&-4(\infty)+2(0)+\sum_{i=1}^{10}(p_i).
\end{eqnarray*}
Note that $\deg(\mathcal{A}_0)=-1$, which implies that $L(\mathcal{A}_0)=0$, reflecting the fact that $\mathcal{E}$ is rational; there are no global $(2,0)$-forms. 
On the other hand, $\mathcal{A}$ is has degree $9$. 
A general member of $L(\mathcal{A})$ looks like
$$Z=\frac{p(t)}{(4t^{10}+27)t^2},$$
where $\deg(p(t))\leq 8$. 
Following the procedure in \cite{stiller_picard_1987}, we only need to check the residue condition at the apparent singularity $t=\infty$.
One finds that $Z$ will satisfy the residue condition if and only if $p(t)=\sum_{i=0}^8a_it^i$ and $a_4=0$. 
Since the Wronksian is equal to a constant multiple of $\frac{t^2}{4t^{10}+27}$, it follows that a basis of $(1,1)$-forms is given by $$t^{i-4}dt\wedge\frac{dx}{y},\ i=0,1,2,3,5,6,7,8.$$

\end{example}

We close this section by describing how the Mordell-Weil lattice fits into the de Rham descripiton of parabolic cohomology. 
Given a section $\sigma$ of the elliptic surface, the function 
$$f(t):=\int_\mathcal{O}^\sigma\frac{dx}{y}$$
defines a multi-valued holmorphic function. Applying the Picard-Fuchs operator gives rise to a rational funcion; that is, 
$$Z_\sigma=\Lambda\int_\mathcal{O}^\sigma\frac{dx}{y}$$
is a rational function. 
In fact, $Z_\sigma$ lies in the space $L^\textrm{para}(\mathcal{A})$ and corresponds to a holomorphic $(1,1)$-form. 
The space of all such forms as $\sigma$ varies over the Mordell-Weil group exactly corresponds to the algebraic piece of parabolic cohomology described in the previous section.

\section{Middle Convolutions and Quadratic Twists}\label{sec:middle_convolution}
Variations of local systems are a generalization of the middle convolution functor on the category of local systems orginally introduced by Katz in \cite{katz_rigid_1996}.
For our technical results, the reader is referred to \cite{bogner_symplectically_2013} and for  a modern expository treatment of the middle convolution, the reader is referred to \cites{simpson_katzs_2006,dettweiler_middle_2008}.
In this section, we use the middle convolution to construct geometric variations of local systems corresponding to families of elliptic surfaces.
We begin by reviewing some theory of the middle convolution and the closely related middle Hadamard product. 
After our preparation, we describe a geometric construction of families of elliptic surfaces whose geometric variations of local systems correspond to the middle convolution. 
We then use this to calculate the integral variations of Hodge structure underlying these families.

\subsection{The Middle Convolution}
Here, we review some details about the middle convolution, most of which can be found in \cite{bogner_symplectically_2013}.
The Kummer sheaf $\mathcal{K}^{x_0}_\lambda$ for $\lambda\in\mathbf{C}^\times$ and $x\in\mathbf{C}$ is the rank $1$ local system $\mathcal{L}$ that has $\lambda$ monodromy at $x_0$ and and $\lambda^{-1}$ monodromy at $\infty$. 
Explicitly, the Kummer sheaf is the local system corresponding to the rank $1$ ODE
$$\frac{df}{dx}=\frac{\mu}{x-x_0}f(x),$$
where $\lambda=e^{2\pi i\mu}$.

Let $U_0=\mathbf{P}^1-\{x_1,\dots, x_{r-1},\infty\}$, $X=\mathbf{P}^1\times U_0$, and $\cls{\pi}\colon X\to U_0$ be the second projection. 
Then, with the divisor $D_y=\{y,x_1,\dots, x_{r-1},\infty\}$, $(X,D)$ is an $r+1$-configuration over $U_0$. 
Let $p\colon X\to\mathbf{P}^1$ be defined by
\begin{eqnarray*}
p\colon X&\to&\mathbf{P}^1\\
(x,y)&\mapsto & y-x,
\end{eqnarray*}
and let $j\colon U_0\to \mathbf{P}^1$ be the inclusion. 
Let $\mathcal{V}_0$ be a local system on $U_0=\mathbf{P}^1-\{x_1,\dots, x_{r-1},\infty\}$, and choose a basepoint $y_0\in U_0$. 
Then $\mathcal{V}:=j_*\mathcal{V}_0\otimes p^*\mathcal{K}^0_\lambda$ defines a variation of the local system $\tilde{\mathcal{V}}_0=\mathcal{V}_0\otimes\mathcal{K}^{y_0}_{\lambda}$ over the $r+1$ configuration $(X,D)$.
Note that monodromy of $\tilde{\mathcal{V}}_0$ is the same as that of $\mathcal{V}$ except at $x=y_0$, where the monodromy is given by multiplication-by-$\lambda$, and $x=\infty$, where the monodromy is now $\lambda^{-1}\cdot\rho(\gamma_\infty)$.

\begin{definition}
Notation as above, the \emph{middle convolution} $\textrm{MC}_{\lambda}(\mathcal{V}_0)$ of $\mathcal{V}_0$ with $\lambda$ is the parabolic cohomology local system of the previously defined variation of local systems. 
In particular, it is a local system on $U_0$. 
Fix a singular point $x_i$ of $\mathcal{V}_0$.
The \emph{middle Hadamard product} $\textrm{MH}_{-1}(\mathcal{V}_0)$ of $\mathcal{V}_0$ and $\lambda$ (at $x_0$) is the middle convolution of $\mathcal{V}_0\otimes\mathcal{K}^{x_i}_\lambda$ with $\lambda$.
That is, the Hadamard product of $\mathcal{V}_0$ with $\lambda$ is the same as the middle convolution of $\mathcal{V}$ with $\lambda$ after we twist the monodromy at $x=x_i$ and $\infty$. 
\end{definition}

Suppose now that $\mathcal{V}_0$ is the solution sheaf of the Fuchsian differential equation $Lf=0$ and let $\mu\in\mathbf{Q}-\mathbf{Z}$ and $p\in\mathbf{P}^1$. 
Let $\gamma_p$ and $\gamma_y$ be two simple closed loops based at $x_0$ that encircle $p$ and $y$ respectively; let $[\gamma_p,\gamma_y]=\gamma_p^{-1}\gamma_y^{-1}\gamma_p\gamma_y$ be the Pochammer contour.
If $f$ is a solution of $Lf=0$, then convolution and hadamard product of functions is defined as follows:
\begin{enumerate}
    \item $$C^p_\mu(f):=\int_{[\gamma_p,\gamma_y]}\frac{f(x)}{(y-x)^{1-\mu}}dx$$is the \emph{convolution} of $f$ and $y^\mu$ with respect to $[\gamma_p,\gamma_z]$;
    \item $$H^p_\mu(f):=\int_{[\gamma_p,\gamma_y]}\frac{f(x)}{(x-y)^\mu x^{1-\mu}}dx$$ is the \emph{Hadamard product} of $f$ and $(1-y)^{-\mu}$ with respect to $[\gamma_p,\gamma_y]$.
\end{enumerate}

According to \cite{bogner_symplectically_2013}, the middle convolution of an irreducible local system is itself an irreducible local system. 
Moreover, it is shown that if $\mathcal{V}_0$ is the solution sheaf of the Fuchsian differential equation $Lf=0$, then $C^p_\mu(f)$ is a flat section of the middle convolution sheaf. 
The analogous statements hold for the Hadamard product as well. 
Finally, an algorithm is given \cite{bogner_symplectically_2013} that computes a differential equation that annihilates the functions $C_\mu^p(f)$; it follows that if $\mathcal{V}_0$ is irreducible, then the differential operator that annihilates the middle convolution local system can be computed as a factor. 
All of this is implemented in Maple, which means that if we start with a Fuchsian differential equation, we can explicitly compute the differential operator that annihilates the middle convolution and middle Hadamard products. 

We close this section by relating the above Hadamard product to the classical Hadamard product of power series defined as
$$\sum_{n=0}^\infty a_nx^n\star\sum_{n=0}^\infty b_nx^n:=\sum_{n=0}^\infty a_nb_nx^n.$$ 
Notation as above, let $p=0$ and suppose that $f$ is a holomorphic solution to $f=0$ so that we may write
$$f=\sum_{n=0}^\infty a_nx^n,\ a_n\in\mathbf{C}.$$
The function $(1-y)^{-\mu}$ is holomorphic at $y=0$ and admits the following power series represenation:
$$\frac{1}{(1-y)^\mu}=\sum_{n=0}^\infty\frac{(\mu)_n}{n!}y^n,\ (\mu)_n:=\mu(\mu+1)\cdots(\mu+n-1).$$
Then, by applying the residue theorem and carefully considering the monodromy, one sees that
$$\frac{1}{4\pi i}H^0_\mu(f)=\sum_{n=0}^\infty\frac{a_n(\mu)_n}{n!}x^n.$$
In other words, up to a constant, $H^0_\mu(f)$ is the Hadamard product of the power series of $f$ and $(1-y)^{-\mu}$.

\subsection{Quadratic Twists and the Middle Convolution}
In \cites{besser_picard-fuchs_2012,laza_universal_2013}, Besser-Livn\'{e} consider families of elliptic surfaces obtained by starting with a fixed elliptic surface $\mathcal{E}$ and performing a quadratic twist at some fixed singular fibre and one mobile smooth fibre.
They provide an algorithm that calculates a differential equation that annihilates periods for such families and use this algorithm to compute the Picard-Fuchs equations for certain rank $19$ K3 surfaces associated to Shimura curves. 
Specifically, they consider these twisted families in the cases that the starting elliptic surface $\mathcal{E}$ is a rigid rational elliptic surface with three multiplicative fibres and one additive fibre, with the additive fibre playing the role of the fixed twisted fibre.

In this section, we revisit this construction through the lens of the middle convolution functor defined above. 
In particular, this allows us to calculate the integral variations of Hodge structure underlying the families of elliptic surfaces arising from Besser-Livn\'{e}'s constuction, as well as the corresponding Picard-Fuchs equations. 
We proceed to calculate the integral variations of Hodge structure underlying the twisted families corresponding to all 38 isolated rational elliptic surface with four fibres found in Herfurtner's classification \cite{herfurtner_elliptic_1991}.

Let $\mathcal{E}$ be any elliptic surface over $\mathbf{P}^1_t$; let $\Sigma=\{t_1,\dots, t_{r-1},\infty\}$ denote the support of the singular fibres and $U_0$ the compliment $\mathbf{P}^1-\Sigma$; denote by $\mathcal{V}_0$ the homological invariant. 
Fix a singular fibre $t_i$ and let $a\in U_0$.
Let $\mathcal{E}^{t_i}_a$ denote the elliptic surface obtained by performing a quadratic twist at $t=t_i$ and $t=a$. 
Then, $$\mathcal{E}^{t_i}_a\to U_0$$ is a family of elliptic surfaces parameterized by $U_0$.
For each fixed $a$, $\mathcal{E}^{t_i}_a$ has singular fibres at $\Sigma\cup\{a\}\ins\mathbf{P}^1_t$ and the family $\mathcal{E}^{t_i}_a$ is a family of elliptic curves over the $r+1$ configuration $(X,D)$ where $X=\mathbf{P}^1\times U_0$ and $D_a=\{a,t_1,\dots, t_{r-1},\infty\}$.
Let $\mathcal{V}^{t_i}$ denote the corresponding geometric variation of local system and note that
$$\mathcal{V}^{t_i}_a\cong\mathcal{V}_0\otimes\mathcal{K}_{-1}^{a}\otimes\mathcal{K}_{-1}^{t_i},$$
which is the local system obtained by twisting $\mathcal{V}_0$ at $t=t_i$ and $t=a$ by $-1$. 
It follows that the parabolic cohomology of $\mathcal{V}^{\infty}$ is equal to the middle convolution $\textrm{MC}_{-1}(\mathcal{V}_0)$; the parabolic cohomology of $\mathcal{V}^{t_i}$ is equal to the middle Hadamard product $\textrm{MH}_{-1}(\mathcal{V}_0)$.

Summing up, we have:
\begin{proposition}\label{prop:middle_convolution}
Let $f\colon\mathcal{E}\to \mathbf{P}^1_t$ be an elliptic surface with singular fibres located at $\Sigma=\{t_1,\dots, t_{r-1},t_r=\infty\}$, and let $\mathcal{V}_0=R^1f_*\mathbf{Z}$ be the homological invariant.
Then, the parabolic cohomology of the geometric variation of local systems $\mathcal{V}^{t_i}$ is is equal to the middle Hadamard product $\textrm{MH}_{-1}(\mathcal{V}_0)$; the parabolic cohomology of the geometric variation of local systems $\mathcal{V}^\infty$ is the middle convolution $\textrm{MC}_{-1}(\mathcal{V}_0)$.  
In particular, the parabolic cohomology local system $\mathcal{W}$ is irreducible. 

If $\nu$ denotes the number of multiplicative fibres in the fibration $f$, then the rank of $\mathcal{W}$ is  is equal to
$$r_\mathcal{W}=\left\{\begin{array}{cl}
     2r-\nu-1&\textrm{if $t_i$ is of type $\I_m, m\geq 1$}  \\
    2r-\nu-2 &\textrm{if $t_i$ is of type $\textrm{II},\textrm{II}^*,\textrm{III},\textrm{III}^*,\textrm{IV},\textrm{IV}^*$}\\
    2r-\nu-3&\textrm{if $t_i$ is of type $\I_m^*,m\geq 1$}\\
    2r-\nu-4&\textrm{if $t_i$ is of type $\I_0^*$}
\end{array}\right.$$
Finally, the Jordan forms of the local monodromy transformations of  $\mathcal{W}$ are tabulated in Tables \ref{tab:middle_conv_jordan1} and \ref{tab:middle_conv_jordani}.

\begin{table}[]
    \centering
    \resizebox{\linewidth}{!}{
    \begin{tabular}{||c|c|c||}
    \hline 
    Fibre Type&$J(M_j)$&$J(M_\infty)$\\
    \hline\hline 
        $\I_0$&$[1]\oplus\cdots\oplus[1]$&$[-1]\oplus\cdots\oplus[-1]$\\
        $\I_N$ & $[-1]\oplus[1]\oplus\cdots\oplus[1]$ &$\begin{bmatrix}-1&1&0\\0&-1&1\\0&0&-1\end{bmatrix}\oplus[-1]\oplus\cdots\oplus[-1]$\\
        
        $\textrm{II}$&$[e^\frac{2\pi i}{3}]\oplus[e^\frac{-2\pi i}{3}]\oplus[1]\oplus\cdots\oplus[1]$
        &$[e^\frac{2\pi i}{3}]\oplus[e^\frac{-2\pi i}{3}]\oplus[-1]\oplus\cdots\oplus[-1]$\\
        
        $\textrm{III}$&$[i]\oplus[-i]\oplus[1]\oplus\cdots[1]$&$[i]\oplus[-i]\oplus[-1]\oplus\cdots[-1]$\\
        
        $\textrm{IV}$&$[e^\frac{2\pi i}{6}]\oplus[e^\frac{-2\pi i}{6}]\oplus[1]\oplus\cdots\oplus[1]$&$[e^\frac{2\pi i}{6}]\oplus[e^\frac{-2\pi i}{6}]\oplus[-1]\oplus\cdots\oplus[-1]$\\
        $\I_N^*$&$\begin{bmatrix}-1&1&0\\0&-1&1\\0&0&-1\end{bmatrix}\oplus[1]\oplus\cdots\oplus[1]$&$[-1]\oplus\cdots\oplus[-1]$\\
        $\textrm{IV}^*$&$[e^\frac{2\pi i}{6}]\oplus[e^\frac{-2\pi i}{6}]\oplus[1]\oplus\cdots\oplus[1]$&$[e^\frac{2\pi i}{6}]\oplus[e^\frac{-2\pi i}{6}]\oplus[-1]\oplus\cdots\oplus[-1]$\\
        $\textrm{III}^*$&$[i]\oplus[-i]\oplus[1]\oplus\cdots[1]$&$[i]\oplus[-i]\oplus[-1]\oplus\cdots[-1]$\\
        $\textrm{II}^*$&$[e^\frac{2\pi i}{3}]\oplus[e^\frac{-2\pi i}{3}]\oplus[1]\oplus\cdots\oplus[1]$ &$[e^\frac{2\pi i}{3}]\oplus[e^\frac{-2\pi i}{3}]\oplus[-1]\oplus\cdots\oplus[-1]$\\
        \hline 
    \end{tabular}}
    \caption{Local monodromies for the family $\mathcal{E}^\infty$}
    \label{tab:middle_conv_jordan1}
\end{table}

\begin{table}[]
    \centering
    \resizebox{\linewidth}{!}{
    \begin{tabular}{||c|c|c|c||}
    \hline 
    Fibre Type&$J(M_j)$&$J(M_\infty)$&$J(M_i)$\\
    \hline\hline 
        $\I_0$&$J(M_j)$&$[-1]\oplus\cdots\oplus[-1]$&$[1]\oplus\cdots\oplus[1]$\\
        $\I_N$ &$J(M_j)$& $[1]\oplus[-1]\oplus\cdots\oplus[-1]$ &$\begin{bmatrix}1&1&0\\0&1&1\\0&0&1\end{bmatrix}\oplus[1]\oplus\cdots\oplus[1]$\\
        
        $\textrm{II}$&$J(M_j)$&$[e^\frac{2\pi i}{6}]\oplus[e^\frac{-2\pi i}{6}]\oplus[-1]\oplus\cdots\oplus[-1]$
        &$[e^\frac{2\pi i}{6}]\oplus[e^\frac{-2\pi i}{6}]\oplus[1]\oplus\cdots\oplus[1]$\\
        
        $\textrm{III}$&$J(M_j)$&$[i]\oplus[-i]\oplus[-1]\oplus\cdots[-1]$&$[i]\oplus[-i]\oplus[1]\oplus\cdots[1]$\\
        
        $\textrm{IV}$&$J(M_j)$&$[e^\frac{2\pi i}{3}]\oplus[e^\frac{-2\pi i}{3}]\oplus[-1]\oplus\cdots\oplus[-1]$&$[e^\frac{2\pi i}{3}]\oplus[e^\frac{-2\pi i}{3}]\oplus[1]\oplus\cdots\oplus[1]$\\
        $\I_N^*$&$J(M_j)$&$\begin{bmatrix}1&1&0\\0&1&1\\0&0&1\end{bmatrix}\oplus[-1]\oplus\cdots\oplus[-1]$&$[1]\oplus\cdots\oplus[1]$\\
        $\textrm{IV}^*$&$J(M_j)$&$[e^\frac{2\pi i}{3}]\oplus[e^\frac{-2\pi i}{3}]\oplus[-1]\oplus\cdots\oplus[-1]$&$[e^\frac{2\pi i}{3}]\oplus[e^\frac{-2\pi i}{3}]\oplus[1]\oplus\cdots\oplus[1]$\\
        $\textrm{III}^*$&$J(M_j)$&$[i]\oplus[-i]\oplus[-1]\oplus\cdots[-1]$&$[i]\oplus[-i]\oplus[1]\oplus\cdots[1]$\\
        $\textrm{II}^*$&$J(M_j)$&$[e^\frac{2\pi i}{6}]\oplus[e^\frac{-2\pi i}{6}]\oplus[-1]\oplus\cdots\oplus[-1]$ &$[e^\frac{2\pi i}{6}]\oplus[e^\frac{-2\pi i}{6}]\oplus[1]\oplus\cdots\oplus[1]$\\
        \hline 
    \end{tabular}}
    \caption{Local monodromies for the family $\mathcal{E}^{t_i}$}
    \label{tab:middle_conv_jordani}
\end{table}

\end{proposition}

\begin{proof}
Identifying the parabolic cohomology of the geometric variation of local system with the appropriate middle convolution was done above. 
The statement about the rank of parabolic cohomology follows from the rank formula \eqref{rank_formula}.
\end{proof}

Next, we prove some results about the geometry and period functions for these families of twists. 

\begin{proposition}\label{prop:EulerChar}
Let $f\colon\mathcal{E}\to\mathbf{P}^1_t$ be an elliptic surface with singular fibres located at $\Sigma=\{t_1,\dots, t_r=\infty\}$. 
Let $\Sigma^+$ denote the support of fibres of types $\I_N,\textrm{II},\textrm{III},\textrm{IV}$,  and $\Sigma^-$ the support of fibres of types $\I_N^*,\textrm{IV}^*,\textrm{III}^*,\textrm{II}^*$. 
The Euler characteristics of $\mathcal{E}$ and $\mathcal{E}^{t_i}_a$ are related as follows:
$$\chi(\mathcal{E}^{t_i}_a)=\left\{\begin{array}{ll}
\chi(\mathcal{E})+12&\textrm{if }\  t_i\in \Sigma^+\\
\chi(\mathcal{E})&\textrm{if }\  t_i\in \Sigma^-
\end{array}\right..$$

\end{proposition}
\begin{proof}
The statement about the Euler characteristic follows from the fact that the Euler characteristic of an $\I_0^*$ fibre is 6 and that the Euler characteristics of the fibres in $\Sigma^-$ are exactly 6 higher than their un-starred counterparts in $\Sigma^+$. 
\end{proof}

\begin{lemma}
Suppose that $\mathcal{E}\to \mathbf{P}^1_t$ is the elliptic surface given by the minimal (in the sense of \cite{schuett_elliptic_2010}) Weierstrass model:
$$Y^2=4X^3-g_2(t)X-g_3(t),$$
so that $\deg g_i\leq 2i\cdot d$, where $d$ is the arithmetic genus of $\mathcal{E}$.
Suppose further that $\mathcal{E}$ has no $\I_0^*$ fibres. 
If $d>1$, then $\frac{dX}{Y}\wedge dt$ is a global holomorphic $2$-form on $\mathcal{E}$. If $t_i\in\Sigma$ is a singular fibre, then the $2$-form $$\frac{dX}{\sqrt{(t-t_i)(t-a)}Y}\wedge dt$$
is holomorphic on $\mathcal{E}^{t_i}_a$; the same is true if $d=1$, i.e., $\mathcal{E}$ is rational and $t_i\in\Sigma^+$.
\end{lemma}
\begin{proof}
First, note that if $\mathcal{E}$ is given by such a Weierstrass model, then the quadratic twist is obtained by scaling the invariants as follows:
$$(g_2(t),g_3(t))\mapsto (g_2(t)(t-t_i)^2(t-a)^2,g_3(t)(t-t_i)^3(t-a)^3),$$
or just
$$(g_2(t)(t-a)^2,g_3(t)(t-a)^3)$$ if $t_i=\infty$.
It follows that the standard holomorphic $1$-form on the elliptic curves is given by
$$\omega=\frac{dX}{\sqrt{(t-t_i)(t-a)}Y}\ \textrm{or}\ \frac{dX}{\sqrt{t-a}Y}.$$

To show that $\omega\wedge dt$ is holomorphic on the total space, one must perform a local analysis such as what is carried out in \cite{stiller_picard_1987}. 
The assumption that the model is minimal places bounds on the characteristic exponents appearing in the Picard-Fuchs equations and allows one to carry this out; further details are left to the reader. 
\end{proof}

\begin{proposition}
Let $\mathcal{E}\to \mathbf{P}^1_t$ be an elliptic surface given by a minimal Weierstrass model with no $\I_0^*$ fibres, and let $t_i\in\Sigma$ denote a singular fibre; if $\mathcal{E}$ is rational, we assume further that $t_i\in\Sigma^+$. 
Denote by $\mathcal{L}$ the Picard-Fuchs operator associated to $\mathcal{E}$. 
The Picard-Fuchs equation of $\mathcal{E}^{t_i}$ is the middle Hadamard product, resp. middle convolution, of $\mathcal{L}$ with $\sqrt{1-t}$, resp. $\sqrt{t}$.
\end{proposition}
\begin{proof}
Let $[\gamma_a,\gamma_{t_i}]$ be the Pochammer contour that encircle $t_i$ and $a\in\mathbf{P}^1-\Sigma$, and let $f(t)$ be a period function of $\mathcal{E}$. 
If $t_i\neq\infty$, then the function
$$H^{t_i}_{\frac{1}{2}}(f)=\int_{[\gamma_a,\gamma_{t_i}]} \frac{f(t)}{\sqrt{(t-t_i)(t-a)}}dt$$
is annihilated by the middle Hadamard product of $\mathcal{L}$ and $\sqrt{1-t}$. 
We now show that this function is a period of the elliptic surface family. 

By the preceding lemma, $\frac{\omega}{\sqrt{(t-t_i)(t-a)}}\wedge dt$ is a holomorphic $2$-form on the surface $\mathcal{E}^{t_i}$. 
It suffices to show that the the integrals defined above correspond to integrating $\omega\wedge dt$ across a $2$-cycle.
The period function $f(t)$ itself as an integral of $\omega$ over some $t$-varying $1$-cycle $\zeta$. 
If we drag $\zeta$ along the Pochammer contour $[\gamma_a,\gamma_{t_i}]$, we obtain a $2$ chain that necessarily closes since the monodromy at the $\I_0^*$ fibre is equal to $-1$.
Thus
$$H^{t_i}_{\frac{1}{2}}(f)=\int_{[\gamma_a,\gamma_{t_i}]}\frac{f(t)}{\sqrt{(t-t_i)(t-a)}}dt=\int_{[\gamma_a,\gamma_{t_i}]}\int_\zeta\frac{\omega}{\sqrt{(t-t_i)(t-a)}}\wedge dt$$
is the integral of the holomorphic $2$-form across a $2$-cycle.
Finally, the fact that the Picard-Fuchs operator is equal to the the middle Hadamard products follows from the fact that both operators are irreducible and annihilate the same period. 

If $t_i=\infty$, the proof is analogous.
\end{proof}

\begin{corollary}
Except in the case that $\mathcal{E}$ is a rational elliptic surface and $t_i\in\Sigma^-$, the parabolic cohomology local system $\mathcal{W}$ of the associated geometric variation of local systems $\mathcal{V}^{t_i}$ is equal to the local system of transcendental lattices $\mathbf{T}$. 
In particular, the transcendental rank is determined by the fibre types of the original elliptic surface $\mathcal{E}$. 
\end{corollary}
\begin{proof}
This follows immediately from the Propositions above; the hypotheses guarantee that we do not have a rational elliptic surface after twisting. 
\end{proof}

We can use the Dettweiler-Wewers algorithm to compute the parabolic cohomology of these twisted families explicitly. 
Choose a homeomorphism $\mathbf{P}^1\to\mathbf{P}^1$ for which $t_1,\dots t_n$ are mapped to $x_1<x_2<\dots x_{n-1}<x_n=\infty$ on the real axis, and let $x_0<x_1$ be a base point. 
Let $\gamma_i$ denote the loop that travels to $x_i$ in the upper half-plane and encircles $x_i$ positively; denote the pre-image of this loop by $\gamma_i$ as well. 
With respect to this basis of loops for $\pi_1(U_0,t_0)$, the braiding map for the GVLS is determined by 
$$\ph(\gamma_1)=\beta_1^2,\ \ph(\gamma_2)=\beta_1^{-1}\beta_2^2\beta_1,\dots, \ph(\gamma_{n-1})=\beta_{1}^{-1}\cdots\beta_{n-2}^{-1}\beta_{n-1}^2\beta_{n-2}\cdots\beta_1.$$
See Figure \ref{fig:middle_convolution} for reference.
We apply this algorithm to the totality of twists arising from the 38 isolated rational elliptic surfaces with four singular fibres on the Herfurtner list and summarize the results below.

\begin{figure}
    \centering
    \includegraphics[width=10cm]{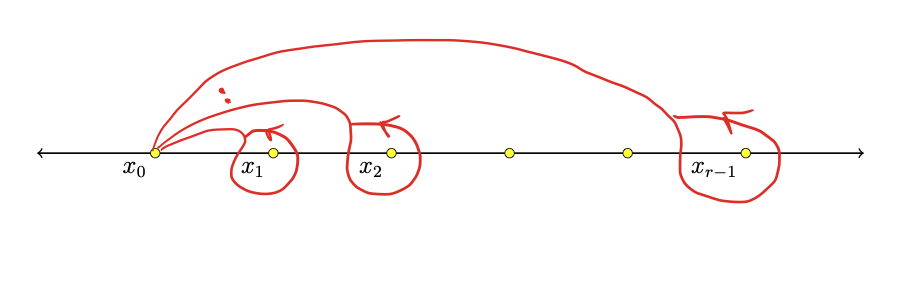}
    \caption{The loops $\gamma_i$ on $\mathbf{P}^1-\{x_1,\dots, x_r=\infty\}$ that are used to compute the homological invariant. The braiding map can be determined from this figure.}
    \label{fig:middle_convolution}
\end{figure}

\begin{corollary}\label{cor:midd_con}
Let $\mathcal{E}$ be one of the 38 isolated elliptic surfaces with four singular fibres tabulated in the Herfurnter list \cite{herfurtner_elliptic_1991}. 
Then the twists $\mathcal{E}^{t_i}_a$ for $i=1,2,3,4$ are all K3 surfaces. 
The local systems $\mathbf{T}$ of transcendental lattices and Picard-Fuchs equations for these twists are tabulated in Tables \ref{tab:PF_0_additive}---\ref{tab:LocalSystems_3_I4} in the Appendix. 
In particular, there are $35$ families of rank $19$ K3 surfaces that arise from this construction---$6\cdot 4=24$ coming from the $6$ surfaces with $4$ multiplicative fibres and $11$ coming from the $11$ surfaces with $3$ multiplicative fibres. 
Each of the $24$ families arising from surfaces with $4$ multiplicative fibres has unipotent monodromy at the twisted fibre; they cover a modular curve with at least one cusp.
Each of the $11$ families arising from surfaces with $3$ multiplicative fibres has all local monodromies of finite order; they cover a Shimura curve.
\end{corollary}

\begin{remark}
The rank 19 families described in Corollary \ref{cor:midd_con} were previously considered in \cite{laza_universal_2013} and their Picard-Fuchs equations were computed in \cite{besser_picard-fuchs_2012} by different means. 
It is remarked in \cite{laza_universal_2013} that the rank 19 families corresponding to split quaternion algebras correspond precisely to the rational elliptic surfaces with four multiplicative fibres, an observation for which the authors were unable to give an explanation. 
Our work offers such an explanation: the middle convolution of the homological invariant associated to an elliptic surface with only multiplicative fibres will necessarily have a point of infinite order monodromy corresponding to the twisted multiplicative fibre, which means the corresponding Shimura curve will have cusps. 
In contrast, the other rank 19 families appear by twisting at non-multiplicative fibres and result in finite order local monodromies, which means the associated Shimura curve will not have cusps and therefore the corresponding quaternion algebra will be non-split.
\end{remark}

\begin{remark}
In addition to the geometric setting described in this section, the Hadamard product plays a prominent role in period computations in other settings, such as the computations considered in \cite{doran_calabi-yau_2015} and \cite{doran_doran-harder-thompson_2019}.
\end{remark}

\begin{example}
As an illustration of these techniques, we construct a family of K3 surfaces with generic Picard-rank $10$ for which we have complete control over the underlying $\mathbf{Z}$-VHS. 
Let $\mathcal{E}$ be the following rational elliptic surface:
$$\mathcal{E}:\ Y^2=4X^3+4t^4X+4t$$
which was considered earlier.
The surface $\mathcal{E}$ has ten $\I_1$ fibres located at the roots of $4t^{10}+27$ and a type $\textrm{II}$ fibre located at $t=0$. 
The Mordell-Weil rank is equal to $8$, which is the maximum rank that any rational elliptic surface can have. 
The Picard-Fuchs equation for $\mathcal{E}$ is 
\begin{equation}
    \frac{d^2f}{dt^2}+\frac{32t^{10}-54}{t(4t^{10}+27)}\frac{df}{dt}+\frac{96t^{10}-57}{4t^2(4t^{10}+27)}f=0.
\end{equation}
Note that $t=\infty$ is a singularity of the Picard-Fuchs operator, but is not a singular fibre---it is an \emph{apparent singularity}.
The family of twists $\mathcal{E}^0_a$ are K3 surfaces and the $\mathbf{Z}$-VHS of the underlying transcendental lattices, which is equal to the parabolic cohomology local system, has rank $10$. 
The Picard-Fuchs differential equations corresponding to the twisted family of K3 surfaces is obtained by applying the Hadamard product:
\begin{multline}
      \frac{d^{10}f}{da^{10}}+ \frac{10(38a^{10} + 27)}{a(4a^{10} + 27)}\frac{d^{9}f}{da^{9}}
+\frac{15(3876a^{10} + 25)}{4a^2(4a^{10} + 27)}\frac{d^{8}f}{da^{8}}
+\frac{75(3876a^{10} - 5)}{a^3(4a^{10} + 27)}\frac{d^{7}f}{da^{7}}\\
+\frac{225(29393a^{10} + 10)}{2a^4(4a^{10} + 27)}\frac{d^{6}f}{da^{6}}
+\frac{45(969969a^{10} - 100)}{2a^5(4a^{10} + 27)}\frac{d^{5}f}{da^{5}}
+\frac{225(2909907a^{10} + 80)}{8a^6(4a^{10} + 27)}\frac{d^{4}f}{da^{4}}\\
+\frac{654729075a^3}{4(4a^{10} + 27)}\frac{d^{3}f}{da^{3}}
+\frac{9820936125a^2}{64(4a^{10} + 27)}\frac{d^{2}f}{da^{2}}
+\frac{3273645375a}{64(4a^{10} + 27)}\frac{df}{da}
+\frac{654729075}{256(4a^{10} + 27)}f=0
\end{multline}
With respect to a basis of parabolic cocylcles, the intersection matrix and global monodromy representation is calculated to be:
\begin{equation}
\resizebox{\linewidth}{!}{$
Q=\left(\begin{array}{rrrrrrrrrr}
0 & 0 & 1 & 0 & -1 & 0 & -1 & 0 & 0 & 0 \\
0 & 0 & -1 & 1 & 1 & 1 & 1 & -1 & 0 & 0 \\
1 & -1 & 0 & 0 & -1 & 0 & -1 & 0 & 0 & 0 \\
0 & 1 & 0 & 0 & 1 & 1 & 1 & -1 & 0 & 0 \\
-1 & 1 & -1 & 1 & 0 & 0 & 1 & 0 & 0 & 0 \\
0 & 1 & 0 & 1 & 0 & 0 & 1 & -1 & 0 & 0 \\
-1 & 1 & -1 & 1 & 1 & 1 & 0 & 0 & 0 & 0 \\
0 & -1 & 0 & -1 & 0 & -1 & 0 & 0 & 0 & 0 \\
0 & 0 & 0 & 0 & 0 & 0 & 0 & 0 & -2 & -1 \\
0 & 0 & 0 & 0 & 0 & 0 & 0 & 0 & -1 & -2
\end{array}\right),\ \det(Q)=12,\ \textrm{disc}(Q)=\mathbf{Z}/2\mathbf{Z}\oplus\mathbf{Z}/6\mathbf{Z};
$}
\end{equation}
The monodromy representation with respect to this basis is:
\begin{multline*}
\resizebox{\linewidth}{!}{$
   M_1= \left(\begin{array}{rrrrrrrrrr}
1 & 0 & 0 & 0 & 0 & 0 & 0 & 0 & 0 & 0 \\
0 & 1 & 0 & 0 & 0 & 0 & 0 & 0 & 0 & 0 \\
0 & 0 & 1 & 0 & 0 & 0 & 0 & 0 & 0 & 0 \\
0 & 0 & 0 & 1 & 0 & 0 & 0 & 0 & 0 & 0 \\
0 & 0 & 0 & 0 & 1 & 0 & 0 & 0 & 0 & 0 \\
0 & 0 & 0 & 0 & 0 & 1 & 0 & 0 & 0 & 0 \\
0 & 0 & 0 & 0 & 0 & 0 & 1 & 0 & 0 & 0 \\
-1 & 0 & -1 & -1 & 0 & -1 & -1 & -1 & 0 & 0 \\
0 & 0 & 0 & 0 & 0 & 0 & 0 & 0 & 1 & 0 \\
0 & 0 & 0 & 0 & 0 & 0 & 0 & 0 & 0 & 1
\end{array}\right), M_2=\left(\begin{array}{rrrrrrrrrr}
1 & 0 & 0 & 0 & 0 & 0 & 0 & 0 & 0 & 0 \\
0 & 1 & 0 & 0 & 0 & 0 & 0 & 0 & 0 & 0 \\
0 & 0 & 1 & 0 & 0 & 0 & 0 & 0 & 0 & 0 \\
0 & 0 & 0 & 1 & 0 & 0 & 0 & 0 & 0 & 0 \\
0 & 0 & 0 & 0 & 1 & 0 & 0 & 0 & 0 & 0 \\
0 & 0 & 0 & 0 & 0 & 1 & 0 & 0 & 0 & 0 \\
0 & 1 & -1 & 0 & 1 & -1 & -1 & -1 & 0 & 0 \\
0 & 0 & 0 & 0 & 0 & 0 & 0 & 1 & 0 & 0 \\
0 & 0 & 0 & 0 & 0 & 0 & 0 & 0 & 1 & 0 \\
0 & 0 & 0 & 0 & 0 & 0 & 0 & 0 & 0 & 1
\end{array}\right), 
M_3=\left(\begin{array}{rrrrrrrrrr}
1 & 0 & 0 & 0 & 0 & 0 & 0 & 0 & 0 & 0 \\
0 & 1 & 0 & 0 & 0 & 0 & 0 & 0 & 0 & 0 \\
0 & 0 & 1 & 0 & 0 & 0 & 0 & 0 & 0 & 0 \\
0 & 0 & 0 & 1 & 0 & 0 & 0 & 0 & 0 & 0 \\
0 & 0 & 0 & 0 & 1 & 0 & 0 & 0 & 0 & 0 \\
1 & 1 & 0 & 1 & 1 & -1 & -1 & -1 & 0 & 0 \\
0 & 0 & 0 & 0 & 0 & 0 & 1 & 0 & 0 & 0 \\
0 & 0 & 0 & 0 & 0 & 0 & 0 & 1 & 0 & 0 \\
0 & 0 & 0 & 0 & 0 & 0 & 0 & 0 & 1 & 0 \\
0 & 0 & 0 & 0 & 0 & 0 & 0 & 0 & 0 & 1
\end{array}\right)$}, 
\\
\resizebox{\linewidth}{!}{$
  M_4= \left(\begin{array}{rrrrrrrrrr}
1 & 0 & 0 & 0 & 0 & 0 & 0 & 0 & 0 & 0 \\
0 & 1 & 0 & 0 & 0 & 0 & 0 & 0 & 0 & 0 \\
0 & 0 & 1 & 0 & 0 & 0 & 0 & 0 & 0 & 0 \\
0 & 0 & 0 & 1 & 0 & 0 & 0 & 0 & 0 & 0 \\
-1 & 0 & -1 & -1 & -1 & 1 & 1 & 0 & 0 & 0 \\
0 & 0 & 0 & 0 & 0 & 1 & 0 & 0 & 0 & 0 \\
0 & 0 & 0 & 0 & 0 & 0 & 1 & 0 & 0 & 0 \\
0 & 0 & 0 & 0 & 0 & 0 & 0 & 1 & 0 & 0 \\
0 & 0 & 0 & 0 & 0 & 0 & 0 & 0 & 1 & 0 \\
0 & 0 & 0 & 0 & 0 & 0 & 0 & 0 & 0 & 1
\end{array}\right),  
M_5=\left(\begin{array}{rrrrrrrrrr}
1 & 0 & 0 & 0 & 0 & 0 & 0 & 0 & 0 & 0 \\
0 & 1 & 0 & 0 & 0 & 0 & 0 & 0 & 0 & 0 \\
0 & 0 & 1 & 0 & 0 & 0 & 0 & 0 & 0 & 0 \\
0 & 1 & -1 & -1 & -1 & 1 & 0 & -1 & 0 & 0 \\
0 & 0 & 0 & 0 & 1 & 0 & 0 & 0 & 0 & 0 \\
0 & 0 & 0 & 0 & 0 & 1 & 0 & 0 & 0 & 0 \\
0 & 0 & 0 & 0 & 0 & 0 & 1 & 0 & 0 & 0 \\
0 & 0 & 0 & 0 & 0 & 0 & 0 & 1 & 0 & 0 \\
0 & 0 & 0 & 0 & 0 & 0 & 0 & 0 & 1 & 0 \\
0 & 0 & 0 & 0 & 0 & 0 & 0 & 0 & 0 & 1
\end{array}\right), 
M_6=\left(\begin{array}{rrrrrrrrrr}
1 & 0 & 0 & 0 & 0 & 0 & 0 & 0 & 0 & 0 \\
0 & 1 & 0 & 0 & 0 & 0 & 0 & 0 & 0 & 0 \\
1 & 1 & -1 & -1 & -1 & 0 & -1 & -1 & 0 & 0 \\
0 & 0 & 0 & 1 & 0 & 0 & 0 & 0 & 0 & 0 \\
0 & 0 & 0 & 0 & 1 & 0 & 0 & 0 & 0 & 0 \\
0 & 0 & 0 & 0 & 0 & 1 & 0 & 0 & 0 & 0 \\
0 & 0 & 0 & 0 & 0 & 0 & 1 & 0 & 0 & 0 \\
0 & 0 & 0 & 0 & 0 & 0 & 0 & 1 & 0 & 0 \\
0 & 0 & 0 & 0 & 0 & 0 & 0 & 0 & 1 & 0 \\
0 & 0 & 0 & 0 & 0 & 0 & 0 & 0 & 0 & 1
\end{array}\right)$},
\end{multline*}
\begin{multline*}
\resizebox{\linewidth}{!}{$
  M_7= \left(\begin{array}{rrrrrrrrrr}
1 & 0 & 0 & 0 & 0 & 0 & 0 & 0 & 0 & 0 \\
-1 & -1 & 1 & 1 & 0 & 1 & 1 & 0 & 0 & 0 \\
0 & 0 & 1 & 0 & 0 & 0 & 0 & 0 & 0 & 0 \\
0 & 0 & 0 & 1 & 0 & 0 & 0 & 0 & 0 & 0 \\
0 & 0 & 0 & 0 & 1 & 0 & 0 & 0 & 0 & 0 \\
0 & 0 & 0 & 0 & 0 & 1 & 0 & 0 & 0 & 0 \\
0 & 0 & 0 & 0 & 0 & 0 & 1 & 0 & 0 & 0 \\
0 & 0 & 0 & 0 & 0 & 0 & 0 & 1 & 0 & 0 \\
0 & 0 & 0 & 0 & 0 & 0 & 0 & 0 & 1 & 0 \\
0 & 0 & 0 & 0 & 0 & 0 & 0 & 0 & 0 & 1
\end{array}\right), 
M_8=\left(\begin{array}{rrrrrrrrrr}
-1 & -1 & 1 & 0 & -1 & 1 & 0 & -1 & 0 & 0 \\
0 & 1 & 0 & 0 & 0 & 0 & 0 & 0 & 0 & 0 \\
0 & 0 & 1 & 0 & 0 & 0 & 0 & 0 & 0 & 0 \\
0 & 0 & 0 & 1 & 0 & 0 & 0 & 0 & 0 & 0 \\
0 & 0 & 0 & 0 & 1 & 0 & 0 & 0 & 0 & 0 \\
0 & 0 & 0 & 0 & 0 & 1 & 0 & 0 & 0 & 0 \\
0 & 0 & 0 & 0 & 0 & 0 & 1 & 0 & 0 & 0 \\
0 & 0 & 0 & 0 & 0 & 0 & 0 & 1 & 0 & 0 \\
0 & 0 & 0 & 0 & 0 & 0 & 0 & 0 & 1 & 0 \\
0 & 0 & 0 & 0 & 0 & 0 & 0 & 0 & 0 & 1
\end{array}\right), 
M_9=\left(\begin{array}{rrrrrrrrrr}
0 & -1 & 0 & -1 & -1 & 0 & -1 & -1 & -2 & 2 \\
1 & 2 & 0 & 1 & 1 & 0 & 1 & 1 & 2 & -2 \\
-1 & -1 & 1 & -1 & -1 & 0 & -1 & -1 & -2 & 2 \\
1 & 1 & 0 & 2 & 1 & 0 & 1 & 1 & 2 & -2 \\
1 & 1 & 0 & 1 & 2 & 0 & 1 & 1 & 2 & -2 \\
1 & 1 & 0 & 1 & 1 & 1 & 1 & 1 & 2 & -2 \\
1 & 1 & 0 & 1 & 1 & 0 & 2 & 1 & 2 & -2 \\
-1 & -1 & 0 & -1 & -1 & 0 & -1 & 0 & -2 & 2 \\
-1 & -1 & 0 & -1 & -1 & 0 & -1 & -1 & -1 & 2 \\
1 & 1 & 0 & 1 & 1 & 0 & 1 & 1 & 2 & -1
\end{array}\right)$}, 
\\
\resizebox{\linewidth}{!}{$
 M_{10}=\left(\begin{array}{rrrrrrrrrr}
1 & 0 & 0 & 0 & 0 & 0 & 0 & 0 & 0 & 0 \\
1 & 1 & 1 & 1 & 0 & 1 & 1 & 0 & 0 & -2 \\
0 & 0 & 1 & 0 & 0 & 0 & 0 & 0 & 0 & 0 \\
1 & 0 & 1 & 2 & 0 & 1 & 1 & 0 & 0 & -2 \\
0 & 0 & 0 & 0 & 1 & 0 & 0 & 0 & 0 & 0 \\
1 & 0 & 1 & 1 & 0 & 2 & 1 & 0 & 0 & -2 \\
0 & 0 & 0 & 0 & 0 & 0 & 1 & 0 & 0 & 0 \\
-1 & 0 & -1 & -1 & 0 & -1 & -1 & 1 & 0 & 2 \\
1 & 0 & 1 & 1 & 0 & 1 & 1 & 0 & 1 & -2 \\
2 & 0 & 2 & 2 & 0 & 2 & 2 & 0 & 0 & -3
\end{array}\right), 
M_0=\left(\begin{array}{rrrrrrrrrr}
3 & 1 & 1 & 2 & 1 & 1 & 2 & 1 & 2 & -2 \\
-1 & 2 & -2 & -1 & 1 & -2 & -1 & 1 & 0 & 2 \\
2 & 1 & 2 & 2 & 1 & 1 & 2 & 1 & 2 & -2 \\
-1 & 1 & -2 & 0 & 1 & -2 & -1 & 1 & 0 & 2 \\
-2 & -1 & -1 & -2 & 0 & -1 & -2 & -1 & -2 & 2 \\
-1 & 1 & -2 & -1 & 1 & -1 & -1 & 1 & 0 & 2 \\
-2 & -1 & -1 & -2 & -1 & -1 & -1 & -1 & -2 & 2 \\
1 & -1 & 2 & 1 & -1 & 2 & 1 & 0 & 0 & -2 \\
2 & 1 & 1 & 2 & 1 & 1 & 2 & 1 & 3 & -2 \\
1 & 2 & -1 & 1 & 2 & -1 & 1 & 2 & 2 & 1
\end{array}\right), 
M_\infty=\left(\begin{array}{rrrrrrrrrr}
-1 & 0 & 0 & 0 & 0 & 0 & 0 & 0 & 0 & 0 \\
0 & -1 & 0 & 0 & 0 & 0 & 0 & 0 & 0 & 0 \\
0 & 0 & -1 & 0 & 0 & 0 & 0 & 0 & 0 & 0 \\
0 & 0 & 0 & -1 & 0 & 0 & 0 & 0 & 0 & 0 \\
0 & 0 & 0 & 0 & -1 & 0 & 0 & 0 & 0 & 0 \\
0 & 0 & 0 & 0 & 0 & -1 & 0 & 0 & 0 & 0 \\
0 & 0 & 0 & 0 & 0 & 0 & -1 & 0 & 0 & 0 \\
0 & 0 & 0 & 0 & 0 & 0 & 0 & -1 & 0 & 0 \\
0 & 0 & 0 & 0 & 0 & 0 & 0 & 0 & -1 & 0 \\
0 & 0 & 0 & 0 & 0 & 0 & 0 & 0 & 0 & -1
\end{array}\right)$}.
\end{multline*}
\end{example}

\section{Lattice-Polarized Families of K3 Surfaces}\label{sec:K3surfaces}

In this section, we will describe how to use the theory of GVLSs to compute the $\mathbf{Z}$-VHS underlying a \emph{lattice-polarized family of K3 surfaces}. 
Of particular interest to us are $M_N$-polarized families of K3 surfaces. 
After explaining the method, we will use our techniques to calculate the $\mathbf{Z}$-VHS underlying the universal families $\mathcal{X}_N\to X_0(N)^+$ of such K3 surfaces classified in \cite{doran_calabi-yau_2017}.

We start by recalling the definition of a \emph{lattice-polarized family of K3 surfaces} as found in \cite{doran_families_2015}.

\begin{definition}
Let $L\subseteq\Lambda_{K3}$ be a lattice and $\pi^A\colon\mathcal{X}\to A$ be a smooth projective family of K3 surfaces over a smooth quasiprojective base $A$. 
We say that $\mathcal{X}^A$ is an \emph{$L$-polarized family of K3 surfaces} if
\begin{itemize}
    \item there is trivial local subsystem $\mathcal{L}$ of $R^2\pi_*\mathbf{Z}$ so that, for each $p\in A$, the fibre $\mathcal{L}_p\ins H^2(\mathcal{X}_p,\mathbf{Z})$ of $\mathcal{L}$ over $p$ is a primitive sublattice of $\textrm{NS}(\mathcal{X}_p)$ that is isomorphic to $L$, and
    \item there is a line bundle $\mathcal{A}$ on $\mathcal{X}^A$ whose restriction $\mathcal{A}_p$ to any fibre $\mathcal{X}_p$ is ample with first Chern class $c_1(\mathcal{A}_p)$ contained in $\mathcal{L}_p$ and primitive in $\textrm{NS}(\mathcal{X}_p)$. 
\end{itemize}
\end{definition}

Notice that the above definition is strictly stronger than requiring that each K3 surface in the family be polarized by the lattice $L$---it is important that monodromy acts trivially on $L$. 

In order to compute the weight two $\mathbf{Z}$-VHS underlying a lattice-polarized family $\mathcal{X}\to A$ of K3 surfaces using GVLSs, we need an internal fibration structure. 
To make the most use out of the assumption that we have a lattice-polarized family, we insist that we choose an internal fibration structure for which the associated local system of essential lattices is contained in the trivial local system $\mathcal{L}$. 


\begin{proposition}\label{prop:latticepolarizedstructure}
Let $\mathcal{X}\to U\times A-D$ be a family of elliptic curves for which $\mathcal{X}\to A$ is an $L$-polarized family of K3-surfaces. Let $\mathcal{V}$ be the weight one GVLS corresponding to the first cohomology of the elliptic curves. 
Suppose that, for each $p\in A$, the essential lattice of $\mathcal{X}_p$ is contained in $\mathcal{L}_p\cong L$. 
Let $\mathcal{W}$ denote the parabolic cohomology of $\mathcal{V}$, and $\mathbf{T}$ the local system of transcendental lattices. 
If $\mathbf{T}^{\pi_1(A)}=\{0\}$, then we have
$$\mathbf{T}=\mathcal{W}\cap\left(\mathcal{W}_\mathbf{Q}^{\pi_1(A)}\right)^\perp.$$
\end{proposition}

\begin{proof}
The assumptions above imply that $\mathcal{W}_\mathbf{Q}^{\pi_1(A)}$ is equal to the essential lattice tensored with $\mathbf{Q}$. The result now follows from Proposition \ref{Prop:EllipticSurfacePCohomStructure}.
\end{proof}

\begin{definition}
Let $\mathcal{X}^U\to U$ be an $M_N$-polarized family of K3 surfaces, and let $\mathcal{M}_{M_N}$ denote the compact moduli space of $M_N$-polarized K3 surfaces. 
Then, the \emph{generalized functional invariant} is the rational map $g\colon U\to \mathcal{M}_N$ defined by sending each point $p$ to the point in moduli corresponding to the fibre $\mathcal{X}_p$. 
\end{definition}

\begin{theorem*}[\cite{doran_calabi-yau_2017}]
Suppose $N\geq 2$. Let $\pi^U\colon\mathcal{X}\to U$ denote a non-isotrivial $M_N$-polarized family of K3 surfaces over a quasi-projective curve $U$, such that the N\'{e}ron-Severi group of a general fibre of $\mathcal{X}^U$ is isomorphic to $M_N$. 
Then $\pi^U\colon\mathcal{X}^U\to U$ is uniquely determined (up to isomorphism) by its generalized functional invariant map $g\colon U\to\mathcal{M}_{M_N}$. 

\end{theorem*}

The moduli spaces $\mathcal{M}_N$ can be naturally identified with the modular curves $X_0(N)^+$, which are the quotients of $X_0(N)$ by the Fricke involution. 
A consequence of the above theorem is that any $M_N$-polarized family of K3 surfaces, for $N\geq 2$, is the pull-back of a fundamental modular family $\mathcal{X}_N\to X_0(N)^+$ under the generalized functional invariant. 
There is an analogous result on the level of Hodge theory. 
Let $\mathbf{V}_N^+$ denote the weight 2 $\mathbf{Z}$-VHS on $X_0(N)^+$ underlying the transcendental lattices of the K3 surfaces in $\mathcal{X}_N$. 
Then we have:
\begin{theorem*}[\cite{doran_calabi-yau_2017}]
Assumptions and notation as in the previous theorem, the variation of Hodge structure on $\mathbf{T}(\mathcal{X}^U)$ is the pullback of $\mathbf{V}_N^+$ by the generalized functional invariant map $g\colon U\to X_0(N)^+$.
\end{theorem*}

These results reduce the classification of Calabi-Yau threefolds fibred non-isotrivially by $M_N$-polarized K3 surfaces to the classification rational maps with specified ramification profiles; the classification is carried out in \cite{doran_calabi-yau_2017}. 
To obtain non-rigid Calabi-Yau threefolds, one must insist that $N\in\{2,3,4,5,6,7,8,9,11\}$. 
The authors in \cite{doran_calabi-yau_2017} have tabulated the kinds of generalized functional invariants that are compatible with these values of $N$. 
From here, one can compute the associated VHS on the Calabi-Yau threefolds in terms of the data of the generalized functional invariants and the initial local system $\mathbf{V}_N^+$. 
For their own purposes, it was enough to know the local monodromy matrices associated to the VHS $\mathbf{V}_N^+$. 
However, if one wants a more complete understanding of the VHS underlying the threefolds, one would need the global monodromy representation and interesection matrix for $\mathbf{V}_N^+$. 

In the remainder of this section, we compute the local system data associated to $\mathbf{V}_N^+$ for $N\in\{2,3,4,5,6,7,8,9,11\}$ by starting with the universal families $\mathcal{X}_N\to X_0(N)^+$, finding an internal fibration structure, and then computing the parabolic cohomology of the corresponding GVLS. 

We parameterize the modular curve $X_0(N)^+$ with a coordinate $\lambda$ in such a way that the orbifold $X_0(N)^+$ has a point of infinite order at $\lambda=\infty$, a point of order $a$ at $\lambda=0$ and points of order $2$ at $\lambda=\lambda_1,\dots, \lambda_q$. 
These parameterizations are tabulated in Table \ref{Table:X0(N)}.
\begin{table}
\centering
\begin{tabular}{||c|c|l||}
\hline
{ $N$}
&
{ Orbifold Type}
& 
{ $\lambda_1,\dots, \lambda_q$}
\\
\hline
$2$&$(2,4,\infty)$&$256$\\
\hline
$3$&$(2,6,\infty)$&$108$\\
\hline
$4$&$(2,\infty,\infty)$&$64$\\
\hline
$5$&$(2,2,2,\infty)$&$22+10\sqrt{5},22-10\sqrt{5}$\\
\hline
$6$&$(2,2,\infty,\infty)$&$17+12\sqrt{2},17-12\sqrt{2}$\\
\hline
$7$&$(2,2,3,\infty)$&$-1,27$\\
\hline
$8$&$(2,2,\infty,\infty)$&$12+8\sqrt{2},12-8\sqrt{2}$\\
\hline
$9$&$(2,2,\infty,\infty)$&$9+6\sqrt{3},9-6\sqrt{3}$\\
\hline
$11$&$(2,2,2,2,\infty)$&Roots of $\lambda^3-20\lambda^2+56\lambda-44$\\
\hline
\end{tabular}
\caption{Parameterization of $X_0(N)^+$}\label{Table:X0(N)}

\end{table}
For these values of $N$, the authors of \cite{doran_calabi-yau_2017} have computed models for the universal families $\mathcal{X}_N$. 
We reproduce these models in Table \ref{tab:X_NFam}.

\begin{table}
    \centering
    \resizebox{\textwidth}{!}{
    \begin{tabular}{||c|c|l||}
    \hline
         $N$&Ambient Space&Family  \\
         \hline\hline
         $2$&$\mathbf{P}^3[x,y,z,w]$&$w^4+\lambda xyz(x+y+z-t)=0$ \\
         \hline
         $3$&$\mathbf{P}^1[r,s]\times\mathbf{P}^2[x,y,z]$&$s^2z^3+\lambda r(r-s)xy(z-x-y)=0$\\
         \hline
         $4$&$\prod_{i=1}^3\mathbf{P}^1[r_i,s_i]$&$s_1^2s_2^2s_3^2-\lambda r_1(s_1-r_1)r_2(s_2-r_2)r_3(s_3-r_3)=0$\\
         \hline
         $5$&$\mathbf{P}^3[x,y,z,w]$&$(w^2+xw+yw+zw+xy+xz+yz)^2-\lambda xyzw=0$\\
         \hline
         $6$&$\mathbf{P}^3[x,y,z,w]$&$(x+z+w)(x+y+z+w)(z+w)(y+z)-\lambda xyzw=0$\\
         \hline
         $7$&$\mathbf{P}^3[x,y,z,w]$&$(x+y+z+w)(y+z+w)(z+w)^2+(x+y+z+w)^2yz-\lambda xyzw=0$\\
         \hline
         $8$&$\mathbf{P}^3[x,y,z,w]$&$(x+y+z+w)(x+w)(y+w)(z+w)-\lambda xyzw=0$\\
         \hline
         $9$&$\mathbf{P}^3[x,y,z,w]$&$(x+y+z)(xw^2+xzw+xyw+xyz+zw^2+yw^2+yzw)-\lambda xyzw=0$\\
         \hline
         $11$&$\mathbf{P}^3[x,y,z,w]$&$(z+w)(x+y+w)(xy+zw)+x^2y^2+xyz^2+3xyzw-\lambda xyzw=0$\\
         \hline
    \end{tabular}}
    \caption{The Families $\mathcal{X}_N$}
    \label{tab:X_NFam}
\end{table}

In order to compute the VHS $\mathbf{V}_N^+$, we start by constructing an internal elliptic fibration on each of the families $\mathcal{X}_N$. 
For $N=2,3,4$, an internal elliptic fibration structure is given by a projection to $\mathbf{P}^1$. 
For all other $N$, we find elliptic fibration structures using the following method described in \cite{ilten_toric_2013}. 
Each quartic surface in these families contains a line (all such lines are described in \cite{ilten_toric_2013}); 
by subtracting this line from the pencil of hyperplane sections containing it, we obtain a pencil of cubic curves on the fibres;
blowing up base points and resolving singularities gives rise to the desired elliptic surfaces. 

In principle, any choice of internal elliptic fibration will do for our purposes. 
In practice, we want to choose the simplest such so that we can most easily determine the braiding map when computing the monodromy action on parabolic cohomology. 
In Table \ref{tab:X_N-InternalFibratiton}, we provide a choice of internal fibration structure in terms of its Weierstrass invariants as functions of $\lambda\in X_0(N)^+$ and $t\in\mathbf{P}^1_t$.

\begin{table}
    \centering
    \resizebox{\textwidth}{!}{
    \begin{tabular}{||c|c|l||}
    \hline
         $N$&Line/Method&Weierstrass Invariants  \\
         \hline
         $2$&Projection to $x$-line&$g_2=3 \cdot t^{3} \cdot \lambda^{3} \cdot (\lambda t - 48 t^{2} - 96 t - 48)$\\
         &&$g_3=t^{5} \cdot \lambda^{5} \cdot (-\lambda t + 72 t^{2} + 144 t + 72)$\\
         \hline
         $3$&Projection to $\mathbf{P}^1[r,s]$&$g_2=3 \cdot t^{3} \cdot \lambda^{3} \cdot (\lambda t - 24 t^{2} - 48 t - 24)$\\
         &&$g_3=t^{4} \cdot \lambda^{4} \cdot (-\lambda^{2} t^{2} + 36 \lambda t^{3} - 216 t^{4} + 72 \lambda t^{2} - 864 t^{3} + 36 \lambda t - 1296 t^{2} - 864 t - 216)$\\
         \hline
         $4$&Projection to $\mathbf{P}^1[r_1,s_1]$&$g_2=3 \cdot t^{2} \cdot \lambda^{2} \cdot (\lambda^{2} t^{2} - 16 \lambda t^{3} + 16 t^{4} - 32 \lambda t^{2} + 64 t^{3} - 16 \lambda t + 96 t^{2} + 64 t + 16)$\\
         &&$g_3=t^{3} \cdot \lambda^{3} \cdot (\lambda t - 8 t^{2} - 16 t - 8) \cdot (-\lambda^{2} t^{2} + 16 \lambda t^{3} + 8 t^{4} + 32 \lambda t^{2} + 32 t^{3} + 16 \lambda t + 48 t^{2} + 32 t + 8)$\\
         \hline
         $5$&Contains &$g_2=12\cdot t^{2} \cdot \lambda^{2} \cdot (\lambda^{2} t^{2} - 16 \lambda t^{3} + 16 t^{4} - 24 \lambda t^{2} + 48 t^{3} - 8 \lambda t + 64 t^{2} + 48 t + 16)$\\
         &$\{x=y+w=0\}$&$g_3=8\cdot t^{3} \cdot \lambda^{3} \cdot (\lambda t - 8 t^{2} - 12 t - 4) \cdot (-\lambda^{2} t^{2} + 16 \lambda t^{3} + 8 t^{4} + 24 \lambda t^{2} + 24 t^{3} + 8 \lambda t + 8 t^{2} - 24 t - 16)$\\
         \hline

         $6$&Contains &$g_2=3 \cdot t^{2} \cdot (\lambda^{4} t^{2} - 12 \lambda^{3} t^{3} + 14 \lambda^{2} t^{4} + 12 \lambda t^{5} + t^{6} - 20 \lambda^{3} t^{2} + 52 \lambda^{2} t^{3} + 52 \lambda t^{4}$\\
         &&$\ \ \ + 4 t^{5} - 8 \lambda^{3} t + 78 \lambda^{2} t^{2} + 84 \lambda t^{3} + 6 t^{4} + 56 \lambda^{2} t + 60 \lambda t^{2} + 4 t^{3} + 16 \lambda^{2} + 16 \lambda t + t^{2})$\\
        &$\{x=y+z+w=0\}$& $g_3=t^{3} \cdot \Big(\lambda^{6} t^{3} - 18 \lambda^{5} t^{4} + 75 \lambda^{4} t^{5} + 75 \lambda^{2} t^{7} + 18 \lambda t^{8} + t^{9} - 30 \lambda^{5} t^{3} + 258 \lambda^{4} t^{4}$\\
        &&\ \ \ $- 60 \lambda^{3} t^{5} + 480 \lambda^{2} t^{6} + 114 \lambda t^{7} + 6 t^{8} - 12 \lambda^{5} t^{2} + 339 \lambda^{4} t^{3} - 324 \lambda^{3} t^{4} + 1290 \lambda^{2} t^{5} + 300 \lambda t^{6} + 15 t^{7} $\\
        &&\ \ \ $+ 204 \lambda^{4} t^{2} - 676 \lambda^{3} t^{3}+ 1860 \lambda^{2} t^{4} + 420 \lambda t^{5} + 20 t^{6} + 48 \lambda^{4} t - 684 \lambda^{3} t^{2} + 1515 \lambda^{2} t^{3} + 330 \lambda t^{4} + 15 t^{5}$\\
        &&\ \ \ $- 336 \lambda^{3} t + 660 \lambda^{2} t^{2} + 138 \lambda t^{3} + 6 t^{4} - 64 \lambda^{3} + 120 \lambda^{2} t + 24 \lambda t^{2} + t^{3}\Big)$\\
\hline

$7$&Contains&$g_2=3\cdot t^{2} \cdot \lambda^{2} \cdot (\lambda^{2} t^{2} - 8 \lambda t^{3} + 16 t^{4} - 16 \lambda t^{2} + 40 t^{3} - 8 \lambda t + 48 t^{2} + 40 t + 16)$\\
&$\{y=z+w=0\}$&$g_3=\lambda^{2} \cdot t^{3} \cdot \Big(-\lambda^{4} t^{3} + 12 \lambda^{3} t^{4} - 48 \lambda^{2} t^{5} + 64 \lambda t^{6} + 24 \lambda^{3} t^{3} - 156 \lambda^{2} t^{4} + 240 \lambda t^{5} + 12 \lambda^{3} t^{2} $ \\
&&$\ \ \ - 216 \lambda^{2} t^{3} +384 \lambda t^{4} - 216 t^{5} - 156 \lambda^{2} t^{2} + 416 \lambda t^{3}$\\
&&$\ \ \ - 864 t^{4} - 48 \lambda^{2} t + 384 \lambda t^{2} - 1296 t^{3} + 240 \lambda t - 864 t^{2} + 64 \lambda - 216 t\Big)$\\
\hline

$8$&Contains&$g_2=3\cdot (\lambda^{4} t^{4} - 12 \lambda^{3} t^{5} + 14 \lambda^{2} t^{6} + 12 \lambda t^{7} + t^{8} - 16 \lambda^{3} t^{4} + 48 \lambda^{2} t^{5} + 32 \lambda t^{6}$\\
&&$\ \ \ - 4 \lambda^{3} t^{3} + 60 \lambda^{2} t^{4} + 20 \lambda t^{5} - 4 t^{6} + 32 \lambda^{2} t^{3} - 16 \lambda t^{4} + 6 \lambda^{2} t^{2} - 28 \lambda t^{3} + 6 t^{4} - 16 \lambda t^{2} - 4 \lambda t - 4 t^{2} + 1)$\\
&$\{x=y+z+w=0\}$&$g_3=\lambda^{6} t^{6} - 18 \lambda^{5} t^{7} + 75 \lambda^{4} t^{8} + 75 \lambda^{2} t^{10} + 18 \lambda t^{11} + t^{12} - 24 \lambda^{5} t^{6} + 216 \lambda^{4} t^{7} - 120 \lambda^{3} t^{8} + 360 \lambda^{2} t^{9}  $\\
&&$\ \ \ + 48 \lambda t^{10}- 6 \lambda^{5} t^{5} + 222 \lambda^{4} t^{6} - 444 \lambda^{3} t^{7} + 720 \lambda^{2} t^{8} - 6 \lambda t^{9} - 6 t^{10} +96 \lambda^{4} t^{5} - 632 \lambda^{3} t^{6}+816 \lambda^{2} t^{7}$\\
&&$\ \ \ - 120 \lambda t^{8} + 15 \lambda^{4} t^{4} - 432 \lambda^{3} t^{5} + 654 \lambda^{2} t^{6} - 84 \lambda t^{7} + 15 t^{8} - 144 \lambda^{3} t^{4} + 456 \lambda^{2} t^{5} + 72 \lambda t^{6} - 20 \lambda^{3} t^{3}$\\
&&$\ \ \  + 264 \lambda^{2} t^{4} + 108 \lambda t^{5} - 20 t^{6} + 96 \lambda^{2} t^{3} + 24 \lambda t^{4} + 15 \lambda^{2} t^{2} - 30 \lambda t^{3} + 15 t^{4} - 24 \lambda t^{2} - 6 \lambda t - 6 t^{2} + 1$\\
\hline

$9$&Contains&$g_2=3 \cdot (-\lambda + t + 3) \cdot t^{3} \cdot (-\lambda^{3} t + 3 \lambda^{2} t^{2} - 3 \lambda t^{3} + t^{4} + 9 \lambda^{2} t - 18 \lambda t^{2} + 9 t^{3} + 24 \lambda^{2} - 51 \lambda t + 27 t^{2} + 3 t)$\\
&$\{x+y+z=w=0\}$&$g_3=t^{4} \cdot (\lambda^{6} t^{2} - 6 \lambda^{5} t^{3} + 15 \lambda^{4} t^{4} - 20 \lambda^{3} t^{5} + 15 \lambda^{2} t^{6} - 6 \lambda t^{7} + t^{8} - 18 \lambda^{5} t^{2}$\\
&&$\ \ \ + 90 \lambda^{4} t^{3} - 180 \lambda^{3} t^{4} + 180 \lambda^{2} t^{5} - 90 \lambda t^{6} +18 t^{7} - 36 \lambda^{5} t + 279 \lambda^{4} t^{2} - 756 \lambda^{3} t^{3} +954 \lambda^{2} t^{4} - 576 \lambda t^{5} +$\\
&&$\ \ \  135 t^{6} + 324 \lambda^{4} t - 1476 \lambda^{3} t^{2} + 2484 \lambda^{2} t^{3} - 1836 \lambda t^{4} + 504 t^{5} + 216 \lambda^{4} - 1404 \lambda^{3} t$\\
&&$\ \ \ + 3051 \lambda^{2} t^{2} - 2754 \lambda t^{3} + 891 t^{4} + 540 \lambda^{2} t - 1026 \lambda t^{2} + 486 t^{3} - 27 t^{2})$\\
\hline

$11$&Contains&$g_2=3\cdot (\lambda^{4} t^{4} - 8 \lambda^{3} t^{5} + 16 \lambda t^{7} + 16 t^{8} - 16 \lambda^{3} t^{4} + 48 \lambda^{2} t^{5} + 48 \lambda t^{6} + 32 t^{7} - 4 \lambda^{3} t^{3}+ 88 \lambda^{2} t^{4} - 16 \lambda t^{5}$\\
&&$\ \ \  + 40 \lambda^{2} t^{3} - 112 \lambda t^{4} - 32 t^{5} + 6 \lambda^{2} t^{2} - 88 \lambda t^{3} - 16 t^{4} - 32 \lambda t^{2} + 16 t^{3} - 4 \lambda t + 24 t^{2} + 8 t + 1)$\\
&$\{x+z=w=0\}$&$g_3=(-\lambda^{2} t^{2} + 4 \lambda t^{3} + 8 t^{4} + 8 \lambda t^{2} + 8 t^{3} + 2 \lambda t - 4 t^{2} - 4 t - 1) \cdot (-\lambda^{4} t^{4} + 8 \lambda^{3} t^{5} + 8 \lambda t^{7} + 8 t^{8} +$\\
&&$\ \ \  16 \lambda^{3} t^{4} - 48 \lambda^{2} t^{5} + 24 \lambda t^{6} + 16 t^{7} + 4 \lambda^{3} t^{3} - 88 \lambda^{2} t^{4} + 88 \lambda t^{5} - 40 \lambda^{2} t^{3}$\\
&&$\ \ \ + 136 \lambda t^{4} - 16 t^{5} - 6 \lambda^{2} t^{2} + 88 \lambda t^{3} - 8 t^{4} + 32 \lambda t^{2} - 16 t^{3} + 4 \lambda t - 24 t^{2} - 8 t - 1)$\\
\hline
    \end{tabular}}
    \caption{Internal Fibration Structures on $\mathcal{X}_N$}
    \label{tab:X_N-InternalFibratiton}
\end{table}

\begin{proposition}\label{prop:VNVHS_Calc}
Let $\mathcal{X}_N\to X_0(N)^+$ be the universal family of $M_N$-polarized K3 surfaces, and let $\mathbf{V}_N^+$ denote the weight $2$ $\mathbf{Z}$-VHS corresponding the essential part of $R^2\pi_*\mathbf{Z}$. 
Then, the Picard-Fuchs equation for $\mathbf{V}_N^+$ is found in Table \ref{tab:PFEquations_VN}, and the global monodromy representation is tabulated in Table \ref{tab:VN_Monodromy}. 
The Picard-Fuchs equations are normalized in such a way that the characteristic exponents at $\lambda=\infty$ are all equal to $0$ and the characteristic exponents at $\lambda=\lambda_1,\dots, \lambda_q$ are equal to $0,\frac{1}{2},1$. 
The monodromy representations are chosen with respect to the standard loops on the $q+2$-punctured sphere with a base point chosen in the upper half-plane and with respect to an integral basis of parabolic cocyles for which the intersection pairing is given by
$$Q=\left(\begin{array}{rrr}2N&0&0\\
0&0&1\\
0&1&0\end{array}\right).$$
\end{proposition}

\begin{table}
    \centering
    \resizebox{\textwidth}{!}{
    \begin{tabular}{||c|l||}
    \hline
        $N$ & Picard-Fuchs Equation  \\
        \hline
        \hline
        $2$ &$\frac{d^3f}{d\lambda^3}+\frac{3 \lambda - 384}{\lambda^{2} - 256 \lambda}\frac{d^2f}{d\lambda^2}+\frac{\lambda - 48}{\lambda^{3} - 256 \lambda^{2}}\frac{df}{d\lambda}+ \frac{24}{\lambda^{4} - 256 \lambda^{3}}f=0$\\
        \hline
        $3$&$\frac{d^3f}{d\lambda^3}+\frac{3 \lambda - 162}{\lambda^{2} - 108 \lambda}\frac{d^2f}{d\lambda^2}+ \frac{\lambda - 24}{\lambda^{3} - 108 \lambda^{2}}\frac{df}{d\lambda}+ \frac{12}{\lambda^{4} - 108 \lambda^{3}}f=0$\\
        \hline
        $4$&$\frac{d^3f}{d\lambda^3}+\frac{3 \lambda - 96}{\lambda^{2} - 64 \lambda}\frac{d^2f}{d\lambda^2}+\frac{\lambda - 16}{\lambda^{3} - 64 \lambda^{2}}\frac{df}{d\lambda}+ \frac{8}{\lambda^{4} - 64 \lambda^{3}}f=0$\\
        \hline
        $5$&$\frac{d^3}{d\lambda^3}+\frac{3 \lambda - 66}{\lambda^{2} - 44 \lambda - 16}\frac{d^2f}{d\lambda^2}+\frac{\lambda^{2} - 12 \lambda - 12}{\lambda^{4} - 44 \lambda^{3} - 16 \lambda^{2}}\frac{df}{d\lambda}+\frac{6 \lambda + 12}{\lambda^{5} - 44 \lambda^{4} - 16 \lambda^{3}}f=0$\\
        \hline
        $6$&$\frac{d^3f}{d\lambda^3}+\frac{3 \lambda - 51}{\lambda^{2} - 34 \lambda + 1}\frac{d^2f}{d\lambda^2}+\frac{\lambda^{2} - 10 \lambda + 1}{\lambda^{4} - 34 \lambda^{3} + \lambda^{2}}\frac{df}{d\lambda}+ \frac{5 \lambda - 1}{\lambda^{5} - 34 \lambda^{4} + \lambda^{3}}f=0$\\
        \hline
        $7$&$\frac{d^3f}{d\lambda^3}+\frac{3 \lambda - 39}{\lambda^{2} - 26 \lambda - 27}\frac{d^2f}{d\lambda^2}+\frac{\lambda^{2} - 8 \lambda - 24}{\lambda^{4} - 26 \lambda^{3} - 27 \lambda^{2}}\frac{df}{d\lambda}+\frac{4 \lambda + 24}{\lambda^{5} - 26 \lambda^{4} - 27 \lambda^{3}}f=0$\\
        \hline
        $8$&$\frac{df^3}{d\lambda^3}+\frac{3 \lambda - 36}{\lambda^{2} - 24 \lambda + 16}\frac{d^2f}{d\lambda^2}+\frac{\lambda^{2} - 8 \lambda + 16}{\lambda^{4} - 24 \lambda^{3} + 16 \lambda^{2}}\frac{df}{d\lambda}+\frac{4 \lambda - 16}{\lambda^{5} - 24 \lambda^{4} + 16 \lambda^{3}}f=0$\\
        \hline
        $9$&$\frac{d^3f}{d\lambda^3}+\frac{3 \lambda - 27}{\lambda^{2} - 18 \lambda - 27}\frac{d^2f}{d\lambda^2}+\frac{\lambda^{2} - 6 \lambda - 27}{\lambda^{4} - 18 \lambda^{3} - 27 \lambda^{2}}\frac{df}{d\lambda}+\frac{3 \lambda + 27}{\lambda^{5} - 18 \lambda^{4} - 27 \lambda^{3}}f=0$\\
        \hline
        $11$&$\frac{d^3f}{d\lambda^3}+\frac{3 \lambda^{3} - 30 \lambda^{2} + 66}{\lambda^{4} - 20 \lambda^{3} + 56 \lambda^{2} - 44 \lambda}\frac{d^2f}{d\lambda^2}+\frac{\lambda^{3} - 8 \lambda^{2} + 64 \lambda - 132}{\lambda^{5} - 20 \lambda^{4} + 56 \lambda^{3} - 44 \lambda^{2}}\frac{df}{d\lambda}+ \frac{4 \lambda^{2} - 64 \lambda + 132}{\lambda^{6} - 20 \lambda^{5} + 56 \lambda^{4} - 44 \lambda^{3}}f=0$\\
        \hline
    \end{tabular}}
    \caption{Picard-Fuchs Equations for $\mathbf{V}_N^+$ over $X_0(N)^+$.}
    \label{tab:PFEquations_VN}
\end{table}

\begin{table}
    \centering
     \resizebox{\textwidth}{!}{
    \begin{tabular}{||c|l||}
    \hline
        $N$ & Monodromy  \\
        \hline\hline
        $2$ &$\left(\begin{array}{rrr}
1 & 0 & 4 \\
0 & 0 & 1 \\
-1 & 1 & -2
\end{array}\right), \left(\begin{array}{rrr}
1 & 0 & 0 \\
0 & 0 & 1 \\
0 & 1 & 0
\end{array}\right), \left(\begin{array}{rrr}
1 & -4 & 0 \\
0 & 1 & 0 \\
1 & -2 & 1
\end{array}\right)$\\
\hline
$3$&$\left(\begin{array}{rrr}
1 & 0 & 6 \\
0 & 0 & 1 \\
-1 & 1 & -3
\end{array}\right), \left(\begin{array}{rrr}
1 & 0 & 0 \\
0 & 0 & 1 \\
0 & 1 & 0
\end{array}\right), \left(\begin{array}{rrr}
1 & -6 & 0 \\
0 & 1 & 0 \\
1 & -3 & 1
\end{array}\right)$\\
\hline
$4$&$\left(\begin{array}{rrr}
1 & 0 & 8 \\
0 & 0 & 1 \\
-1 & 1 & -4
\end{array}\right), \left(\begin{array}{rrr}
1 & 0 & 0 \\
0 & 0 & 1 \\
0 & 1 & 0
\end{array}\right), \left(\begin{array}{rrr}
1 & -8 & 0 \\
0 & 1 & 0 \\
1 & -4 & 1
\end{array}\right)$\\
\hline
$5$&$\left(\begin{array}{rrr}
11 & -20 & 30 \\
3 & -5 & 9 \\
-2 & 4 & -5
\end{array}\right), \left(\begin{array}{rrr}
-9 & 20 & -20 \\
-2 & 4 & -5 \\
2 & -5 & 4
\end{array}\right), \left(\begin{array}{rrr}
1 & 0 & 0 \\
0 & 0 & 1 \\
0 & 1 & 0
\end{array}\right), \left(\begin{array}{rrr}
1 & 0 & 10 \\
-1 & 1 & -5 \\
0 & 0 & 1
\end{array}\right)$\\
\hline
$6$&$\left(\begin{array}{rrr}
1 & 0 & 12 \\
-1 & 1 & -6 \\
0 & 0 & 1
\end{array}\right), \left(\begin{array}{rrr}
1 & 0 & 0 \\
0 & 0 & 1 \\
0 & 1 & 0
\end{array}\right), \left(\begin{array}{rrr}
49 & 120 & -120 \\
-10 & -24 & 25 \\
10 & 25 & -24
\end{array}\right), \left(\begin{array}{rrr}
-71 & -120 & 252 \\
15 & 25 & -54 \\
-14 & -24 & 49
\end{array}\right)$\\
\hline
$7$&$\left(\begin{array}{rrr}
15 & 28 & -56 \\
-4 & -7 & 16 \\
2 & 4 & -7
\end{array}\right), \left(\begin{array}{rrr}
-13 & -42 & 28 \\
3 & 9 & -7 \\
-2 & -7 & 4
\end{array}\right), \left(\begin{array}{rrr}
1 & 0 & 0 \\
0 & 0 & 1 \\
0 & 1 & 0
\end{array}\right), \left(\begin{array}{rrr}
1 & 0 & -14 \\
1 & 1 & -7 \\
0 & 0 & 1
\end{array}\right)$\\
\hline
$8$&$\left(\begin{array}{rrr}
-31 & 80 & -96 \\
-10 & 25 & -32 \\
3 & -8 & 9
\end{array}\right), \left(\begin{array}{rrr}
17 & -48 & 48 \\
3 & -8 & 9 \\
-3 & 9 & -8
\end{array}\right), \left(\begin{array}{rrr}
1 & 0 & 0 \\
0 & 0 & 1 \\
0 & 1 & 0
\end{array}\right), \left(\begin{array}{rrr}
1 & 0 & 16 \\
-1 & 1 & -8 \\
0 & 0 & 1
\end{array}\right)$\\
\hline
$9$&$\left(\begin{array}{rrr}
19 & -36 & 90 \\
5 & -9 & 25 \\
-2 & 4 & -9
\end{array}\right), \left(\begin{array}{rrr}
-17 & 72 & -36 \\
-4 & 16 & -9 \\
2 & -9 & 4
\end{array}\right), \left(\begin{array}{rrr}
1 & 0 & 0 \\
0 & 0 & 1 \\
0 & 1 & 0
\end{array}\right), \left(\begin{array}{rrr}
1 & 0 & 18 \\
-1 & 1 & -9 \\
0 & 0 & 1
\end{array}\right)$\\
\hline
$11$&$\left(\begin{array}{rrr}
23 & -44 & 132 \\
6 & -11 & 36 \\
-2 & 4 & -11
\end{array}\right), \left(\begin{array}{rrr}
23 & -66 & 88 \\
4 & -11 & 16 \\
-3 & 9 & -11
\end{array}\right), \left(\begin{array}{rrr}
1 & 0 & 0 \\
0 & 0 & 1 \\
0 & 1 & 0
\end{array}\right), \left(\begin{array}{rrr}
23 & 66 & -88 \\
-4 & -11 & 16 \\
3 & 9 & -11
\end{array}\right), \left(\begin{array}{rrr}
1 & 0 & 22 \\
-1 & 1 & -11 \\
0 & 0 & 1
\end{array}\right)$\\
\hline
    \end{tabular}}
    \caption{Monodromy Representation for $\mathbf{V}_N^+$}
    \label{tab:VN_Monodromy}
\end{table}

\begin{proof}
The computation of the Picard-Fuchs equations for these families is done using the Griffiths-Dwork algorithm, and a complete description of how this calculation is performed is found in \cite{clingher_normal_2007}. 
Normalizing the characteristic exponents as described in the theorem statement leads to the particularly nice forms of the Picard-Fuchs equations found in Table \ref{tab:PFEquations_VN} and is our main motivation in making such a choice. 

To calculate the local system data, we compute the parabolic cohomology local system coming from the internal elliptic fibration structures described in Table \ref{tab:X_N-InternalFibratiton}.
One starts by choosing a base point for the variation of local system and computes the corresponding initial homological invariant. 
This pins down the fibre of the parabolic cohomology local system, as well as the intersection matrix for a basis of parabolic cocycles. 
Using Mathematica, we draw plots depicting the motion of the singular fibres in the internal fibration as the deformation parameter $\lambda\in X_0(N)^+$ varies; in turn, we write down the braiding map and, therefore, determine the (projective) monodromy representation of the parabolic cohomology local system. 

For each $N$, $\mathcal{X}_N$ is an $M_N$-polarized family of K3 surfaces and the essential lattices corresponding to the internal fibration structure are contained in $M_N$. 
Since the Hodge structures $\mathbf{V}_N^+$ are irreducible, Proposition \ref{prop:latticepolarizedstructure} implies that
 $$\mathbf{V}_N^+=\mathcal{W}_\mathbf{Z}\cap \bigcap_{p\in\{0,\lambda_1,\dots, \lambda_q,\infty\}}\ker_\mathbf{Q}(N_p-1),$$
where $N_p$ denotes the monodromy of $\mathcal{W}$ at the point $p$. 

Given the projective monodromy transformations of $\mathcal{W}$, we use the Picard-Fuchs equation to determine the honest monodromy transformations, i.e., eliminate the sign ambiguity. 
For each $N\in\{2,3,4,5,6,7,8,9,11\}$, the monodromy matrices around $\lambda=\lambda_1,\dots, \lambda_q$ are conjugate to 
$$\pm[-1]\oplus[1]\oplus\cdots\oplus[1].$$
Since the characteristic exponents of the Picard-Fuchs equations at these points are $0,\frac{1}{2},1$, it follows that the trace of these monodromy transformations (on parabolic cohomology) is equal to $1+r$, where $r$ denotes the Mordell-Weil rank. 
Similarly, the trace of monodromy at $\lambda=\infty$ is equal to $3+r$ and since, in all cases, the monodromy at $\infty$ is conjugate to 
$$\pm \left(\begin{array}{rrr}1&1&0\\
0&1&1\\
0&0&1\end{array}\right)\oplus[1]\oplus\cdots\oplus[1],$$
we can similarly get rid of the sign ambiguity. 
Finally, since the signs of the monodromy transformations at $\lambda=\lambda_1,\dots, \lambda_q,\infty$ are determined, the sign of the monodromy transformation at $\lambda=0$ is also determined.

Having computed the global monodromy representation for $\mathcal{W}$, we obtain the global monodromy representation for $\mathbf{V}_N^+$ by restricting the monodromy transformations to $\mathbf{V}_N^+$. 
Finally, by playing around with the intersection form, one changes bases appropriately to obtain one for which the intersection pairing is as in the theorem statement.
 
\end{proof}

\begin{example}
As an illustration, we supply some details in the case $N=5$. 
The modular curve $X_0(5)^+$ has $2$-orbifold points at $\lambda=0,22\pm 10\sqrt{5}$, and a cusp at $\lambda=\infty$; the family $\mathcal{X}_5$ is cut out by 
$$(w^2+xw+yw+zw+xy+xz+yz)^2-\lambda xyzw=0.$$
The family $\mathcal{X}_5$ contains the line $x=y+w=0$ and, as explained in \cite{ilten_toric_2013}, we can use this to construct an internal elliptic fibration. 
In terms of Weierstrass invariants, one obtains:
\begin{align*}
    g_2&=\resizebox{.7\textwidth}{!}{$12\cdot t^{2} \cdot \lambda^{2} \cdot (\lambda^{2} t^{2} - 16 \lambda t^{3} + 16 t^{4} - 24 \lambda t^{2} + 48 t^{3} - 8 \lambda t + 64 t^{2} + 48 t + 16)$}\\
         g_3&=\resizebox{.8\textwidth}{!}{$8\cdot t^{3} \cdot \lambda^{3} \cdot (\lambda t - 8 t^{2} - 12 t - 4) \cdot (-\lambda^{2} t^{2} + 16 \lambda t^{3} + 8 t^{4} + 24 \lambda t^{2} + 24 t^{3} + 8 \lambda t + 8 t^{2} - 24 t - 16)$}
\end{align*}
The discriminant of this fibration is 
\begin{equation*}
    \Delta_5=-4096\cdot\lambda^6(t + 1)^6t^8(16\lambda t^3 + (-\lambda^2 + 24\lambda - 16)t^2 + (8\lambda - 32)t - 16).
\end{equation*}
The singular fibre at $t=0$ is of type $\I_2^*$; the fibre at $t=-1$ is of type $\I_6$; the fibre at $t=\infty$ is of type $\I_1^*$; and the fibres located at the roots $p_1,p_2,p_3$ of the cubic in the above discriminant are of type $\I_1$. 
On $X_0(5)^+-\{0,22\pm 10\sqrt{5},\infty\}$, there are no fibre collisions and we obtain a geometric variation of local system on the $6$-configuration $\left(\mathbf{P}^1_t\times X_0(5)^+-\{0,22\pm 10\sqrt{5},\infty\},\Delta_5\right)$ corresponding the variation of Hodge structure describing the elliptic curve family. 
We use $\lambda=-1$ as a basepoint for $X_0(5)^+-\{0,22\pm 10\sqrt{5},\infty\}$ and choose some point in the second quadrant as a basepoint for $\mathbf{P}^1_t$.
With respect to these basepoints and a basis of loops $\gamma_1,\dots, \gamma_5$ as depicted in Figure \ref{fig:N=5Example}, the initial homological invariant is given by
\begin{equation*}
    \resizebox{.95\textwidth}{!}{$M_{p_1}=\left(\begin{array}{rr}
2 & 1 \\
-1 & 0
\end{array}\right),\ M_{p_2}=
\left(\begin{array}{rr}
0 & 1 \\
-1 & 2
\end{array}\right),\ M_{p_3}=
\left(\begin{array}{rr}
-2 & 9 \\
-1 & 4
\end{array}\right),\ M_{-1}=
\left(\begin{array}{rr}
1 & 6 \\
0 & 1
\end{array}\right),\ M_{0}=
\left(\begin{array}{rr}
-5 & -8 \\
2 & 3
\end{array}\right),\ M_{\infty}=
\left(\begin{array}{rr}
-2 & -1 \\
1 & 0
\end{array}\right)$}.
\end{equation*}
Next, let $\sigma_{-},\sigma_0,\sigma_{+}$ denote loops in the upper half plane that encircle $22-10\sqrt{5},0,22+10\sqrt{5}$ respectively. 
As $\lambda$ varies, the positions of $p_1,p_2,p_3$ vary according to the motions depicted in Figure \ref{fig:N=5Example}. 
The braiding map $\ph$ is determined by these figures and we have
\begin{eqnarray*}
\ph(\sigma_-)&=&\beta_1\beta_2^{-1}\beta_1^{-1}\\
\ph(\sigma_0)&=&\beta_2^{-1}\beta_1^{-1}\cdot (\beta_2\beta_3)^3\beta_2\beta_1^{-1}\beta_2^{-1}\beta_3^{-1}\beta_{4}^{-2}\beta_3^{-1}\beta_2^{-1}\beta_1^{-1}\cdot \beta_1\beta_2\\
\ph(\sigma_+)&=& \left(\beta_3\beta_4\beta_1^{-1}\beta_2^{-1}\beta_3\beta_4\beta_{1}^{-1}\beta_2^{-1}\beta_3^{-1}\beta_4^{-1}\right)\cdot\beta_4\cdot\left(\beta_3\beta_4\beta_1^{-1}\beta_2^{-1}\beta_3\beta_4\beta_{1}^{-1}\beta_2^{-1}\beta_3^{-1}\beta_4^{-1}\right)^{-1}.
\end{eqnarray*}

The parabolic cohomology of this geometric variation of local systems has rank four and these braids allow us to run the Dettweiler-Wewers algorithm to compute the projective monodromy represenation. 
We obtain the following monodromy representation and intersection pairing:
\begin{equation*}
    \resizebox{\textwidth}{!}{$N_-=\left(\begin{array}{rrrr}
1 & 0 & 0 & 0 \\
0 & 1 & 0 & 0 \\
0 & 0 & 1 & 0 \\
0 & -1 & 1 & -1
\end{array}\right),\ N_0=
\left(\begin{array}{rrrr}
-3 & 0 & 0 & 4 \\
0 & -1 & 0 & 0 \\
-3 & -1 & 1 & 3 \\
-2 & 0 & 0 & 3
\end{array}\right),\ N_+=
\left(\begin{array}{rrrr}
-3 & -4 & -8 & 20 \\
4 & 5 & 8 & -20 \\
-1 & -1 & -1 & 5 \\
0 & 0 & 0 & 1
\end{array}\right),\ N_\infty=
\left(\begin{array}{rrrr}
-7 & -12 & 16 & -24 \\
4 & 7 & -12 & 20 \\
-4 & -6 & 7 & -8 \\
-2 & -3 & 3 & -3
\end{array}\right),\ Q=\left(\begin{array}{rrrr}
12 & -4 & 4 & 4 \\
-4 & 2 & -2 & -2 \\
4 & -2 & -4 & -1 \\
4 & -2 & -1 & 0
\end{array}\right)$.}
\end{equation*}
The intersection of the $1$-eigenspaces of these matrices is $1$-dimensional and necessarily corresponds to the algebraic piece of parabolic cohomology since we are working with an $M_5$-polarized family. 
The intersection of the $\mathbf{Q}$-orthogonal complement of this $1$-dimensional space with parabolic cohomology is the local system of transcendental lattices. 
After changing bases appropriately, we obtain the data found in Table \ref{tab:VN_Monodromy}.

Finally, as a check on the sign of the monodromies, we use the Picard-Fuchs equation
\begin{equation*}\frac{d^3}{d\lambda^3}+\frac{3 \lambda - 66}{\lambda^{2} - 44 \lambda - 16}\frac{d^2f}{d\lambda^2}+\frac{\lambda^{2} - 12 \lambda - 12}{\lambda^{4} - 44 \lambda^{3} - 16 \lambda^{2}}\frac{df}{d\lambda}+\frac{6 \lambda + 12}{\lambda^{5} - 44 \lambda^{4} - 16 \lambda^{3}}f=0.
\end{equation*}
By calculating the characteristic exponents of this equation, we see that the monodromy at $22\pm 10\sqrt{5}$ must have trace equal to $1$ and the monodromy at $\infty$ has trace $3$. 
Since the monodromy matrices calculated above have this property, we have the correct signs. 
\end{example}

\begin{figure}%
    \centering
    \subfloat[Basis of loops.]{{\includegraphics[width=6cm]{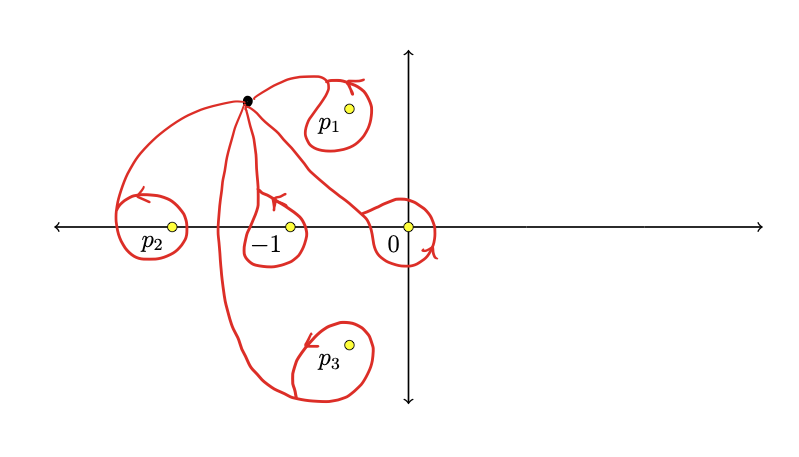} }}%
    \qquad
    \subfloat[The braid $\sigma_-$.]{{\includegraphics[width=6cm]{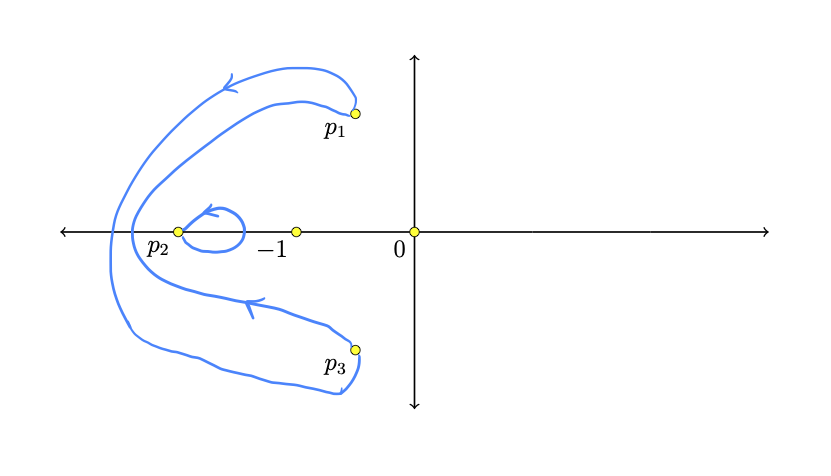} }}%
    \qquad
    \subfloat[The braid $\sigma_0$.]{{\includegraphics[width=6cm]{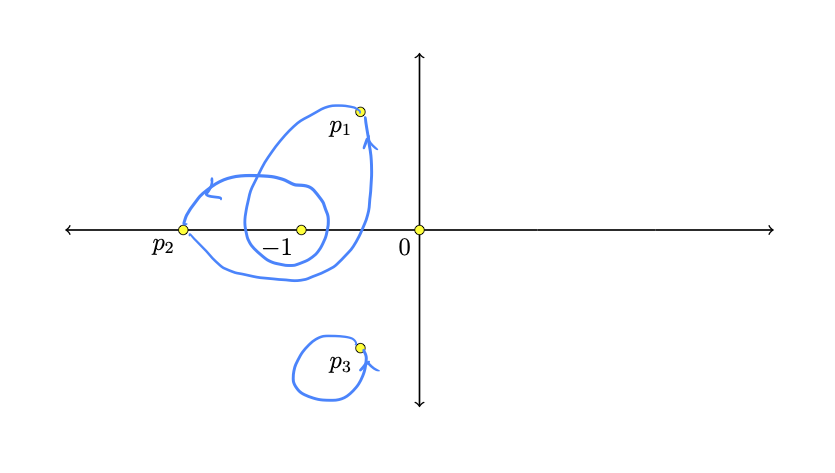} }}%
    \qquad
    \subfloat[The braid $\sigma_+$.]{{\includegraphics[width=6cm]{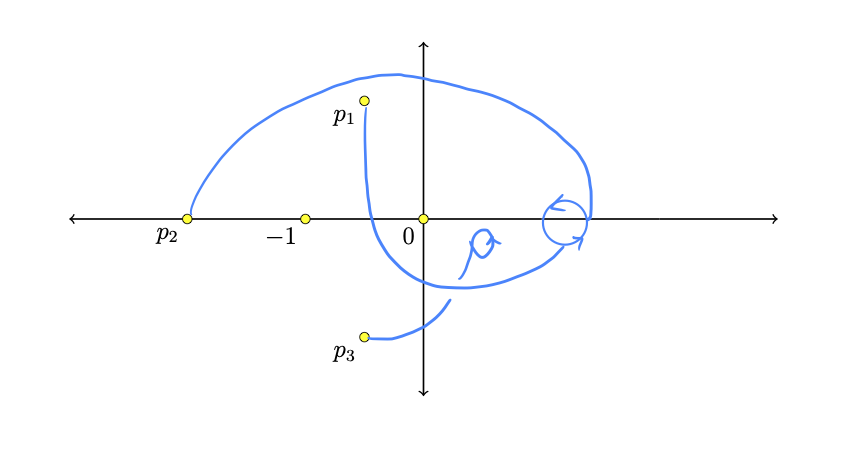} }}%
    \caption{The basis of loops we use on $X_0(5)^+$ is depicted on the top left corner. The motion of the singular fibres as $\lambda$ varies along the loops $\sigma_-,\sigma_0,\sigma_+$ is depicted in the other three figures. }%
    \label{fig:N=5Example}%
\end{figure}

\section{Geometric Isomonodromic Deformations Revisited}\label{sec:conclusion}
One of the goals of the first author's work in \cite{doran_algebraic_2001} was to construct solutions to the non-linear second order differential equation known as Painlev\'{e} VI. 
More generally, a notion of \emph{geometric isomonodromic deformation} was introduced to capture the solutions of isomonodromic deformation equations that arose geometrically. 
In this section, we review some of the deformations and explain the connection between geometric isomonodromic deformations and geometric variations of local systems. 
For a more precise treatment of isomonodromy, the reader is referred to the original article \cite{doran_algebraic_2001} and the discussion therein of Malgrange's formalism in \cite{malgrange_sur_1982}.

Let $A$ be a connected complex manifold, which we will refer to as the \emph{deformation space}, and let
$$t_i\colon A\to\mathbf{P}^1,\ 1\leq i\leq m$$
be deformation functions, which we take to be holomorphic. 
Assume that $t_i(a)\neq t_j(a)$ for $i\neq j$ throughout $A$. 
Let $X=\mathbf{P}^1\times A$ and let $D_a$ denote the relative divisor
$$D=\bigcup_{i=1}^{m}\{(x,t_i(a))\}.$$
\begin{definition}
An \emph{isomonodromic deformation} $(E,\nabla)$ of the pair $(E^0,\nabla^0)$ with deformation space $A$, deformation functions $\{t_i(a)\}$, and base point $a_0\in A$ is given by 
\begin{itemize}
    \item a holomorphic vector bundle $E$ over $\mathbf{P}^1\times T$ of rank $n$;
    \item a flat connection $\nabla$ of $E|_{X-D}$, meromorphic over $D$, for which the restriction of $(E,\nabla)$ to $\mathbf{P}^1\times\{a_0\}$ is isomorphic to $(E^0,\nabla^0)$.
\end{itemize}
\end{definition}
Loosely speaking, an isomonodromic deformation is a family of differential equations for which the corresponding monodromy representations do not vary with the deformation parameter.
In particular, an isomonodromic deformation is a variation of local systems; conversely, any variation of complex local systems parameterized by a deformation space of the type described above gives rise to an isomonodromic deformation.
A \emph{geometric isomonodromic deformation}, as defined in \cite{doran_algebraic_2001}, is an isomonodromic deformation for which the flat connection is the Gauss-Manin connection associated to a flat family of algebraic varities. 
Thus, a geometric isomonodromic deformation gives rise a geometric variation of local systems. 

\begin{remark}
As they were defined in \cite{doran_algebraic_2001}, geometric isomonodromic deformations were necessarily algebraic. 
This means that the space $A$, the deformation functions $t_i$, and flat connection $A$ were all algebraic. 
The notion of geometric variations of local systems is more flexible and allows for analytic families. 
In particular, if $A$  parameterizes analytic families of K\"{a}hler manifolds that are not algebraic, then we still obtain an isomonodromic deformation that is ostensibly geometric, even though it will fail to be algebraic. 
\end{remark}

The main benefit of recasting isomonodromic deformations in terms of variations of local systems is the fact that the variation gives rise to a local system on the deformation space that is intrinsic to the varying family of connections.
For example, consider the situation of Painlev\'{e} VI. 
A solution to Painlev\'{e} VI is determined by an analytic covering $t\colon A\mathbf{P}^1-\{0,1,\infty\}$, together with a holomorphic map $\lambda\colon S\mathbf{P}^1$. 
The curve $A$ is called the \emph{Painlev\'{e} curve} and the solution is called algebraic if $t$ is an algebraic morphism. 
The functions $t$ and $\lambda$ determine the isomonodromic deformation of second order Fuchsian equation. 
Viewed now as a variation of local systems, the Painlev\'{e} curve is equipped with the corresponding parabolic cohomology local system. 
If we further specialize to geometric isomonodromic deformations, then we obtain the stronger result that the Painlev\'{e} curve carries a variation of Hodge structure.

As an illustration, we now analyze the five geometric isomonodromic deformations that were used in \cite{doran_algebraic_2001} to produce algebraic solutions to Painlev\'{e} VI as geometric variations of local systems. 
These solutions arose by taking the five deforming families on Herfurtner's list that do not possess $\I_0^*$ fibres.
The Weierstrass and homological invariants for these five families are tabulated in Tables \ref{tab:FiveFamWeierstrass} and \ref{tab:five_fam_hom}. 
For each of these families, we run the Dettweiler-Wewers algorithm to calculate the variation of Hodge structure on the deformation space corresponding to the geometric variation of local systems. 
For the purposes of constructing Painlev\'{e} VI solutions, the isomonodromic deformation was treated as equivalent to its projective normal form. 
In particular, the isomonodromic deformation obtained by performing a quadratic twist at two of the bad fibres was considered to produce the same solution to Painlev\'{e} VI. 
Since we wish to remain sensitive to the change in integral structure, we also compute the parabolic cohomology local systems associated to all of the allowable quadratic twists.

\begin{table}
\centering
\begin{tabular}{||c|c|l||}
\hline
{ Family}
&
{ Sing. Fibre Types}
& 
{ Weierstrass Presentation}
\\ [0.5ex]
\hline
\hline
$\mathbf{1}$&$\textrm{I}_1,\textrm{I}_1, \textrm{II},\textrm{IV}^*$&$g_2(t,a)=3(t-1)(t-a^2)^3$\\
&$(0,\infty,1,a^2)$&$g_3(t,a)=(t-1)(t-a^2)^4(t+a)$\\
\hline
\hline
$\mathbf{2}$&$\textrm{I}_1,\textrm{I}_1,\textrm{I}_2,\textrm{I}_2^*$&$g_2(t,a)=12t^2(t^2+at+1)$\\
&$(\omega_1,\omega_2,\infty,0)$&$g_3(t,a)=4t^3(2t^3+3at^2+3at+2)$\\
\hline
\hline
$\mathbf{3}$&$\textrm{I}_1,\textrm{I}_1,\textrm{I}_1,\textrm{I}_3^*$&$g_2(t,a)=12t^2(t^2+2at+1)$\\
&$(\omega_1,\omega_2,\infty,0)$&$g_3(t,a)=4t^3(2t^3+3(a^2+1)t^2+6at+2)$\\
\hline
\hline
$\mathbf{4}$&$\textrm{I}_1,\textrm{I}_1,\textrm{I}_1,\textrm{III}^*$&$g_2(t,a)=3t^3(t+a)$\\
&$(\omega_1,\omega_1,\infty,0)$&$g_3(t,a)=t^5(t+1)$\\
\hline
\hline
$\mathbf{5}$&$\textrm{I}_1,\textrm{I}_1,\textrm{I}_2,\textrm{IV}^*$&$g_2(t,a)=3t^3(t+2a)$\\
&$(\omega_1,\omega_2,\infty,0)$&$g_3(t,a)=t^4(t^2+3at+1)$\\
\hline
\end{tabular}
\caption{The five families considered by Doran.}\label{tab:FiveFamWeierstrass}

\end{table}

\begin{table}
\centering
\begin{tabular}{||c|c|c|c|c|c||}
\hline
{Name}
&
{$a_0$}
& 
{Homological Invariant}&{$Q$}
&{disc.}
&{$n$}
\\ [0.5ex]
\hline
\hline
$\mathbf{1}$&$-2$&$\left(\begin{array}{rr}
1 & 1 \\
0 & 1
\end{array}\right), \left(\begin{array}{rr}
1 & 1 \\
-1 & 0
\end{array}\right), \left(\begin{array}{rr}
0 & -1 \\
1 & -1
\end{array}\right)$&$\begin{bmatrix}-6 & -3 \\
-3 & -2\end{bmatrix}$&$3$&$1$\\
\hline
\hline
$\mathbf{2}$&$0$&$\left(\begin{array}{rr}
0 & 1 \\
-1 & 2
\end{array}\right), \left(\begin{array}{rr}
-1 & -2 \\
0 & -1
\end{array}\right), \left(\begin{array}{rr}
2 & 1 \\
-1 & 0
\end{array}\right)$&$[-2]$&$-2$&$2$\\
\hline
\hline
$\mathbf{3}$&$-3$&$\left(\begin{array}{rr}
1 & 1 \\
0 & 1
\end{array}\right), \left(\begin{array}{rr}
-1 & -3 \\
0 & -1
\end{array}\right), \left(\begin{array}{rr}
3 & 4 \\
-1 & -1
\end{array}\right)$&$[-4]$&$-4$&1\\
\hline
\hline
$\mathbf{4}$&$-1$&$\left(\begin{array}{rr}
1 & 1 \\
0 & 1
\end{array}\right), \left(\begin{array}{rr}
0 & -1 \\
1 & 0
\end{array}\right), \left(\begin{array}{rr}
1 & 1 \\
0 & 1
\end{array}\right), \left(\begin{array}{rr}
1 & 0 \\
-1 & 1
\end{array}\right)$&$[-2]$&$-2$&1\\
\hline
\hline
$\mathbf{5}$&$-1$&$\left(\begin{array}{rr}
1 & 1 \\
0 & 1
\end{array}\right), \left(\begin{array}{rr}
-1 & -1 \\
1 & 0
\end{array}\right), \left(\begin{array}{rr}
0 & 1 \\
-1 & 2
\end{array}\right)$&$[-6]$&$-6$&$1$\\
\hline
\end{tabular}
\caption{The homological invariants corresponding to the five families of Doran.}\label{tab:five_fam_hom}

\end{table}

\begin{proposition}
The parabolic cohomology local systems associated to the five deforming families of elliptic surfaces with four singular fibres are tabulated in Table \ref{tab:five_fam_pc_data}.
These local systems are all irreducible and correspond to the essential lattice or transcendental lattice local system depending on if the surfaces are rational or K3 respectively. 
The local systems with matching colour in Table \ref{tab:five_fam_pc_data} that are associated with the same family are isomorphic.
\end{proposition}

\begin{table}[]
    \centering
    \resizebox{\textwidth}{!}{
    \begin{tabular}{||c|c|c|l||}
    \hline
    Twist&Type&$Q$&Monodromy\\
    \hline\hline
         $\mathbf{1}$&Rational&$\left(\begin{array}{rr}
-6 & -3 \\
-3 & -2
\end{array}\right)$&$ \left(\begin{array}{rr}
1 & 0 \\
0 & 1
\end{array}\right), \left(\begin{array}{rr}
-1 & 3 \\
0 & 1
\end{array}\right), \left(\begin{array}{rr}
2 & -3 \\
1 & -1
\end{array}\right), \left(\begin{array}{rr}
1 & 0 \\
1 & -1
\end{array}\right)$  \\
$0,1$&K3&$\textcolor{red}{\left(\begin{array}{rrr}
-6 & 3 & 3 \\
3 & 6 & 3 \\
3 & 3 & 2
\end{array}\right)}$&$\left(\begin{array}{rrr}
2 & 3 & -3 \\
-3 & -4 & 3 \\
-2 & -2 & 1
\end{array}\right), \left(\begin{array}{rrr}
5 & 3 & 0 \\
-8 & -5 & 0 \\
-6 & -3 & -1
\end{array}\right), \left(\begin{array}{rrr}
4 & -3 & 6 \\
-6 & 4 & -9 \\
-5 & 4 & -8
\end{array}\right), \left(\begin{array}{rrr}
1 & 0 & 0 \\
-1 & -1 & 0 \\
-1 & 0 & -1
\end{array}\right)$\\
$0,a^2$&K3&$\textcolor{red}{\left(\begin{array}{rrr}
-2 & -1 & -1 \\
-1 & -2 & 0 \\
-1 & 0 & -2
\end{array}\right)}$&$\left(\begin{array}{rrr}
-1 & 1 & 1 \\
0 & 0 & 1 \\
0 & 1 & 0
\end{array}\right), \left(\begin{array}{rrr}
-1 & 1 & 0 \\
0 & 1 & 0 \\
0 & 0 & -1
\end{array}\right), \left(\begin{array}{rrr}
0 & 1 & 0 \\
1 & 0 & -1 \\
-1 & 1 & 0
\end{array}\right), \left(\begin{array}{rrr}
1 & 0 & 0 \\
1 & -1 & 0 \\
1 & 0 & -1
\end{array}\right)$\\
$0,\infty$&K3&$\textcolor{blue}{\left(\begin{array}{rrrr}
12 & 12 & 6 & 6 \\
12 & 14 & 7 & 5 \\
6 & 7 & 2 & 3 \\
6 & 5 & 3 & 2
\end{array}\right)}$&$\left(\begin{array}{rrrr}
1 & -6 & 6 & 6 \\
0 & -6 & 7 & 9 \\
0 & -5 & 6 & 7 \\
0 & -2 & 2 & 3
\end{array}\right), \left(\begin{array}{rrrr}
5 & 0 & -6 & -12 \\
6 & -3 & -3 & -10 \\
4 & -4 & 1 & -4 \\
0 & 2 & -3 & -3
\end{array}\right), \left(\begin{array}{rrrr}
5 & -6 & 6 & 0 \\
2 & -4 & 7 & 2 \\
-2 & 0 & 4 & 3 \\
4 & -3 & 0 & -2
\end{array}\right), \left(\begin{array}{rrrr}
1 & 0 & 0 & 0 \\
0 & 1 & 0 & 0 \\
0 & 1 & -1 & 0 \\
2 & -1 & 0 & -1
\end{array}\right)$\\
$1,a^2$&Rational&$\left(\begin{array}{rr}
-2 & -3 \\
-3 & -6
\end{array}\right)$&$\left(\begin{array}{rr}
1 & 0 \\
0 & 1
\end{array}\right), \left(\begin{array}{rr}
-1 & 1 \\
0 & 1
\end{array}\right), \left(\begin{array}{rr}
2 & -1 \\
3 & -1
\end{array}\right), \left(\begin{array}{rr}
1 & 0 \\
3 & -1
\end{array}\right)$\\
$1,\infty$&K3&$\textcolor{brown}{\left(\begin{array}{rrr}
12 & 12 & -6 \\
12 & 14 & -5 \\
-6 & -5 & 2
\end{array}\right)}$&$\left(\begin{array}{rrr}
-1 & 0 & 0 \\
3 & -3 & 2 \\
6 & -4 & 3
\end{array}\right), \left(\begin{array}{rrr}
11 & -6 & 6 \\
5 & -2 & 3 \\
-15 & 9 & -8
\end{array}\right), \left(\begin{array}{rrr}
-5 & 0 & -6 \\
-3 & -1 & -5 \\
6 & -1 & 6
\end{array}\right), \left(\begin{array}{rrr}
1 & 0 & 0 \\
0 & 1 & 0 \\
-2 & 1 & -1
\end{array}\right)$\\
$a^2,\infty$&K3&$\textcolor{brown}{\left(\begin{array}{rrr}
-4 & -4 & -2 \\
-4 & -6 & -3 \\
-2 & -3 & -2
\end{array}\right)}$&$\left(\begin{array}{rrr}
-1 & 0 & 0 \\
-1 & 1 & -2 \\
0 & 0 & -1
\end{array}\right), \left(\begin{array}{rrr}
1 & -2 & 2 \\
0 & -2 & 3 \\
0 & -1 & 2
\end{array}\right), \left(\begin{array}{rrr}
1 & 0 & -2 \\
2 & -1 & -1 \\
1 & -1 & 0
\end{array}\right), \left(\begin{array}{rrr}
1 & 0 & 0 \\
0 & 1 & 0 \\
0 & 1 & -1
\end{array}\right)$\\
All&K3&$\textcolor{blue}{\left(\begin{array}{rrrr}
-4 & -4 & 2 & 2 \\
-4 & -6 & 1 & 3 \\
2 & 1 & 6 & 4 \\
2 & 3 & 4 & 2
\end{array}\right)}$&$\left(\begin{array}{rrrr}
1 & 0 & 2 & -2 \\
-2 & 3 & 3 & -1 \\
2 & -2 & 2 & -3 \\
2 & -2 & 1 & -2
\end{array}\right), \left(\begin{array}{rrrr}
-5 & 4 & -2 & 4 \\
-1 & -1 & -3 & 2 \\
-2 & 4 & 3 & 0 \\
-6 & 8 & 2 & 3
\end{array}\right), \left(\begin{array}{rrrr}
-5 & 6 & -2 & 4 \\
-7 & 6 & -3 & 4 \\
8 & -7 & 4 & -5 \\
6 & -3 & 3 & -2
\end{array}\right), \left(\begin{array}{rrrr}
1 & 0 & 0 & 0 \\
0 & 1 & 0 & 0 \\
-2 & 1 & -1 & 0 \\
0 & -1 & 0 & -1
\end{array}\right)$\\
\hline\hline
$\mathbf{2}$&Rational&$\left(\begin{array}{r}
-2
\end{array}\right)$&$\left(\begin{array}{r}
-1
\end{array}\right), \left(\begin{array}{r}
1
\end{array}\right), \left(\begin{array}{r}
-1
\end{array}\right)$\\
$\omega_1,\omega_2$&K3&$\textcolor{red}{\left(\begin{array}{rrr}
-4 & -4 & 0 \\
-4 & -2 & 0 \\
0 & 0 & 4
\end{array}\right)}$&$ \left(\begin{array}{rrr}
0 & 2 & 1 \\
0 & 3 & 2 \\
1 & -2 & 0
\end{array}\right), \left(\begin{array}{rrr}
0 & 2 & -1 \\
0 & 1 & 0 \\
1 & -2 & 0
\end{array}\right), \left(\begin{array}{rrr}
1 & 0 & 0 \\
2 & -1 & 0 \\
0 & 0 & -1
\end{array}\right)$\\
$0,\infty$&Rational&$\left(\begin{array}{r}
-2
\end{array}\right)$&$\left(\begin{array}{r}
-1
\end{array}\right), \left(\begin{array}{r}
1
\end{array}\right), \left(\begin{array}{r}
-1
\end{array}\right)$\\
All&K3&$\textcolor{red}{\left(\begin{array}{rrr}
12 & 4 & -4 \\
4 & 2 & -2 \\
-4 & -2 & -2
\end{array}\right)}$&$\left(\begin{array}{rrr}
-2 & 5 & -3 \\
-1 & 2 & -1 \\
2 & -4 & 3
\end{array}\right), \left(\begin{array}{rrr}
-2 & 5 & 1 \\
-1 & 2 & 1 \\
0 & 0 & 1
\end{array}\right), \left(\begin{array}{rrr}
-1 & 0 & 0 \\
0 & -1 & 0 \\
0 & 2 & 1
\end{array}\right)$\\
\hline\hline
$\mathbf{3}$&Rational&$\left(\begin{array}{r}
-4
\end{array}\right)$&$\left(\begin{array}{r}
-1
\end{array}\right), \left(\begin{array}{r}
1
\end{array}\right), \left(\begin{array}{r}
1
\end{array}\right), \left(\begin{array}{r}
-1
\end{array}\right)$\\
$\omega_1,\omega_2$&K3&$\left(\begin{array}{rrr}
0 & 0 & -4 \\
0 & 4 & -24 \\
-4 & -24 & 124
\end{array}\right)$&$\left(\begin{array}{rrr}
1 & 0 & 0 \\
4 & 1 & 0 \\
-16 & 4 & 1
\end{array}\right), \left(\begin{array}{rrr}
2 & 3 & 1 \\
5 & 2 & 1 \\
-20 & -4 & -3
\end{array}\right), \left(\begin{array}{rrr}
10 & 1 & 1 \\
35 & 2 & 3 \\
-124 & -8 & -11
\end{array}\right), \left(\begin{array}{rrr}
1 & 0 & 0 \\
0 & 1 & 0 \\
0 & 0 & 1
\end{array}\right)$\\
$0,\infty$&Rational&$\left(\begin{array}{r}
-12
\end{array}\right)$&$\left(\begin{array}{r}
1
\end{array}\right), \left(\begin{array}{r}
1
\end{array}\right), \left(\begin{array}{r}
1
\end{array}\right), \left(\begin{array}{r}
1
\end{array}\right)$\\
All&K3&$\left(\begin{array}{rrr}
8 & -4 & -8 \\
-4 & 4 & 8 \\
-8 & 8 & 4
\end{array}\right)$&$\left(\begin{array}{rrr}
1 & 0 & 0 \\
0 & 1 & 0 \\
0 & 4 & -1
\end{array}\right), \left(\begin{array}{rrr}
-2 & -3 & 1 \\
2 & 3 & -2 \\
1 & 1 & 0
\end{array}\right), \left(\begin{array}{rrr}
-22 & -43 & 13 \\
10 & 19 & -6 \\
-7 & -15 & 4
\end{array}\right), \left(\begin{array}{rrr}
7 & 14 & -4 \\
0 & 1 & 0 \\
12 & 28 & -7
\end{array}\right)$\\
\hline\hline
$\mathbf{4}$&Rational&$\left(\begin{array}{r}
-2
\end{array}\right)$&$\left(\begin{array}{r}
-1
\end{array}\right), \left(\begin{array}{r}
-1
\end{array}\right), \left(\begin{array}{r}
1
\end{array}\right), \left(\begin{array}{r}
1
\end{array}\right), \left(\begin{array}{r}
1
\end{array}\right)$\\
$\omega_1,\omega_2$&K3&$\left(\begin{array}{rrr}
-12 & -4 & 0 \\
-4 & 2 & 2 \\
0 & 2 & 2
\end{array}\right)$&$\left(\begin{array}{rrr}
1 & 1 & 1 \\
0 & 2 & -1 \\
0 & 1 & 0
\end{array}\right), \left(\begin{array}{rrr}
1 & 0 & 0 \\
0 & 1 & 0 \\
0 & 0 & 1
\end{array}\right), \left(\begin{array}{rrr}
5 & 0 & 12 \\
0 & -1 & 0 \\
-2 & 0 & -5
\end{array}\right), \left(\begin{array}{rrr}
5 & -7 & 9 \\
0 & 0 & -1 \\
-2 & 3 & -4
\end{array}\right), \left(\begin{array}{rrr}
1 & 0 & 0 \\
0 & 1 & 0 \\
0 & 0 & 1
\end{array}\right)$\\
$0,\infty$&Rational&$\left(\begin{array}{rr}
-12 & -4 \\
-4 & -2
\end{array}\right)$&$\left(\begin{array}{rr}
1 & -4 \\
0 & -1
\end{array}\right), \left(\begin{array}{rr}
1 & 0 \\
0 & 1
\end{array}\right), \left(\begin{array}{rr}
-1 & 4 \\
0 & 1
\end{array}\right), \left(\begin{array}{rr}
1 & 0 \\
0 & 1
\end{array}\right), \left(\begin{array}{rr}
-1 & 0 \\
0 & -1
\end{array}\right)$\\
All&K3&$\left(\begin{array}{rrrr}
4 & 0 & 0 & 0 \\
0 & 4 & -4 & -4 \\
0 & -4 & -2 & 0 \\
0 & -4 & 0 & 0
\end{array}\right)$&$\left(\begin{array}{rrrr}
1 & 0 & 0 & 2 \\
2 & 1 & -2 & 3 \\
0 & 0 & 1 & 1 \\
0 & 0 & 0 & 1
\end{array}\right), \left(\begin{array}{rrrr}
9 & 4 & -8 & 16 \\
-4 & -1 & 0 & -8 \\
0 & 0 & -1 & 0 \\
-4 & -2 & 4 & -7
\end{array}\right), \left(\begin{array}{rrrr}
35 & 16 & -44 & 60 \\
-72 & -29 & 68 & -132 \\
-6 & -2 & 3 & -12 \\
-8 & -4 & 12 & -13
\end{array}\right), \left(\begin{array}{rrrr}
-3 & -4 & 4 & -6 \\
6 & 3 & 2 & 5 \\
2 & 2 & -1 & 3 \\
0 & 2 & -4 & 3
\end{array}\right), \left(\begin{array}{rrrr}
-49 & 88 & 24 & 32 \\
32 & -73 & -32 & -8 \\
36 & -64 & -17 & -24 \\
24 & -32 & 0 & -25
\end{array}\right)$\\
\hline\hline
$\mathbf{5}$&Rational&$\left(\begin{array}{r}
-6
\end{array}\right)$&$\left(\begin{array}{r}
1
\end{array}\right), \left(\begin{array}{r}
1
\end{array}\right), \left(\begin{array}{r}
1
\end{array}\right), \left(\begin{array}{r}
1
\end{array}\right), \left(\begin{array}{r}
1
\end{array}\right)$\\
$\omega_1,\omega_2$&K3&$\left(\begin{array}{rrr}
-4 & 4 & 4 \\
4 & 26 & 20 \\
4 & 20 & 16
\end{array}\right)$&$\left(\begin{array}{rrr}
3 & 4 & -4 \\
-10 & -11 & 10 \\
-8 & -8 & 7
\end{array}\right), \left(\begin{array}{rrr}
9 & -2 & 5 \\
-32 & 9 & -20 \\
-26 & 8 & -17
\end{array}\right), \left(\begin{array}{rrr}
7 & -10 & 11 \\
-12 & 17 & -20 \\
-14 & 20 & -23
\end{array}\right), \left(\begin{array}{rrr}
3 & -4 & 4 \\
2 & -3 & 2 \\
0 & 0 & -1
\end{array}\right), \left(\begin{array}{rrr}
1 & 0 & 0 \\
0 & 1 & 0 \\
0 & 0 & 1
\end{array}\right)$\\
$0,\infty$&Rational&$\left(\begin{array}{rr}
-2 & 0 \\
0 & -2
\end{array}\right)$&$\left(\begin{array}{rr}
0 & 1 \\
1 & 0
\end{array}\right), \left(\begin{array}{rr}
1 & 0 \\
0 & 1
\end{array}\right), \left(\begin{array}{rr}
1 & 0 \\
0 & 1
\end{array}\right), \left(\begin{array}{rr}
0 & -1 \\
-1 & 0
\end{array}\right), \left(\begin{array}{rr}
-1 & 0 \\
0 & -1
\end{array}\right)$\\
All&K3&$\left(\begin{array}{rrrr}
4 & 4 & -4 & 0 \\
4 & 6 & -4 & 0 \\
-4 & -4 & -2 & -2 \\
0 & 0 & -2 & -2
\end{array}\right)$&$\left(\begin{array}{rrrr}
-5 & 4 & 0 & -4 \\
-8 & 6 & 1 & -8 \\
-2 & 1 & 0 & -2 \\
-4 & 3 & 1 & -5
\end{array}\right), \left(\begin{array}{rrrr}
1 & -2 & 3 & -7 \\
4 & -3 & 2 & -2 \\
4 & 0 & -3 & 12 \\
2 & 0 & -2 & 7
\end{array}\right), \left(\begin{array}{rrrr}
1 & -6 & -1 & -7 \\
4 & -7 & 0 & -8 \\
0 & 12 & 3 & 14 \\
-2 & 4 & 0 & 5
\end{array}\right), \left(\begin{array}{rrrr}
-5 & 4 & 0 & 4 \\
-10 & 6 & -1 & 8 \\
12 & -9 & 0 & -10 \\
6 & -3 & 1 & -5
\end{array}\right), \left(\begin{array}{rrrr}
-1 & 0 & 0 & 0 \\
0 & -1 & 0 & 0 \\
0 & 0 & -1 & 0 \\
0 & 0 & 0 & -1
\end{array}\right)$\\
\hline
    \end{tabular}}
    \caption{The parabolic cohomology local systems corresponding to quadratic twists of the five families. }
    \label{tab:five_fam_pc_data}
\end{table}

\begin{example}[Family 1]
	 The elliptic surfaces in Family 1 have their singular fibres located at $0,1,a^2,\infty$ for $a\notin\{0,\pm 1,\infty\}$ and we choose $a=-2$ as a base point for the deformation space, so that the initial local system has its singular points at $t=0,1,4,\infty$. 
	 With respect to the standard basis of loops on the four-punctured sphere, the homological invariant is given by
	 \begin{equation}\resizebox{.8\textwidth}{!}{$M_0=\left(\begin{array}{rr}
1 & 1 \\
0 & 1
\end{array}\right),\ M_1= \left(\begin{array}{rr}
1 & 1 \\
-1 & 0
\end{array}\right),\ M_4= \left(\begin{array}{rr}
0 & -1 \\
1 & -1
\end{array}\right),\ M_\infty=\left(\begin{array}{rr}
1 & 1 \\
0 & 1
\end{array}\right)$}.
	 \end{equation}
On the deformation space $A=\mathbf{P}^1_{a}-\{0,\pm 1,\infty\}$, we choose as generators of the fundamental group loops $\gamma_{-1},\gamma_0,\gamma_1$ that start at $a=-2$ and travel to the singular points via arcs in the upper half-plane as shown in Figure \ref{fig:fam1braids}. 
As $a$ varies along these loops, the motion of the poles is depicted in Figure \ref{fig:fam1braids}.
\begin{figure}%
    \centering
    \subfloat[Basis of loops for Family 1.]{{\includegraphics[width=6cm]{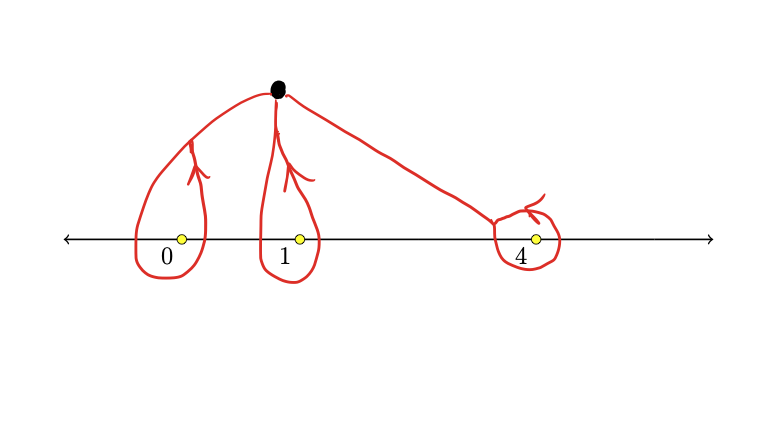} }}%
    \qquad
    \subfloat[The braid $\gamma_{-1}$.]{{\includegraphics[width=6cm]{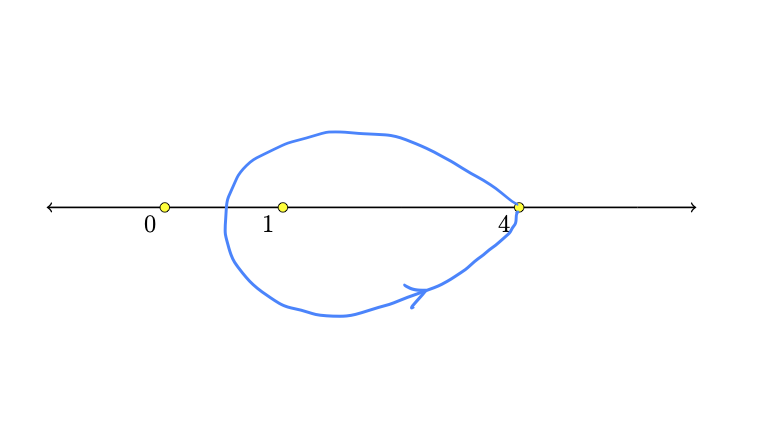} }}%
    \qquad
    \subfloat[The braid $\gamma_0$.]{{\includegraphics[width=6cm]{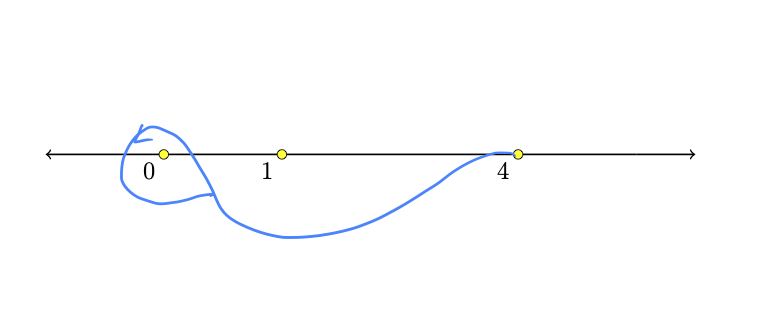} }}%
    \qquad
    \subfloat[The braid $\gamma_1$.]{{\includegraphics[width=6cm]{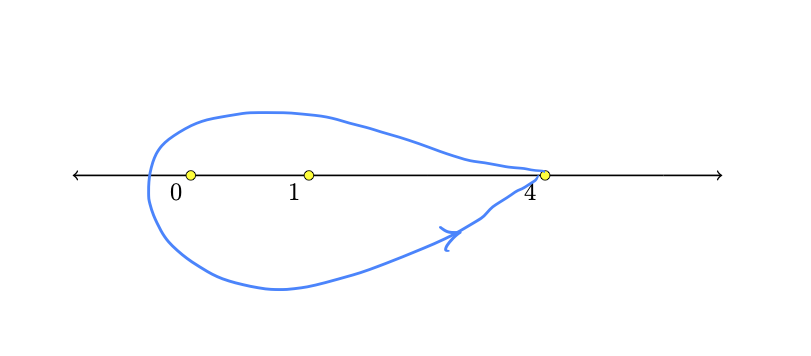} }}%
    \caption{}%
    \label{fig:fam1braids}%
\end{figure}

We determine the braiding map $\ph$ by considering these figures:
\begin{equation}
\ph(\gamma_{-1})=\beta_2^2,\ \ph(\gamma_0)=\beta_2^{-1}\beta_1^4\beta_2,\ \ph(\gamma_1)=\beta_2^{-1}\beta_1^{-2}\beta_2^2\beta_1^2\beta_2.	
\end{equation}
Applying the Dettweiler-Wewers algorithm, we calculate the monodromy representation with respect to a basis of parabolic cocycles found in the first row of Table \ref{tab:five_fam_pc_data}. 
If we perform a quadratic twist at the two $\textrm{I}_1$-fibres, we obtain a family of K3 surfaces. 
The resulting parabolic cohomology local system corresponds to the fourth row of Table \ref{tab:five_fam_pc_data}
\end{example}

Doran also examined solutions to Painlev\'{e} VI that arose by pulling back a fixed local system on a base curve $B$ via a family of rational functions.
In this manner, it was shown that certain Hurwitz curves could be realized as solutions to Painlev\'{e} VI. 
Moreover, since all of the geometric isomonodromic deformations that were considered in \cite{doran_algebraic_2001} arose by pulling back the uniformizing ODE for $\textrm{PSL}_2(\mathbf{Z})$ by families of functional invariants, these solutions are special cases of this construction. 

In the language of variations of local systems, this construction takes the following form. 
Let $\mathcal{L}$ be a local system on a genus zero base curve $B$, and let $\mathcal{H}$ denote a Hurwitz space that paramaterizes covers of $B$ with prescribed ramification profile. 
The pull-back of $\mathcal{L}$ via the family of covers parameterized by $\mathcal{H}$ yields a variation of local systems; its parabolic cohomology is a local system on $\mathcal{H}$. 
Local systems of this kind were briefly considered already in \cite{dettweiler_variation_2006} with an eye towards the inverse Galois problem.

Taking $\mathcal{L}$ to be some polarized variation of Hodge structure on the base curve $B$, then the resulting pullback family is a geometric variations of local systems and we obtain a polarized variation of Hodge structure on the Hurwitz space. 
All of the geometric variations of local systems considered in this paper are close to this form since any variation of Hodge structure underlying a family of elliptic curves is determined by the functional invariant up to quadratic twist. 
Thus, up to quadratic twists, any geometric variation of local systems underlying a family of elliptic surfaces corresponds to some Hurwitz space. 

As we described earlier, the variation of Hodge structure of any $M_N$-polarized family of K3 surfaces is similar determined by its generalized functional invariant \cite{doran_calabi-yau_2017}. 
Therefore, families of threefolds that arise as the total space of $M_N$-polarized families of K3 surfaces are described by Hurwitz covers. 
A complete description of the Hurwitz spaces describing families of Calabi-Yau threefolds that arise in this manner was classified in \cite{doran_calabi-yau_2017}, and our methods allow us to calculate the corresponding variations of Hodge structure underlying these Hurwitz spaces, with applications to the Doran-Harder-Thompson ``gluing/splitting'' mirror conjecture \cite{doran_doran-harder-thompson_2019}.
This is work that will be addressed in a future paper.

\bibliographystyle{amsxport}
\bibliography{output.bbl}{}

\section*{Appendix}
In this appendix, we tabulate the Picard-Fuchs equations and transcendental lattice local systems corresponding to the families of quadratic twists discussed in Corollary \ref{cor:midd_con}, corresponding to the 38 isolated rational elliptic surfaces with four singular fibres classified in \cite{herfurtner_elliptic_1991}. 
The tables are organized in terms of the number of additive singular fibres present in the initial rational elliptic surface starting with zero such fibres and going up to three such fibres.

\begin{table}
    \centering
    \resizebox{\textwidth}{!}{
\begin{tabular}{||c|l||}
    \hline
    Family&Picard-Fuchs Equations\\
    \hline\hline
        $\I_1,\I_1,\I_1,\I_9$ &$\frac{d^3f}{da^3} + \left(\frac{6 a^{2} + \frac{3}{2} a + \frac{3}{2}}{a^{3} - 1}\right) \frac{d^2f}{da^2} + \left(\frac{\frac{27}{4} a^{2} - \frac{7}{2} a - \frac{1}{4}}{a^{4} - a^{3} - a + 1}\right) \frac{df}{da} + \frac{\frac{3}{4} a - \frac{1}{4}}{a^{4} - a^{3} - a + 1}f=0$\\
        $1,\eta,\eta^2,\infty$&$\frac{d^3f}{da^3} + \left(\frac{6 a^{2} + \frac{3}{2} \eta a - \frac{3}{2} \eta - \frac{3}{2}}{a^{3} - 1}\right) \frac{d^2f}{da^2} + \left(\frac{\frac{27}{4} a^{2} - \frac{7}{2} \eta a + \frac{1}{4} \eta + \frac{1}{4}}{a^{4} - \eta a^{3} - a + \eta}\right) \frac{df}{da} + \frac{\frac{3}{4} a - \frac{1}{4} \eta}{a^{4} - \eta a^{3} - a + \eta}f=0$\\
        &$\frac{d^3f}{da^3} + \left(\frac{6 a^{2} + \left(-\frac{3}{2} \eta - \frac{3}{2}\right) a + \frac{3}{2} \eta}{a^{3} - 1}\right) \frac{d^2f}{da^2} + \left(\frac{\frac{27}{4} a^{2} + \left(\frac{7}{2} \eta + \frac{7}{2}\right) a - \frac{1}{4} \eta}{a^{4} + \left(\eta + 1\right) a^{3} - a - \eta - 1}\right) \frac{df}{da} + \frac{\frac{3}{4} a + \frac{1}{4} \eta + \frac{1}{4}}{a^{4} + \left(\eta + 1\right) a^{3} - a - \eta - 1}f=0$\\
        &$\frac{d^3f}{da^3} + \left(\frac{\frac{9}{2} a^{2}}{a^{3} - 1}\right) \frac{d^2f}{da^2} + \left(\frac{\frac{13}{4} a}{a^{3} - 1}\right) \frac{df}{da} + \frac{\frac{1}{8}}{a^{3} - 1}f=0$\\
        \hline
        
        $\I_1,\I_1,\I_2,\I_8$&
        $\frac{d^3f}{da^3} + \left(\frac{6 a^{2} - \frac{3}{2} a - \frac{3}{2}}{a^{3} - a}\right) \frac{d^2f}{da^2} + \left(\frac{\frac{27}{4} a^{2} + \frac{7}{2} a - \frac{5}{4}}{a^{4} + a^{3} - a^{2} - a}\right) \frac{df}{da} + \frac{\frac{3}{4} a + \frac{1}{4}}{a^{4} + a^{3} - a^{2} - a}f=0$\\
\hline

    \end{tabular}}
    \caption{Picard-Fuchs equations for twists with no additive fibres.}
    \label{tab:PF_0_additive}
\end{table}

\begin{table}
    \centering
    \resizebox{.95\linewidth}{!}{
    \begin{tabular}{||c|c|c|c|c||}
    \hline
        Family&$Q$ & $\det(Q)$&Monodromy& Orders  \\
        \hline\hline
        $\I_1,\I_1,\I_1,\I_9$ & $\left(\begin{array}{rrr}
0 & -12 & -6 \\
-12 & 196 & 116 \\
-6 & 116 & 68
\end{array}\right)$&$-144$&$\left(\begin{array}{rrr}
1 & 0 & 0 \\
-12 & 7 & -12 \\
-7 & 3 & -5
\end{array}\right), \left(\begin{array}{rrr}
7 & -12 & 21 \\
4 & -7 & 14 \\
0 & 0 & 1
\end{array}\right), \left(\begin{array}{rrr}
4 & -9 & 15 \\
-10 & 31 & -50 \\
-7 & 21 & -34
\end{array}\right), \left(\begin{array}{rrr}
1 & 6 & -9 \\
-12 & -37 & 54 \\
-8 & -24 & 35
\end{array}\right)$&$\left[+\infty, 2, 2, 2\right]$\\
&$\left(\begin{array}{rrr}
72 & 36 & 18 \\
36 & 20 & 10 \\
18 & 10 & 4
\end{array}\right)$&$-144$&$\left(\begin{array}{rrr}
1 & 0 & 0 \\
0 & 1 & 0 \\
0 & 1 & -1
\end{array}\right), \left(\begin{array}{rrr}
1 & -9 & 9 \\
2 & -2 & -3 \\
0 & -3 & 4
\end{array}\right), \left(\begin{array}{rrr}
1 & 18 & -27 \\
0 & 11 & -15 \\
0 & 8 & -11
\end{array}\right), \left(\begin{array}{rrr}
1 & -9 & 9 \\
2 & -10 & 9 \\
2 & -9 & 8
\end{array}\right)$&$\left[2, +\infty, 2, 2\right]$\\
&$\left(\begin{array}{rrr}
180 & 144 & 144 \\
144 & 116 & 114 \\
144 & 114 & 116
\end{array}\right)$ &$-144$&$\left(\begin{array}{rrr}
1 & 0 & 0 \\
0 & 1 & 0 \\
4 & -3 & -1
\end{array}\right), \left(\begin{array}{rrr}
1 & 0 & 0 \\
4 & -1 & -3 \\
0 & 0 & 1
\end{array}\right), \left(\begin{array}{rrr}
-89 & 72 & 36 \\
-80 & 64 & 33 \\
-70 & 57 & 28
\end{array}\right), \left(\begin{array}{rrr}
-89 & 72 & 36 \\
-66 & 53 & 27 \\
-88 & 72 & 35
\end{array}\right)$&$\left[2, 2, +\infty, 2\right]$\\
\hline 

    \end{tabular}}
    \caption{Transcendental lattices for twists with no additive fibre.}
    \label{tab:my_label}
\end{table}

\begin{table}
    \centering
    \resizebox{\linewidth}{!}{
    \begin{tabular}{||c|l||}
    \hline
    Family &Picard-Fuchs Equations\\
    \hline\hline
        $\I_1,\I_1,\I_8,\textrm{II}$&$\frac{d^4f}{da^4} + \left(\frac{\frac{23}{2} a^{2} + \left(\frac{5}{2} \omega - 42\right) a + 36}{a^{3} - \frac{14}{3} a^{2} + 9 a}\right) \frac{d^3f}{da^3} + \left(\frac{\frac{135}{4} a^{3} + \left(-\frac{77}{4} \omega - 89\right) a^{2} + \left(40 \omega + \frac{157}{2}\right) a - 20 \omega}{a^{5} + \left(-\omega - \frac{14}{3}\right) a^{4} + \left(\frac{14}{3} \omega + 9\right) a^{3} - 9 \omega a^{2}}\right) \frac{d^2f}{da^2}$\\
        &$+ \left(\frac{\frac{195}{8} a^{2} + \left(-9 \omega - \frac{79}{2}\right) a + \frac{15}{2} \omega + \frac{33}{2}}{a^{5} + \left(-\omega - \frac{14}{3}\right) a^{4} + \left(\frac{14}{3} \omega + 9\right) a^{3} - 9 \omega a^{2}}\right) \frac{df}{da} + \frac{\frac{15}{8} a - \frac{3}{8} \omega - \frac{9}{8}}{a^{5} + \left(-\omega - \frac{14}{3}\right) a^{4} + \left(\frac{14}{3} \omega + 9\right) a^{3} - 9 \omega a^{2}}f=0$\\
        
        $\omega_1,\omega_2,\infty,0$&$\frac{d^4f}{da^4} + \left(\frac{\frac{23}{2} a^{2} + \left(-\frac{5}{2} \omega - \frac{91}{3}\right) a + 36}{a^{3} - \frac{14}{3} a^{2} + 9 a}\right) \frac{d^3f}{da^3} + \left(\frac{\frac{135}{4} a^{3} + \left(\frac{77}{4} \omega - \frac{1073}{6}\right) a^{2} + \left(-40 \omega + \frac{1591}{6}\right) a + 20 \omega - \frac{280}{3}}{a^{5} + \left(\omega - \frac{28}{3}\right) a^{4} + \left(-\frac{14}{3} \omega + \frac{277}{9}\right) a^{3} + \left(9 \omega - 42\right) a^{2}}\right) \frac{d^2f}{da^2}$\\
\hline 
    \end{tabular}}
    \caption{Picard-Fuchs equations for twists with one additive fibre.}
    \label{tab:PF_1_additive}
\end{table}

\begin{table}
\ContinuedFloat
    \centering
    \resizebox{\linewidth}{!}{
    \begin{tabular}{||c|l||}
    \hline
    Family &Picard-Fuchs Equations\\
    \hline\hline
        $\I_1,\I_2,\I_6,\textrm{III}$&$\frac{d^4f}{da^4} + \left(\frac{\frac{23}{2} a^{2} - 35 a + 16}{a^{3} - 5 a^{2} + 4 a}\right) \frac{d^3f}{da^3} + \left(\frac{\frac{135}{4} a^{3} - \frac{689}{4} a^{2} + \frac{823}{4} a - 35}{a^{5} - 9 a^{4} + 24 a^{3} - 16 a^{2}}\right) \frac{d^2f}{da^2} + \left(\frac{\frac{195}{8} a^{2} - 78 a + \frac{309}{8}}{a^{5} - 9 a^{4} + 24 a^{3} - 16 a^{2}}\right) \frac{df}{da} + \frac{\frac{15}{8} a - \frac{21}{8}}{a^{5} - 9 a^{4} + 24 a^{3} - 16 a^{2}}f=0$\\
        $4,1,\infty,0$&$\frac{d^4f}{da^4} + \left(\frac{\frac{23}{2} a^{2} - \frac{85}{2} a + 16}{a^{3} - 5 a^{2} + 4 a}\right) \frac{d^3f}{da^3} + \left(\frac{\frac{135}{4} a^{3} - \frac{229}{2} a^{2} + \frac{155}{2} a - \frac{35}{4}}{a^{5} - 6 a^{4} + 9 a^{3} - 4 a^{2}}\right) \frac{d^2f}{da^2} + \left(\frac{\frac{195}{8} a^{2} - 51 a + 15}{a^{5} - 6 a^{4} + 9 a^{3} - 4 a^{2}}\right) \frac{df}{da} + \frac{\frac{15}{8} a - \frac{3}{2}}{a^{5} - 6 a^{4} + 9 a^{3} - 4 a^{2}}f=0$\\
        &$\frac{d^4f}{da^4} + \left(\frac{9 a^{2} - \frac{65}{2} a + 16}{a^{3} - 5 a^{2} + 4 a}\right) \frac{d^3f}{da^3} + \left(\frac{19 a^{2} - \frac{171}{4} a + \frac{35}{4}}{a^{4} - 5 a^{3} + 4 a^{2}}\right) \frac{d^2f}{da^2} + \left(\frac{\frac{33}{4} a - \frac{63}{8}}{a^{4} - 5 a^{3} + 4 a^{2}}\right) \frac{df}{da} + \frac{\frac{3}{16}}{a^{4} - 5 a^{3} + 4 a^{2}}f=0$\\
        &$\frac{d^3f}{da^3} + \left(\frac{6 a^{2} - \frac{45}{2} a + 12}{a^{3} - 5 a^{2} + 4 a}\right) \frac{d^2f}{da^2} + \left(\frac{\frac{27}{4} a^{2} - \frac{33}{2} a + \frac{15}{4}}{a^{4} - 5 a^{3} + 4 a^{2}}\right) \frac{df}{da} + \frac{\frac{3}{4}}{a^{3} - 4 a^{2}}f=0$\\
\hline

    \end{tabular}}
    \caption{Picard-Fuchs equations for twists with one additive fibre.}
    \label{tab:PF_1_additive}
\end{table}

\begin{table}
    \centering
    \resizebox{\linewidth}{!}{
    \begin{tabular}{||c|c|c|c|c||}
    \hline
        Family&$Q$ & $\det(Q)$&Monodromy& Orders  \\
        \hline\hline
$\I_1,\I_1,\I_8,\textrm{II}$&$\left(\begin{array}{rrr}
-8 & -8 & 0 \\
-8 & 4 & 6 \\
0 & 6 & 4
\end{array}\right)$&$-96$&$\left(\begin{array}{rrr}
1 & 0 & 0 \\
-1 & 2 & -3 \\
-1 & 1 & -1
\end{array}\right), \left(\begin{array}{rrr}
9 & -24 & 40 \\
5 & -14 & 25 \\
1 & -3 & 6
\end{array}\right), \left(\begin{array}{rrr}
9 & -24 & 32 \\
2 & -5 & 8 \\
-1 & 3 & -3
\end{array}\right), \left(\begin{array}{rrr}
-7 & 24 & -24 \\
0 & -1 & 0 \\
2 & -8 & 7
\end{array}\right)$&$\left[6, 2, 2, 2\right]$\\
&$\left(\begin{array}{rrrr}
56 & 24 & 16 & 48 \\
24 & 12 & 8 & 20 \\
16 & 8 & 4 & 14 \\
48 & 20 & 14 & 40
\end{array}\right)$&$128$&$\left(\begin{array}{rrrr}
1 & 0 & 0 & 0 \\
0 & 1 & 0 & 0 \\
-1 & 1 & 0 & 1 \\
2 & 0 & -1 & -1
\end{array}\right), \left(\begin{array}{rrrr}
-7 & -8 & 8 & 8 \\
2 & -1 & -4 & 0 \\
-3 & -3 & 4 & 3 \\
-7 & -7 & 7 & 8
\end{array}\right), \left(\begin{array}{rrrr}
17 & 16 & -24 & -16 \\
8 & 9 & -12 & -8 \\
8 & 8 & -11 & -8 \\
14 & 14 & -21 & -13
\end{array}\right), \left(\begin{array}{rrrr}
-7 & -8 & 8 & 8 \\
-6 & -9 & 8 & 8 \\
-6 & -8 & 7 & 8 \\
-6 & -8 & 8 & 7
\end{array}\right)$&$\left[3, +\infty, 2, 2\right]$\\
&$\left(\begin{array}{rrrr}
8 & -8 & -8 & 0 \\
-8 & 20 & 12 & 4 \\
-8 & 12 & 8 & 2 \\
0 & 4 & 2 & 0
\end{array}\right)$&$128$&$\left(\begin{array}{rrrr}
1 & 0 & 0 & 0 \\
0 & 1 & 0 & 0 \\
-1 & 0 & 0 & 1 \\
0 & 1 & -1 & -1
\end{array}\right), \left(\begin{array}{rrrr}
1 & 0 & 0 & 0 \\
-4 & 9 & -12 & -4 \\
-3 & 6 & -8 & -3 \\
-1 & 2 & -3 & 0
\end{array}\right), \left(\begin{array}{rrrr}
1 & -4 & 8 & 8 \\
0 & -5 & 8 & 4 \\
0 & -2 & 3 & 0 \\
0 & -3 & 5 & 5
\end{array}\right), \left(\begin{array}{rrrr}
1 & -4 & 8 & 8 \\
-4 & 7 & -16 & -16 \\
-2 & 4 & -9 & -8 \\
0 & 0 & 0 & -1
\end{array}\right)$&$\left[3, 2, +\infty, 2\right]$\\
\hline 

\end{tabular}}
    \caption{Transcendental lattices for twists with one additive fibre.}
    \label{tab:LocalSystems_1_additiveII}
\end{table}

\begin{table}
\ContinuedFloat
    \centering
    \resizebox{\linewidth}{!}{
    \begin{tabular}{||c|c|c|c|c||}
    \hline
        Family&$Q$ & $\det(Q)$&Monodromy& Orders  \\
        \hline\hline
$\I_1,\I_2,\I_7,\textrm{III}$&$\left(\begin{array}{rrrr}
-126 & -42 & -14 & -98 \\
-42 & -14 & -6 & -32 \\
-14 & -6 & 22 & -18 \\
-98 & -32 & -18 & -74
\end{array}\right)$&$224$&$\left(\begin{array}{rrrr}
1 & 0 & 0 & 0 \\
0 & 1 & 0 & 0 \\
6 & -8 & -1 & -4 \\
-2 & 1 & 1 & 3
\end{array}\right), \left(\begin{array}{rrrr}
1 & 0 & 0 & 0 \\
-5 & -1 & 2 & 7 \\
0 & 0 & 1 & 0 \\
0 & 0 & 0 & 1
\end{array}\right), \left(\begin{array}{rrrr}
1 & 28 & 0 & -14 \\
-3 & 10 & 1 & -1 \\
12 & 8 & -3 & -18 \\
-3 & 21 & 1 & -6
\end{array}\right), \left(\begin{array}{rrrr}
1 & 28 & 0 & -14 \\
1 & 13 & 0 & -7 \\
-2 & -28 & -1 & 14 \\
2 & 28 & 0 & -15
\end{array}\right)$&$\left[+\infty, 2, 4, 2\right]$\\
&$\left(\begin{array}{rrrr}
-126 & -56 & -98 & -84 \\
-56 & -24 & -44 & -38 \\
-98 & -44 & -74 & -64 \\
-84 & -38 & -64 & -56
\end{array}\right)$&$224$&$\left(\begin{array}{rrrr}
1 & 0 & 0 & 0 \\
0 & 1 & 0 & 0 \\
0 & 0 & 1 & 0 \\
1 & -1 & 1 & -1
\end{array}\right), \left(\begin{array}{rrrr}
1 & 0 & 0 & 0 \\
1 & -1 & -1 & 1 \\
2 & 4 & 3 & -8 \\
0 & 0 & 0 & 1
\end{array}\right), \left(\begin{array}{rrrr}
1 & 0 & -14 & 14 \\
0 & 2 & -6 & 5 \\
0 & 0 & -9 & 10 \\
0 & 1 & -8 & 8
\end{array}\right), \left(\begin{array}{rrrr}
1 & 0 & -14 & 14 \\
1 & -1 & -7 & 7 \\
2 & 0 & -15 & 14 \\
2 & 0 & -14 & 13
\end{array}\right)$&$\left[2, +\infty, 4, 2\right]$\\
&$\left(\begin{array}{rrr}
70 & 42 & 42 \\
42 & 26 & 24 \\
42 & 24 & 26
\end{array}\right)$&$-56$&$\left(\begin{array}{rrr}
1 & 0 & 0 \\
0 & 1 & 0 \\
3 & -3 & -1
\end{array}\right), \left(\begin{array}{rrr}
1 & 0 & 0 \\
3 & -1 & -3 \\
0 & 0 & 1
\end{array}\right), \left(\begin{array}{rrr}
-27 & 28 & 14 \\
-20 & 20 & 11 \\
-16 & 17 & 8
\end{array}\right), \left(\begin{array}{rrr}
-27 & 28 & 14 \\
-13 & 13 & 7 \\
-26 & 28 & 13
\end{array}\right)$&$\left[2, 2, 4, 2\right]$\\
\hline 
\end{tabular}}
    \caption{Transcendental lattices for twists with one additive fibre.}
    \label{tab:LocalSystems_1_additiveIII}
\end{table}

\begin{table}
\ContinuedFloat
    \centering
    \resizebox{\linewidth}{!}{
    \begin{tabular}{||c|c|c|c|c||}
    \hline
        Family&$Q$ & $\det(Q)$&Monodromy& Orders  \\
        \hline\hline
$\I_1,\I_1,\I_6,\textrm{IV}$&$\left(\begin{array}{rrrr}
-18 & -12 & 24 & 6 \\
-12 & -8 & 12 & 2 \\
24 & 12 & 84 & 60 \\
6 & 2 & 60 & 38
\end{array}\right)$&$288$&$\left(\begin{array}{rrrr}
1 & 0 & 0 & 0 \\
0 & 1 & 0 & 0 \\
12 & -12 & 7 & -12 \\
8 & -9 & 3 & -5
\end{array}\right), \left(\begin{array}{rrrr}
-1 & 6 & 4 & -6 \\
-3 & 5 & 1 & -1 \\
0 & 0 & 1 & 0 \\
-1 & 3 & 2 & -2
\end{array}\right), \left(\begin{array}{rrrr}
-17 & 0 & -18 & 30 \\
-15 & 1 & -15 & 25 \\
36 & 0 & 37 & -60 \\
12 & 0 & 12 & -19
\end{array}\right), \left(\begin{array}{rrrr}
-1 & 6 & 4 & -6 \\
0 & 5 & 4 & -6 \\
0 & -24 & -17 & 24 \\
0 & -12 & -8 & 11
\end{array}\right)$&$\left[+\infty, 6, 2, 2\right]$\\
&$\left(\begin{array}{rrr}
30 & 24 & 12 \\
24 & 20 & 10 \\
12 & 10 & 4
\end{array}\right)$&$-24$&$\left(\begin{array}{rrr}
1 & 0 & 0 \\
0 & 1 & 0 \\
0 & 1 & -1
\end{array}\right), \left(\begin{array}{rrr}
1 & -6 & 6 \\
2 & -5 & 3 \\
0 & -3 & 4
\end{array}\right), \left(\begin{array}{rrr}
1 & 12 & -18 \\
0 & 11 & -15 \\
0 & 8 & -11
\end{array}\right), \left(\begin{array}{rrr}
1 & -6 & 6 \\
2 & -7 & 6 \\
2 & -6 & 5
\end{array}\right)$&$\left[2, 3, 2, 2\right]$\\
&$\left(\begin{array}{rrrr}
12 & -24 & -12 & -12 \\
-24 & 18 & 6 & 6 \\
-12 & 6 & 2 & 0 \\
-12 & 6 & 0 & 2
\end{array}\right)$&$288$&$\left(\begin{array}{rrrr}
1 & 0 & 0 & 0 \\
0 & 1 & 0 & 0 \\
0 & 0 & 1 & 0 \\
1 & 3 & -3 & -1
\end{array}\right), \left(\begin{array}{rrrr}
1 & 0 & 0 & 0 \\
-4 & -7 & 6 & 6 \\
-3 & -5 & 5 & 3 \\
-2 & -4 & 3 & 4
\end{array}\right), \left(\begin{array}{rrrr}
-5 & -12 & 12 & 12 \\
0 & -5 & 0 & 18 \\
-3 & -11 & 7 & 21 \\
0 & -2 & 0 & 7
\end{array}\right), \left(\begin{array}{rrrr}
-5 & -12 & 12 & 12 \\
2 & 5 & -6 & -6 \\
0 & 0 & -1 & 0 \\
0 & 0 & 0 & -1
\end{array}\right)$&$\left[2, 6, +\infty, 2\right]$\\
\hline 
\end{tabular}}
    \caption{Transcendental lattices for twists with one additive fibre.}
    \label{tab:LocalSystems_1_additiveIV}
\end{table}

\begin{table}[]
    \centering
    \resizebox{\textwidth}{!}{
    \begin{tabular}{|c|l||}
    \hline
    Family&Picard-Fuchs Equations\\
    \hline\hline
         $\I_1,\I_7,\textrm{II},\textrm{II}$&$\frac{d^5f}{da^5} + \left(\frac{19 a^{2} - 169 a + 343}{a^{3} - 13 a^{2} + 49 a}\right) \frac{d^4f}{da^4} + \left(\frac{\frac{217}{2} a^{4} - \frac{3857}{2} a^{3} + \frac{24647}{2} a^{2} - 31115 a + 21609}{a^{6} - 26 a^{5} + 267 a^{4} - 1274 a^{3} + 2401 a^{2}}\right) \frac{d^3f}{da^3} + \left(\frac{210 a^{3} - \frac{11445}{4} a^{2} + \frac{50225}{4} a - \frac{65905}{4}}{a^{6} - 26 a^{5} + 267 a^{4} - 1274 a^{3} + 2401 a^{2}}\right) \frac{d^2f}{da^2} + \left(\frac{\frac{1785}{16} a^{2} - 1050 a + \frac{38535}{16}}{a^{6} - 26 a^{5} + 267 a^{4} - 1274 a^{3} + 2401 a^{2}}\right) \frac{df}{da} + \frac{\frac{105}{16} a - \frac{525}{16}}{a^{6} - 26 a^{5} + 267 a^{4} - 1274 a^{3} + 2401 a^{2}}f=0$\\
         
         $0,\infty,\omega_1,\omega_2$&$\frac{d^5f}{da^5} + \left(\frac{\frac{31}{2} a^{2} - \frac{247}{2} a + \frac{343}{2}}{a^{3} - 13 a^{2} + 49 a}\right) \frac{d^4f}{da^4} + \left(\frac{\frac{137}{2} a^{3} - 1090 a^{2} + 5782 a - \frac{19551}{2}}{a^{5} - 26 a^{4} + 267 a^{3} - 1274 a^{2} + 2401 a}\right) \frac{d^3f}{da^3} + \left(\frac{\frac{375}{4} a^{2} - \frac{4055}{4} a + \frac{11095}{4}}{a^{5} - 26 a^{4} + 267 a^{3} - 1274 a^{2} + 2401 a}\right) \frac{d^2f}{da^2} + \left(\frac{\frac{465}{16} a - \frac{1305}{8}}{a^{5} - 26 a^{4} + 267 a^{3} - 1274 a^{2} + 2401 a}\right) \frac{df}{da} + \frac{\frac{15}{32}}{a^{5} - 26 a^{4} + 267 a^{3} - 1274 a^{2} + 2401 a}f=0$\\
\hline 
     \end{tabular}}
    \caption{Picard-Fuchs equations for twists with two additive fibres.}
    \label{tab:PF_2_additive}
\end{table}

\begin{table}
\ContinuedFloat
    \centering
    \resizebox{\textwidth}{!}{
    \begin{tabular}{|c|l||}
    \hline
    Family&Picard-Fuchs Equations\\
    \hline\hline

         $\I_2,\I_5,\textrm{II},\textrm{III}$&$\frac{d^5f}{da^5} + \left(\frac{19 a^{2} - \frac{3957}{14} a + \frac{10125}{14}}{a^{3} - \frac{157}{7} a^{2} + \frac{3375}{28} a}\right) \frac{d^4f}{da^4} + \left(\frac{\frac{217}{2} a^{4} - \frac{90411}{28} a^{3} + \frac{12420035}{392} a^{2} - \frac{87144375}{784} a + \frac{1265625}{14}}{a^{6} - \frac{314}{7} a^{5} + \frac{72923}{98} a^{4} - \frac{529875}{98} a^{3} + \frac{11390625}{784} a^{2}}\right) \frac{d^3f}{da^3} $\\

         &$+ \left(\frac{210 a^{3} - \frac{261435}{56} a^{2} + \frac{23605415}{784} a - \frac{9946875}{196}}{a^{6} - \frac{314}{7} a^{5} + \frac{72923}{98} a^{4} - \frac{529875}{98} a^{3} + \frac{11390625}{784} a^{2}}\right) \frac{d^2f}{da^2} + \left(\frac{\frac{1785}{16} a^{2} - \frac{184125}{112} a + \frac{1010175}{196}}{a^{6} - \frac{314}{7} a^{5} + \frac{72923}{98} a^{4} - \frac{529875}{98} a^{3} + \frac{11390625}{784} a^{2}}\right) \frac{df}{da} $\\
         &$+ \frac{\frac{105}{16} a - \frac{675}{14}}{a^{6} - \frac{314}{7} a^{5} + \frac{72923}{98} a^{4} - \frac{529875}{98} a^{3} + \frac{11390625}{784} a^{2}}f=0$\\

         $\frac{125}{14},\infty,0,\frac{27}{2}$&$\frac{d^5f}{da^5} + \left(\frac{\frac{31}{2} a^{2} - \frac{6591}{28} a + \frac{10125}{14}}{a^{3} - \frac{157}{7} a^{2} + \frac{3375}{28} a}\right) \frac{d^4f}{da^4} + \left(\frac{\frac{137}{2} a^{3} - \frac{41155}{28} a^{2} + \frac{237729}{28} a - 10125}{a^{5} - \frac{503}{14} a^{4} + \frac{11853}{28} a^{3} - \frac{91125}{56} a^{2}}\right) \frac{d^3f}{da^3} +$\\
\hline 
    
    \end{tabular}}
    \caption{Picard-Fuchs equations for twists with two additive fibres.}
    \label{tab:PF_2_additive}
\end{table}

\begin{table}
\ContinuedFloat
    \centering
    \resizebox{\textwidth}{!}{
    \begin{tabular}{|c|l||}
    \hline
    Family&Picard-Fuchs Equations\\
    \hline\hline
         $\I_2,\I_4,\textrm{III},\textrm{III}$&$\frac{d^5f}{da^5} + \left(\frac{19 a^{2} - 7}{a^{3} - a}\right) \frac{d^4f}{da^4} + \left(\frac{\frac{217}{2} a^{4} - \frac{371}{4} a^{2} + 9}{a^{6} - 2 a^{4} + a^{2}}\right) \frac{d^3f}{da^3} + \left(\frac{210 a^{2} - \frac{385}{4}}{a^{5} - 2 a^{3} + a}\right) \frac{d^2f}{da^2} + \left(\frac{\frac{1785}{16} a^{2} - \frac{315}{16}}{a^{6} - 2 a^{4} + a^{2}}\right) \frac{df}{da} + \frac{\frac{105}{16}}{a^{5} - 2 a^{3} + a}f=0$\\
         
        $0,\infty,1,-1$ &$\frac{d^5f}{da^5} + \left(\frac{\frac{31}{2} a^{2} - \frac{7}{2}}{a^{3} - a}\right) \frac{d^4f}{da^4} + \left(\frac{\frac{137}{2} a^{2} - \frac{175}{4}}{a^{4} - 2 a^{2} + 1}\right) \frac{d^3f}{da^3} + \left(\frac{\frac{375}{4} a^{2} - \frac{175}{8}}{a^{5} - 2 a^{3} + a}\right) \frac{d^2f}{da^2} + \left(\frac{\frac{465}{16}}{a^{4} - 2 a^{2} + 1}\right) \frac{df}{da} + \frac{\frac{15}{32}}{a^{5} - 2 a^{3} + a}f=0$\\
        
         &$\frac{d^4f}{da^4} + \left(\frac{\frac{23}{2} a^{2} + a - \frac{5}{2}}{a^{3} - a}\right) \frac{d^3f}{da^3} + \left(\frac{\frac{135}{4} a^{3} + \frac{11}{2} a^{2} - \frac{43}{2} a - 2}{a^{5} - 2 a^{3} + a}\right) \frac{d^2f}{da^2} + \left(\frac{\frac{195}{8} a^{2} + \frac{21}{4} a - 6}{a^{5} - 2 a^{3} + a}\right) \frac{df}{da} + \frac{\frac{15}{8} a + \frac{3}{8}}{a^{5} - 2 a^{3} + a}f=0$\\
         
         &$\frac{d^4f}{da^4} + \left(\frac{\frac{23}{2} a^{2} - a - \frac{5}{2}}{a^{3} - a}\right) \frac{d^3f}{da^3} + \left(\frac{\frac{135}{4} a^{3} - \frac{11}{2} a^{2} - \frac{43}{2} a + 2}{a^{5} - 2 a^{3} + a}\right) \frac{d^2f}{da^2} + \left(\frac{\frac{195}{8} a^{2} - \frac{21}{4} a - 6}{a^{5} - 2 a^{3} + a}\right) \frac{df}{da} + \frac{\frac{15}{8} a - \frac{3}{8}}{a^{5} - 2 a^{3} + a}f=0$\\
\hline

    \end{tabular}}
    \caption{Picard-Fuchs equations for twists with two additive fibres.}
    \label{tab:PF_2_additive}
\end{table}

\begin{table}[]
    \centering
    \resizebox{\linewidth}{!}{
    \begin{tabular}{||c|c|c|c|c||}
    \hline
        Family&$Q$ & $\det(Q)$&Monodromy& Orders  \\
        \hline\hline
$\I_1,\I_7,\textrm{II},\textrm{II}$&$\left(\begin{array}{rrrrr}
-56 & -49 & -42 & 0 & 21 \\
-49 & -46 & -38 & 2 & 20 \\
-42 & -38 & -32 & 0 & 16 \\
0 & 2 & 0 & 20 & 14 \\
21 & 20 & 16 & 14 & 2
\end{array}\right)$&$-112$&$\left(\begin{array}{rrrrr}
1 & 0 & 0 & 0 & 0 \\
0 & 1 & 0 & 0 & 0 \\
0 & 0 & 1 & 0 & 0 \\
4 & 4 & -12 & 3 & -4 \\
4 & 4 & -11 & 1 & -1
\end{array}\right), \left(\begin{array}{rrrrr}
1 & 0 & 0 & 0 & 0 \\
0 & -2 & 3 & 1 & -1 \\
2 & -1 & 1 & -2 & 3 \\
0 & 0 & 0 & 1 & 0 \\
0 & 0 & 0 & 0 & 1
\end{array}\right), \left(\begin{array}{rrrrr}
-13 & -7 & 21 & 7 & -7 \\
-10 & -5 & 17 & 3 & -2 \\
-8 & -5 & 15 & 2 & -1 \\
-12 & -2 & 10 & 15 & -18 \\
-4 & 1 & 0 & 8 & -10
\end{array}\right), \left(\begin{array}{rrrrr}
-13 & -7 & 21 & 7 & -7 \\
-12 & -8 & 21 & 7 & -7 \\
-12 & -7 & 20 & 7 & -7 \\
24 & 14 & -42 & -15 & 14 \\
24 & 14 & -42 & -14 & 13
\end{array}\right)$&$\left[+\infty, 3, 3, 2\right]$\\
&$\left(\begin{array}{rrrr}
-56 & -49 & -70 & -21 \\
-49 & -42 & -59 & -19 \\
-70 & -59 & -80 & -27 \\
-21 & -19 & -27 & -8
\end{array}\right)$&$84$&$\left(\begin{array}{rrrr}
1 & 0 & 0 & 0 \\
0 & 1 & 0 & 0 \\
0 & 0 & 1 & 0 \\
3 & -4 & 1 & -1
\end{array}\right), \left(\begin{array}{rrrr}
1 & 0 & 0 & 0 \\
0 & 2 & -1 & 1 \\
1 & 3 & -1 & -3 \\
0 & 0 & 0 & 1
\end{array}\right), \left(\begin{array}{rrrr}
-13 & 21 & -7 & 7 \\
-11 & 17 & -5 & 4 \\
-13 & 18 & -4 & 2 \\
-3 & 4 & -1 & 1
\end{array}\right), \left(\begin{array}{rrrr}
-13 & 21 & -7 & 7 \\
-12 & 20 & -7 & 7 \\
-24 & 42 & -15 & 14 \\
-12 & 21 & -7 & 6
\end{array}\right)$&$\left[2, 6, 3, 2\right]$\\
\hline 
\end{tabular}}
    \caption{Transcendental lattices for twists with two additive fibres.}
    \label{tab:LocalSystems_2_II}
\end{table}

\begin{table}
\ContinuedFloat
    \centering
    \resizebox{\linewidth}{!}{
    \begin{tabular}{||c|c|c|c|c||}
    \hline
        Family&$Q$ & $\det(Q)$&Monodromy& Orders  \\
        \hline\hline

$\I_1,\I_6,\textrm{II},\textrm{III}$&$\left(\begin{array}{rrrr}
-42 & -24 & -18 & -18 \\
-24 & -14 & -10 & -10 \\
-18 & -10 & -4 & -10 \\
-18 & -10 & -10 & -6
\end{array}\right)$&$48$&$\left(\begin{array}{rrrr}
1 & 0 & 0 & 0 \\
0 & 1 & 0 & 0 \\
4 & -4 & -1 & -2 \\
-1 & 1 & 1 & 1
\end{array}\right), \left(\begin{array}{rrrr}
1 & 0 & 0 & 0 \\
-1 & -1 & 2 & 3 \\
0 & 0 & 1 & 0 \\
0 & 0 & 0 & 1
\end{array}\right), \left(\begin{array}{rrrr}
-5 & 12 & 0 & -6 \\
-4 & 6 & 1 & -1 \\
-2 & 10 & -1 & -8 \\
-3 & 3 & 1 & 1
\end{array}\right), \left(\begin{array}{rrrr}
-5 & 12 & 0 & -6 \\
-4 & 11 & 0 & -6 \\
4 & -12 & -1 & 6 \\
-4 & 12 & 0 & -7
\end{array}\right)$&$\left[4, 2, 3, 2\right]$\\
&$\left(\begin{array}{rrrrr}
-42 & -36 & -18 & -24 & -24 \\
-36 & -30 & -16 & -22 & -20 \\
-18 & -16 & -6 & -8 & -12 \\
-24 & -22 & -8 & -14 & -14 \\
-24 & -20 & -12 & -14 & -14
\end{array}\right)$&$-192$&$\left(\begin{array}{rrrrr}
1 & 0 & 0 & 0 & 0 \\
0 & 1 & 0 & 0 & 0 \\
0 & 0 & 1 & 0 & 0 \\
-2 & 0 & 2 & 1 & 2 \\
3 & -1 & -1 & -1 & -1
\end{array}\right),\left(\begin{array}{rrrrr}
1 & 0 & 0 & 0 & 0 \\
1 & -1 & -1 & 1 & 1 \\
-2 & 4 & 3 & -4 & 0 \\
0 & 0 & 0 & 1 & 0 \\
0 & 0 & 0 & 0 & 1
\end{array}\right), \left(\begin{array}{rrrrr}
-5 & 0 & -6 & 6 & 6 \\
-6 & 2 & -4 & 4 & 5 \\
-2 & 0 & -1 & 2 & 2 \\
-4 & 2 & 0 & 1 & 2 \\
-3 & -1 & -5 & 5 & 5
\end{array}\right), \left(\begin{array}{rrrrr}
-5 & 0 & -6 & 6 & 6 \\
-4 & -1 & -6 & 6 & 6 \\
-4 & 0 & -7 & 6 & 6 \\
-8 & 0 & -12 & 11 & 12 \\
0 & 0 & 0 & 0 & -1
\end{array}\right)$&$\left[4, +\infty, 3, 2\right]$\\
\hline 
\end{tabular}}
    \caption{Transcendental lattices for twists with two additive fibres.}
    \label{tab:LocalSystems_2_II_III}
\end{table}

\begin{table}
\ContinuedFloat
    \centering
    \resizebox{\linewidth}{!}{
    \begin{tabular}{||c|c|c|c|c||}
    \hline
        Family&$Q$ & $\det(Q)$&Monodromy& Orders  \\
        \hline\hline

$\I_1,\I_5,\textrm{III}$&$\left(\begin{array}{rrrrr}
-70 & -30 & -50 & 10 & 20 \\
-30 & -14 & -22 & 6 & 10 \\
-50 & -22 & -36 & 6 & 14 \\
10 & 6 & 6 & 14 & 8 \\
20 & 10 & 14 & 8 & 2
\end{array}\right)$&$-320$&$\left(\begin{array}{rrrrr}
1 & 0 & 0 & 0 & 0 \\
0 & 1 & 0 & 0 & 0 \\
0 & 0 & 1 & 0 & 0 \\
2 & 0 & -4 & 3 & -4 \\
2 & 1 & -4 & 1 & -1
\end{array}\right), \left(\begin{array}{rrrrr}
1 & 0 & 0 & 0 & 0 \\
0 & -1 & 1 & 1 & -1 \\
2 & -2 & 1 & -4 & 6 \\
0 & 0 & 0 & 1 & 0 \\
0 & 0 & 0 & 0 & 1
\end{array}\right), \left(\begin{array}{rrrrr}
-9 & 0 & 10 & 10 & -10 \\
-4 & 0 & 5 & 2 & -1 \\
-6 & -2 & 9 & 2 & 0 \\
-2 & 4 & -2 & 11 & -14 \\
0 & 3 & -3 & 6 & -8
\end{array}\right), \left(\begin{array}{rrrrr}
-9 & 0 & 10 & 10 & -10 \\
-4 & -1 & 5 & 5 & -5 \\
-8 & 0 & 9 & 10 & -10 \\
8 & 0 & -10 & -11 & 10 \\
8 & 0 & -10 & -10 & 9
\end{array}\right)$&$\left[+\infty, 4, 4, 2\right]$\\
&$\left(\begin{array}{rrrr}
-70 & -30 & -50 & -50 \\
-30 & -12 & -20 & -22 \\
-50 & -20 & -32 & -36 \\
-50 & -22 & -36 & -36
\end{array}\right)$&$80$&$\left(\begin{array}{rrrr}
1 & 0 & 0 & 0 \\
0 & 1 & 0 & 0 \\
0 & 0 & 1 & 0 \\
2 & -3 & 1 & -1
\end{array}\right), \left(\begin{array}{rrrr}
1 & 0 & 0 & 0 \\
0 & 1 & -1 & 1 \\
2 & 2 & -1 & -2 \\
0 & 0 & 0 & 1
\end{array}\right), \left(\begin{array}{rrrr}
-9 & 20 & -10 & 10 \\
-4 & 10 & -4 & 3 \\
-6 & 14 & -5 & 4 \\
-6 & 13 & -6 & 6
\end{array}\right), \left(\begin{array}{rrrr}
-9 & 20 & -10 & 10 \\
-4 & 9 & -5 & 5 \\
-8 & 20 & -11 & 10 \\
-8 & 20 & -10 & 9
\end{array}\right)$&$\left[2, 4, 4, 2\right]$\\
\hline 
\end{tabular}}
    \caption{Transcendental lattices for twists with two additive fibres.}
    \label{tab:LocalSystems_2_III_III}
\end{table}

\begin{table}
\ContinuedFloat
    \centering
    \resizebox{\linewidth}{!}{
    \begin{tabular}{||c|c|c|c|c||}
    \hline
        Family&$Q$ & $\det(Q)$&Monodromy& Orders  \\
        \hline\hline

$\I_1,\I_4,\textrm{III},\textrm{IV}$&$\left(\begin{array}{rrrrr}
-12 & -12 & -12 & 8 & 4 \\
-12 & -30 & -24 & 16 & 2 \\
-12 & -24 & -20 & 12 & 2 \\
8 & 16 & 12 & 4 & 8 \\
4 & 2 & 2 & 8 & 6
\end{array}\right)$&$-384$&$\left(\begin{array}{rrrrr}
1 & 0 & 0 & 0 & 0 \\
0 & 1 & 0 & 0 & 0 \\
0 & 0 & 1 & 0 & 0 \\
0 & 4 & -4 & 3 & -4 \\
1 & 4 & -5 & 1 & -1
\end{array}\right), \left(\begin{array}{rrrrr}
1 & 0 & 0 & 0 & 0 \\
-2 & -1 & 6 & 6 & -6 \\
0 & -3 & 5 & 2 & -1 \\
0 & 0 & 0 & 1 & 0 \\
-1 & -1 & 3 & 3 & -2
\end{array}\right), \left(\begin{array}{rrrrr}
-3 & 0 & 4 & 4 & -4 \\
-2 & -17 & 14 & -10 & 16 \\
-2 & -15 & 13 & -8 & 13 \\
0 & 12 & -8 & 9 & -12 \\
0 & 0 & 0 & 0 & 1
\end{array}\right), \left(\begin{array}{rrrrr}
-3 & 0 & 4 & 4 & -4 \\
-4 & -1 & 8 & 8 & -8 \\
-4 & 0 & 7 & 8 & -8 \\
4 & 0 & -8 & -9 & 8 \\
2 & 0 & -4 & -4 & 3
\end{array}\right)$&$\left[+\infty, 6, 4, 2\right]$\\
&$\left(\begin{array}{rrrr}
-12 & -8 & -8 & -4 \\
-8 & -2 & -4 & -2 \\
-8 & -4 & -4 & -2 \\
-4 & -2 & -2 & -2
\end{array}\right)$&$32$&$\left(\begin{array}{rrrr}
1 & 0 & 0 & 0 \\
0 & 1 & 0 & 0 \\
0 & 0 & 1 & 0 \\
0 & 0 & 1 & -1
\end{array}\right), \left(\begin{array}{rrrr}
1 & 0 & 0 & 0 \\
2 & -1 & -2 & 2 \\
1 & 1 & -1 & -1 \\
1 & -1 & -1 & 2
\end{array}\right), \left(\begin{array}{rrrr}
1 & 0 & -4 & 4 \\
-2 & 3 & -2 & 0 \\
-1 & 1 & -1 & 1 \\
-2 & 2 & 0 & -1
\end{array}\right), \left(\begin{array}{rrrr}
1 & 0 & -4 & 4 \\
2 & -1 & -4 & 4 \\
2 & 0 & -5 & 4 \\
2 & 0 & -4 & 3
\end{array}\right)$&$\left[2, 3, 4, 2\right]$\\
\hline 
\end{tabular}}
    \caption{Transcendental lattices for twists with two additive fibres.}
    \label{tab:LocalSystems_2_III_IV}
\end{table}

\begin{table}
    \centering
    \resizebox{\textwidth}{!}{
    \begin{tabular}{||c|l||}
    \hline
    Family&Picard-Fuchs Equations\\
    \hline\hline 
         $\I_2,\textrm{IV},\textrm{III},\textrm{III}$&$\frac{d^6f}{da^6} + \left(\frac{24 a^{2} - 8}{a^{3} - a}\right) \frac{d^5f}{da^5} + \left(\frac{\frac{739}{4} a^{4} - \frac{5333}{36} a^{2} + \frac{437}{36}}{a^{6} - 2 a^{4} + a^{2}}\right) \frac{d^4f}{da^4} + \left(\frac{539 a^{2} - \frac{7847}{36}}{a^{5} - 2 a^{3} + a}\right) \frac{d^3f}{da^3} + \left(\frac{\frac{8715}{16} a^{2} - \frac{10885}{144}}{a^{6} - 2 a^{4} + a^{2}}\right) \frac{d^2f}{da^2} + \left(\frac{\frac{525}{4}}{a^{5} - 2 a^{3} + a}\right) \frac{df}{da} + \frac{\frac{105}{64}}{a^{6} - 2 a^{4} + a^{2}}f=0$\\
         
         $\infty,0,-1,-1$&$\frac{d^5f}{da^5} + \left(\frac{19 a^{2} - 7}{a^{3} - a}\right) \frac{d^4f}{da^4} + \left(\frac{\frac{217}{2} a^{4} - \frac{3335}{36} a^{2} + \frac{80}{9}}{a^{6} - 2 a^{4} + a^{2}}\right) \frac{d^3f}{da^3} + \left(\frac{210 a^{2} - \frac{3445}{36}}{a^{5} - 2 a^{3} + a}\right) \frac{d^2f}{da^2} + \left(\frac{\frac{1785}{16} a^{2} - \frac{925}{48}}{a^{6} - 2 a^{4} + a^{2}}\right) \frac{df}{da} + \frac{\frac{105}{16}}{a^{5} - 2 a^{3} + a}f=0$\\
         
         &$\frac{d^5f}{da^5} + \left(\frac{19 a^{2} - a - 6}{a^{3} - a}\right) \frac{d^4f}{da^4} + \left(\frac{\frac{217}{2} a^{4} - 11 a^{3} - \frac{1519}{18} a^{2} + \frac{11}{2} a + \frac{221}{36}}{a^{6} - 2 a^{4} + a^{2}}\right) \frac{d^3f}{da^3} + \left(\frac{210 a^{3} - 30 a^{2} - \frac{2905}{36} a + 5}{a^{6} - 2 a^{4} + a^{2}}\right) \frac{d^2f}{da^2} + \left(\frac{\frac{1785}{16} a^{2} - \frac{75}{4} a - \frac{175}{12}}{a^{6} - 2 a^{4} + a^{2}}\right) \frac{df}{da} + \frac{\frac{105}{16} a - \frac{15}{16}}{a^{6} - 2 a^{4} + a^{2}}f=0$\\
\hline

    \end{tabular}}
    \caption{Picard-Fuchs equations for twists with three additive fibres.}
    \label{tab:PF_3_additive}
\end{table}

\begin{table}
\ContinuedFloat
    \centering
    \resizebox{\textwidth}{!}{
    \begin{tabular}{||c|l||}
    \hline
    Family&Picard-Fuchs Equations\\
    \hline\hline 
         $\I_4,\textrm{II},\textrm{III},\textrm{III}$&$\frac{d^6f}{da^6} + \left(\frac{24 a^{2} + 80 a + 16}{a^{3} + 5 a^{2} + 2 a + 10}\right) \frac{d^5f}{da^5} + \left(\frac{\frac{739}{4} a^{4} + \frac{2463}{2} a^{3} + \frac{8563}{4} a^{2} + 1488 a + 1300}{a^{6} + 10 a^{5} + 29 a^{4} + 40 a^{3} + 104 a^{2} + 40 a + 100}\right) \frac{d^4f}{da^4} + \left(\frac{539 a^{3} + \frac{10773}{4} a^{2} + \frac{12439}{4} a + 1113}{a^{6} + 10 a^{5} + 29 a^{4} + 40 a^{3} + 104 a^{2} + 40 a + 100}\right) \frac{d^3f}{da^3} +$\\
         &$\left(\frac{\frac{8715}{16} a^{2} + \frac{7245}{4} a + \frac{2065}{2}}{a^{6} + 10 a^{5} + 29 a^{4} + 40 a^{3} + 104 a^{2} + 40 a + 100}\right) \frac{d^2f}{da^2} + \left(\frac{\frac{525}{4} a + \frac{3465}{16}}{a^{6} + 10 a^{5} + 29 a^{4} + 40 a^{3} + 104 a^{2} + 40 a + 100}\right) \frac{df}{da} + \frac{\frac{105}{64}}{a^{6} + 10 a^{5} + 29 a^{4} + 40 a^{3} + 104 a^{2} + 40 a + 100}f=0$\\

         $\infty,0,\omega,-\omega$&$\frac{d^5f}{da^5} + \left(\frac{19 a^{2} + 60 a + 14}{a^{3} + 5 a^{2} + 2 a + 10}\right) \frac{d^4f}{da^4} + \left(\frac{\frac{217}{2} a^{4} + \frac{1373}{2} a^{3} + \frac{4519}{4} a^{2} + 878 a + 687}{a^{6} + 10 a^{5} + 29 a^{4} + 40 a^{3} + 104 a^{2} + 40 a + 100}\right) \frac{d^3f}{da^3} + $\\
         &$\left(\frac{210 a^{3} + \frac{3795}{4} a^{2} + \frac{3905}{4} a + 445}{a^{6} + 10 a^{5} + 29 a^{4} + 40 a^{3} + 104 a^{2} + 40 a + 100}\right) \frac{d^2f}{da^2} + \left(\frac{\frac{1785}{16} a^{2} + \frac{615}{2} a + 135}{a^{6} + 10 a^{5} + 29 a^{4} + 40 a^{3} + 104 a^{2} + 40 a + 100}\right) \frac{df}{da} + \frac{\frac{105}{16} a + \frac{15}{2}}{a^{6} + 10 a^{5} + 29 a^{4} + 40 a^{3} + 104 a^{2} + 40 a + 100}f=0$\\
\hline 
    \end{tabular}}
    \caption{Picard-Fuchs equations for twists with three additive fibres.}
    \label{tab:my_label}
\end{table}

\begin{table}[]
    \centering
    \resizebox{\linewidth}{!}{
    \begin{tabular}{||c|c|c|c|c||}
    \hline
        Family&$Q$ & $\det(Q)$&Monodromy& Orders  \\
        \hline\hline

$\I_2,\textrm{IV},\textrm{III},\textrm{III}$&$\left(\begin{array}{rrrrr}
-4 & 6 & 2 & 0 & 4 \\
6 & -30 & -18 & 6 & -18 \\
2 & -18 & -12 & 4 & -10 \\
0 & 6 & 4 & 2 & 2 \\
4 & -18 & -10 & 2 & -10
\end{array}\right)$&$-96$&$\left(\begin{array}{rrrrr}
1 & 0 & 0 & 0 & 0 \\
0 & 1 & 0 & 0 & 0 \\
0 & 0 & 1 & 0 & 0 \\
0 & 2 & -2 & -1 & -2 \\
1 & 1 & -1 & 1 & 1
\end{array}\right), \left(\begin{array}{rrrrr}
1 & 0 & 0 & 0 & 0 \\
-6 & -1 & 6 & 0 & -6 \\
-4 & -2 & 5 & 1 & -3 \\
0 & 0 & 0 & 1 & 0 \\
-3 & -1 & 3 & 0 & -2
\end{array}\right), \left(\begin{array}{rrrrr}
-1 & 0 & 2 & 0 & -2 \\
-6 & -11 & 6 & 6 & 12 \\
-6 & -10 & 7 & 5 & 9 \\
2 & 4 & -2 & -1 & -4 \\
-4 & -8 & 4 & 4 & 9
\end{array}\right), \left(\begin{array}{rrrrr}
-1 & 0 & 2 & 0 & -2 \\
0 & -1 & 0 & 0 & 0 \\
0 & 0 & 1 & 0 & -2 \\
0 & 0 & -2 & -1 & 2 \\
0 & 0 & 0 & 0 & -1
\end{array}\right)$&$\left[4, 6, 4, 2\right]$\\
&$\left(\begin{array}{rrrrr}
-4 & 2 & -2 & 0 & 0 \\
2 & 2 & 2 & -2 & 2 \\
-2 & 2 & 0 & 0 & 0 \\
0 & -2 & 0 & -2 & 0 \\
0 & 2 & 0 & 0 & 0
\end{array}\right)$&$-32$&$\left(\begin{array}{rrrrr}
1 & 0 & 0 & 0 & 0 \\
0 & 1 & 0 & 0 & 0 \\
0 & 0 & 1 & 0 & 0 \\
-2 & -2 & 2 & 1 & 2 \\
1 & 1 & -1 & -1 & -1
\end{array}\right), \left(\begin{array}{rrrrr}
1 & 0 & 0 & 0 & 0 \\
0 & -1 & -2 & 2 & 2 \\
1 & 0 & -1 & 0 & 1 \\
0 & 0 & 0 & 1 & 0 \\
0 & -1 & -1 & 1 & 2
\end{array}\right), \left(\begin{array}{rrrrr}
1 & 0 & -2 & 2 & 2 \\
-2 & 1 & 2 & -2 & 0 \\
-1 & 0 & 1 & 0 & 1 \\
-2 & 0 & 2 & -1 & 0 \\
0 & 0 & 0 & 0 & 1
\end{array}\right), \left(\begin{array}{rrrrr}
1 & 0 & -2 & 2 & 2 \\
0 & -1 & 0 & 0 & 0 \\
2 & 0 & -3 & 2 & 2 \\
2 & 0 & -2 & 1 & 2 \\
0 & 0 & 0 & 0 & -1
\end{array}\right)$&$\left[4, 3, 4, 2\right]$\\
\hline 
\end{tabular}}
    \caption{Transcendental lattices for twists with three additive fibres.}
    \label{tab:LocalSystems_3_I3}
\end{table}

\begin{table}
\ContinuedFloat
    \centering
    \resizebox{\linewidth}{!}{
    \begin{tabular}{||c|c|c|c|c||}
    \hline
        Family&$Q$ & $\det(Q)$&Monodromy& Orders  \\
        \hline\hline
        
$\I_4,\textrm{IV},\textrm{II},\textrm{II}$&$\left(\begin{array}{rrrrr}
-20 & -8 & -8 & -4 & -16 \\
-8 & -8 & -5 & 2 & -7 \\
-8 & -5 & -4 & -1 & -6 \\
-4 & 2 & -1 & 4 & -5 \\
-16 & -7 & -6 & -5 & -12
\end{array}\right)$&$-144$&$\left(\begin{array}{rrrrr}
1 & 0 & 0 & 0 & 0 \\
0 & 1 & 0 & 0 & 0 \\
0 & 0 & 1 & 0 & 0 \\
4 & 0 & -3 & -1 & -3 \\
-1 & 1 & -1 & 1 & 2
\end{array}\right), \left(\begin{array}{rrrrr}
1 & 0 & 0 & 0 & 0 \\
2 & -1 & 3 & 0 & -3 \\
-1 & -1 & 2 & 1 & 1 \\
0 & 0 & 0 & 1 & 0 \\
0 & 0 & 0 & 0 & 1
\end{array}\right), \left(\begin{array}{rrrrr}
1 & 0 & 4 & 0 & -4 \\
-9 & -2 & 4 & 3 & 8 \\
-6 & -2 & 4 & 2 & 5 \\
9 & 3 & -1 & -2 & -11 \\
-3 & -1 & 4 & 1 & 1
\end{array}\right), \left(\begin{array}{rrrrr}
1 & 0 & 4 & 0 & -4 \\
2 & -1 & 4 & 0 & -4 \\
2 & 0 & 3 & 0 & -4 \\
-2 & 0 & -4 & -1 & 4 \\
2 & 0 & 4 & 0 & -5
\end{array}\right)$&$\left[6, 6, 3, 2\right]$\\
&$\left(\begin{array}{rrrrr}
-20 & -16 & -16 & -12 & -4 \\
-16 & -12 & -13 & -10 & -3 \\
-16 & -13 & -12 & -9 & -4 \\
-12 & -10 & -9 & -8 & -1 \\
-4 & -3 & -4 & -1 & -4
\end{array}\right)$&$-16$&$\left(\begin{array}{rrrrr}
1 & 0 & 0 & 0 & 0 \\
0 & 1 & 0 & 0 & 0 \\
0 & 0 & 1 & 0 & 0 \\
-1 & 0 & 1 & 1 & 1 \\
4 & -1 & -1 & -3 & -2
\end{array}\right), \left(\begin{array}{rrrrr}
1 & 0 & 0 & 0 & 0 \\
1 & -1 & -1 & 2 & 1 \\
2 & 1 & 0 & -3 & -1 \\
0 & 0 & 0 & 1 & 0 \\
0 & 0 & 0 & 0 & 1
\end{array}\right), \left(\begin{array}{rrrrr}
-3 & 0 & -4 & 8 & 4 \\
-3 & 2 & -2 & 3 & 2 \\
-3 & 0 & -2 & 6 & 3 \\
-2 & 1 & -1 & 2 & 1 \\
-1 & -1 & -2 & 5 & 3
\end{array}\right), \left(\begin{array}{rrrrr}
-3 & 0 & -4 & 8 & 4 \\
-2 & -1 & -4 & 8 & 4 \\
-2 & 0 & -5 & 8 & 4 \\
-2 & 0 & -4 & 7 & 4 \\
0 & 0 & 0 & 0 & -1
\end{array}\right)$&$\left[3, 3, 3, 2\right]$\\
\hline 
\end{tabular}}
    \caption{Transcendental lattices for twists with three additive fibres.}
    \label{tab:LocalSystems_3_I4}
\end{table}
\end{document}